\RequirePackage[l2tabu,orthodox]{nag}
\documentclass
[letterpaper,11pt]
{article}

\usepackage{etex}
\usepackage{xspace,enumerate}
\usepackage[dvipsnames]{xcolor}
\usepackage[T1]{fontenc}
\usepackage[full]{textcomp}
\usepackage[american]{babel}
\usepackage{amsfonts}
\usepackage{mathtools}
\usepackage{bm}
\usepackage{amsthm}
\newtheorem{theorem}{Theorem}[section]
\newtheorem*{theorem*}{Theorem}

\newtheorem*{proposition*}{Proposition}
\newtheorem{lemma}[theorem]{Lemma}
\newtheorem*{lemma*}{Lemma}

\newtheorem*{conjecture*}{Conjecture}
\newtheorem{fact}[theorem]{Fact}
\newtheorem*{fact*}{Fact}

\newtheorem*{hypothesis*}{Hypothesis}
\newtheorem{conjecture}[theorem]{Conjecture}
\theoremstyle{definition}
\newtheorem{definition}[theorem]{Definition}
\newtheorem*{definition*}{Definition}

\newtheorem{problem}[theorem]{Problem}

\newtheorem{condition}[theorem]{Condition}
\newtheorem*{condition*}{Condition}
\theoremstyle{remark}
\newtheorem{claim}[theorem]{Claim}
\newtheorem*{claim*}{Claim}
\newtheorem{remark}[theorem]{Remark}
\newtheorem*{remark*}{Remark}

\newtheorem*{observation*}{Observation}

\usepackage[
letterpaper,
top=1in,
bottom=1in,
left=1in,
right=1in]{geometry}

\usepackage{paralist}
\usepackage{turnstile}
\usepackage{mdframed}
\usepackage{algorithm}
\usepackage{algpseudocode}
\usepackage{tikz}

\usepackage{microtype}
\usepackage[
pagebackref,
colorlinks=true,
urlcolor=blue,
linkcolor=blue,
citecolor=OliveGreen,
]{hyperref}

\usepackage[capitalise,nameinlink]{cleveref}
\crefname{lemma}{Lemma}{Lemmas}
\crefname{fact}{Fact}{Facts}
\crefname{theorem}{Theorem}{Theorems}
\crefname{corollary}{Corollary}{Corollaries}
\crefname{claim}{Claim}{Claims}
\crefname{example}{Example}{Examples}
\crefname{problem}{Problem}{Problems}
\crefname{definition}{Definition}{Definitions}
\crefname{conjecture}{Conjecture}{Conjectures}

\newcommand{\proves}{\vdash}

\newcommand{\Authornote}[2]{{\sffamily\small\color{red}{[#1: #2]}}}

\newcommand{\Authorfnote}[2]{\footnote{\color{red}{#1: #2}}}

\newcommand{\Ynote}{\Authornote{Y}}

\newcommand{\PRfnote}{\Authorfnote{PR}}

\usepackage{boxedminipage}
\newcommand{\Paren}[1]{\left(#1\right)}

\newcommand{\brac}[1]{[#1]}
\newcommand{\Brac}[1]{\left[#1\right]}

\newcommand{\abs}[1]{\lvert#1\rvert}

\newcommand{\norm}[1]{\lVert#1\rVert}
\newcommand{\Norm}[1]{\left\lVert#1\right\rVert}

\newcommand{\normt}[1]{\norm{#1}_2}
\newcommand{\Normt}[1]{\Norm{#1}_2}

\newcommand{\snormt}[1]{\norm{#1}^2_2}

\newcommand{\frob}[1]{\norm{#1}_F}
\newcommand{\Frob}[1]{\Norm{#1}_F}

\newcommand{\normo}[1]{\norm{#1}_1}
\newcommand{\Normo}[1]{\Norm{#1}_1}

\newcommand{\iprod}[1]{\langle#1\rangle}
\newcommand{\Iprod}[1]{\left\langle#1\right\rangle}

\newcommand{\Esymb}{\mathbb{E}}
\newcommand{\Psymb}{\mathbb{P}}

\DeclareMathOperator*{\E}{\Esymb}

\DeclareMathOperator*{\ProbOp}{\Psymb}
\renewcommand{\Pr}{\ProbOp}
\newcommand{\tensor}{\otimes}
\newcommand{\mper}{\,.}
\newcommand{\mcom}{\,,}
\newcommand\bdot\bullet
\DeclareMathOperator{\Ind}{\mathbf 1}
\DeclareMathOperator{\Tr}{Tr}

\DeclareMathOperator{\poly}{poly}

\newcommand{\Hoelder}{H\"{o}lder\xspace}
\newcommand{\Holder}{\Hoelder}

\newcommand{\N}{\mathbb N}
\newcommand{\R}{\mathbb R}

\newcommand{\bY}{\mathbf Y}

\newcommand{\cA}{\mathcal A}
\newcommand{\cB}{\mathcal B}

\newcommand{\cD}{\mathcal D}

\newcommand{\cG}{\mathcal G}

\newcommand{\cN}{\mathcal N}
\newcommand{\cO}{\mathcal O}

\newcommand{\bZ}{\mathbf Z}
\newcommand{\bb}{\mathbf b}
\let\epsilon=\varepsilon
\numberwithin{equation}{section}
\newcommand\MYcurrentlabel{xxx}
\newcommand{\MYstore}[2]{%
  \global\expandafter \def \csname MYMEMORY #1 \endcsname{#2}%
}
\newcommand{\MYload}[1]{%
  \csname MYMEMORY #1 \endcsname%
}
\newcommand{\MYnewlabel}[1]{%
  \renewcommand\MYcurrentlabel{#1}%
  \MYoldlabel{#1}%
}
\newcommand{\MYdummylabel}[1]{}
\newcommand{\torestate}[1]{%
  \let\MYoldlabel\label%
  \let\label\MYnewlabel%
  #1%
  \MYstore{\MYcurrentlabel}{#1}%
  \let\label\MYoldlabel%
}
\newcommand{\restatetheorem}[1]{%
  \let\MYoldlabel\label
  \let\label\MYdummylabel
  \begin{theorem*}[Restatement of \cref{#1}]
    \MYload{#1}
  \end{theorem*}
  \let\label\MYoldlabel
}
\newcommand{\restatelemma}[1]{%
  \let\MYoldlabel\label
  \let\label\MYdummylabel
  \begin{lemma*}[Restatement of \cref{#1}]
    \MYload{#1}
  \end{lemma*}
  \let\label\MYoldlabel
}
\newcommand{\restateprop}[1]{%
  \let\MYoldlabel\label
  \let\label\MYdummylabel
  \begin{proposition*}[Restatement of \cref{#1}]
    \MYload{#1}
  \end{proposition*}
  \let\label\MYoldlabel
}
\newcommand{\restatefact}[1]{%
  \let\MYoldlabel\label
  \let\label\MYdummylabel
  \begin{fact*}[Restatement of \prettyref{#1}]
    \MYload{#1}
  \end{fact*}
  \let\label\MYoldlabel
}
\newcommand{\restate}[1]{%
  \let\MYoldlabel\label
  \let\label\MYdummylabel
  \MYload{#1}
  \let\label\MYoldlabel
}
\newcommand{\e}{\epsilon}

\allowdisplaybreaks
\newcommand*{\Id}{\mathrm{Id}}

\newcommand*{\loweredwidetildehelper}[2]{\hbox{\csname dimen@\endcsname\accentfontxheight#1%
  \accentfontxheight#11.25\csname dimen@\endcsname
  $\csname m@th\endcsname#1\widetilde{#2}$%
  \accentfontxheight#1\csname dimen@\endcsname
  }%
}
\newcommand*{\accentfontxheight}[1]{\fontdimen5\ifx#1\displaystyle \textfont \else\ifx#1\textstyle \textfont \else\ifx#1\scriptstyle \scriptfont \else \scriptscriptfont \fi\fi\fi3
}

\DeclareMathOperator{\pEE}{\tilde{\mathbb{E}}}
\newcommand{\pE}{\pEE\nolimits}

\newcommand{\tv}{\tilde{v}}
\newcommand{\tw}{\tilde{w}}
\newcommand{\tSigma}{\tilde{\Sigma}}
\newcommand{\tZ}{\tilde{Z}}

\newcommand{\tE}{\tilde{\E}}

\newcommand{\tO}{\tilde{O}}
\newcommand{\sr}{\text{r}(\Sigma)}

\newcommand{\covrate}{r}

\newcommand{\cdest}{20}
\newcommand{\cgd}{100}
\newcommand{\cest}{200}

\newcommand{\sigalg}{\Sigma^*}
\newcommand{\dalg}{d^*}

\newcommand{\distest}{\text{Distance Estimation}}
\newcommand{\gest}{\text{Gradient Estimation}}

\newcommand{\nit}{1000 d\log \frac{d}{\epsilon}}
\newcommand{\nbuck}{10^6 \log \frac{1}{\delta}}

\title{Algorithms for Heavy-Tailed Statistics: Regression, Covariance Estimation, and Beyond}

\author{%
Yeshwanth Cherapanamjeri\thanks{U.C. Berkeley, \href{mailto:yeshwanth@berkeley.edu}{yeshwanth@berkeley.edu}, supported by the Berkeley Artificial Intelligence Research Lab.}
\and
Samuel B. Hopkins\thanks{U.C. Berkeley, \href{mailto:hopkins@berkeley.edu}{hopkins@berkeley.edu}, supported by a Miller Postdoctoral Fellowship.}
\and
Tarun Kathuria\thanks{U.C. Berkeley,  \href{mailto:tarunkathuria@berkeley.edu}{tarunkathuria@berkeley.edu}, supported by NSF Grant CCF 1718695.}
\and
Prasad Raghavendra\thanks{U.C. Berkeley,
\href{mailto:raghavendra@berkeley.edu}{raghavendra@berkeley.edu}, supported by NSF Grant CCF 1718695.}
\and
Nilesh Tripuraneni\thanks{U.C. Berkeley, \href{mailto:nilesh\_tripuraneni@berkeley.edu}{nilesh\_tripuraneni@berkeley.edu}, supported by the RISELab.}
}

\begin{document}

\pagestyle{empty}

% MAKE TITLE

\maketitle
\thispagestyle{empty} % seems to be required here to avoid page number on first page

% ABSTRACT

\begin{abstract}

    We study polynomial-time algorithms for linear regression and covariance estimation in the absence of strong (Gaussian) assumptions on the underlying distributions of samples, making assumptions instead about only finitely-many moments.
We focus on how many samples are required to perform estimation and regression with high accuracy and exponentially-good success probability in the face of heavy-tailed data.

For covariance estimation, linear regression, and several other problems in high-dimensional statistics, estimators have recently been constructed whose sample complexities and rates of statistical error match what is possible when the underlying distribution is Gaussian, but known algorithms for these estimators require exponential time \cite{mendelson2018robust,lugosi2016risk}.
We narrow the gap between the Gaussian and heavy-tailed settings for polynomial-time estimators with:
\begin{itemize}
    \item a polynomial-time estimator which takes $n$ samples from a $d$-dimensional random vector $X$ with covariance $\Sigma$ and produces $\hat{\Sigma}$ such that in spectral norm $\|\hat{\Sigma} - \Sigma \|_2 \leq \tilde{O}(d^{3/4}/\sqrt{n})$ w.p. $1-2^{-d}$.
    Here the information-theoretically optimal error bound is $\tilde{O}(\sqrt{d/n})$, while previous approaches to polynomial-time algorithms were stuck at $\tilde{O}(d/\sqrt{n})$.
    
    \item a polynomial-time algorithm which takes $n$ samples $(X_i,Y_i)$ where $Y_i = \iprod{u,X_i} + \e_i$ where both $X$ and $\e$ have a constant number of bounded moments and produces $\hat{u}$ such that the loss $\|u - \hat{u}\|^2 \leq O(d/n)$ w.p. $1-2^{-d}$ for any $n \geq d^{3/2} \poly \log(d)$.
    This (information-theoretically optimal) error is achieved by inefficient algorithms for any $n \gg d$, while previous approaches to polynomial-time algorithms suffer loss $\Omega(d^2/n)$ and require $n \gg d^2$. 
\end{itemize}

Our algorithms make crucial use of degree-$8$ sum-of-squares semidefinite programs.
Both apply to any $X$ which has constantly-many \emph{certifiably hypercontractive moments}.
We offer preliminary evidence that improving on these rates of error in polynomial time is not possible in the \emph{median of means} framework our algorithms employ.
Our work introduces new techniques to high-probability estimation, and suggests numerous new algorithmic questions in the following vein: \emph{when is it computationally feasible to do statistics in high dimensions with Gaussian-style errors when data is far from Gaussian}?

\end{abstract}

\clearpage

% TOC

\setcounter{tocdepth}{1}

  % assumes microtype
  \microtypesetup{protrusion=false}
  \tableofcontents{}
  \microtypesetup{protrusion=true}

\clearpage

\pagestyle{plain}
\setcounter{page}{1}

% SECTION

\section{Introduction}\label[section]{introduction}

%\Snote{TODO: Can we get techniques back into 10 pages?}

% TODO: make the introduction more broad. "what is the cost to algorithms of losing the Gaussian assumption?"
% TODO: also add references to classic papers on robust regression and covariance estimation. M-estimators, etc. Emphasize that these are *old* problems.
% TODO: discuss improvements to nearly-linear of SDP-based algos for robust stats

%\emph{What are the best algorithms for statistics in high dimensions with non-Gaussian data?}
%Here ``best algorithms'' has (at least) two competing meanings: we would algorithms with \emph{fast runnning times} which achieve \emph{small rates of statistical error}.

Much work in theoretical computer science on algorithms for high-dimensional learning and statistics focuses on the dependence of rates of error (in estimation, regression, PAC learning, etc.) on the number of samples $n$ given to a learning/regression/estimation algorithm and the dimension/number of features $d$ of those samples.
\emph{In statistics it is also of fundamental importance to understand the dependence on the \textbf{level of confidence} $1-\delta$ -- predictions and estimates made from samples are most useful if they come with small confidence intervals.}
Classical estimators for elementary estimation and regression problems often have error rates $r(n,d,\delta)$ with far-from-optimal dependence on $\delta$ unless strong assumptions are made on the underlying distribution of samples.
In this work, we study algorithms for high-dimensional statistics without strong (sub-Gaussian) assumptions, focusing on achieving \emph{small errors  with high probability in polynomial time}.

Consider a prototypical estimation problem: the goal is to take independent samples $X_1,\ldots,X_n \sim \sim p_\theta$, where $p_\theta$ is a member of a family of $d$-dimensional probability distributions indexed by parameters $\theta$ and find $\hat{\theta}$ such that $\|\theta - \hat{\theta}\| \leq r$ with probability $1-\delta$ for some norm $\|\cdot \|$ and some \emph{rate} $r(n,d,\delta)$.
%Classical approaches (empirical averages, ordinary least squares regression, etc) often yield estimators $\hat{\theta}$ with information-theoretically optimal $r(n,d,\delta)$ \emph{under strong assumptions, such as sub-Gaussian-ness, on the distributions $p_\theta$}.
If we make only a weak assumption on $p_\theta$ -- e.g. that it has a small number of finite moments -- then the rates $r(n,d,\delta)$ achieved by classical approaches are typically exponentially-far from optimal with respect to $\delta$ (i.e. $r(n,d,\delta)$ scales like $1/\poly(\delta)$ rather than $\log(1/\delta)$).

Since at least the 1980s it has been known that in low-dimensional settings (e.g. $d=1$) there are estimators for basic problems like estimating the mean which achieve rates $r(n,\delta)$ whose dependence on $\delta$ under such weak assumptions is comparable to that of classical estimators (the empirical mean) under (sub)-Gaussian assumptions (up to constants).
For instance, the \emph{median of means} estimator of the mean achieves the same $r(n,\delta)$ as the empirical mean does in the Gaussian setting but assuming only that $p_\theta$ has finite variance \cite{alon1999space,jerrum1986random,nemirovsky1983problem}.
This immediately proved useful in streaming algorithms \cite{alon1999space}.

Achieving similar guarantees for large dimensions $d$ is much more challenging, even without asking for computationally-efficient algorithms.
A series of exciting developments in the last decade in statistics, however, constructs estimators with $r(n,d,\delta)$ matching the rates achievable in the Gaussian case by classical approaches but with much weaker assumptions.
Such estimators are now known for high dimensional mean estimation, covariance estimation, (sparse) linear regression, and more \cite{lugosi2019mean}.
Unlike their one-dimensional counterparts and classical approaches, however, \emph{naive algorithms to compute this new generation of optimal estimators take time exponential in $n,d$, or both}.
This suggests a key question applying to a wide range of estimation, regression, and learning problems:
\begin{quote}
\emph{Are there efficiently computable estimators achieving optimal $r(n,d,\delta)$ under weak assumptions (like finitely-many bounded moments) on underlying data?}
\end{quote}
Recent work in algorithms shows that such optimal and computationally efficient estimators do exist for the problem of estimating the mean of a random vector $X$ under only the assumption that $X$ has finite covariance \cite{hopkins2018sub,cherapanamjeri2019fast}.
The resulting algorithms, however, are heavily tailored to estimating the mean in $\ell_2$; although they introduce useful techniques, it is unclear whether they suggest any broader answers to the above.

In this work we tackle covariance estimation and linear regression with these goals in mind.
We contribute new algorithms for both problems whose error rates $r(n,d,\delta)$ improve by $\poly(d,\log(1/\delta))$ factors on the previous best polynomial-time algorithms when the underlying data is drawn from a distribution with only finitely-many bounded moments.
Unlike the situation in mean estimation, however, our estimators do not achieve information-theoretically optimal error rates.
We offer evidence (by constructing certain moment-matching distributions) that no efficient algorithm using the median-of-means approach we use here can significantly improve on rates achieved by our algorithms.
This suggests the possibility that the computational landscape for covariance estimation and regression is more complicated than for mean estimation: in particular, it could be that these problems suffer from a novel kind of tradeoff between computational efficiency and error rate \emph{in the small $\delta$ regime}.
(By contrast in the regime $\delta = \Omega(1)$ classical estimators typically have $r(n,d,\delta)$ which is information-theoretically optimal with respect to $n,d$ and are also efficiently computable.)
Whether there is indeed such a tradeoff is a fascinating open question.
%information-computation gap \emph{which appears only in the high-probability (small $\delta$) regime}.
%Giving stronger evidence for such an information-computation gap is a fascinating open problem.
%\PRfnote{can we say: "these problems suffer from an inherent tradeoff between computational efficiency and error rate."  Whether there is indeed a tradeoff between computational efficiency and error rates is a fascinating open question}
%
\paragraph{Why Weak Assumptions?}
We study polynomial-time algorithms for high-dimensional statistics under \emph{weak assumptions} on underlying data.
Both linear regression and covariance estimation boast well-studied and computationally-efficient algorithms which achieve statistically optimal rates $r(n,d,\delta)$ with respect to both $n$ and $\delta$ under (sub)-Gaussian assumptions on $X$ (and $\e$): ordinary least squares regression and the empirical covariance, respectively.
These estimators are among the oldest in statistics: Gauss and Legendre both studied the least-squares estimator for linear regression around 1800 \cite{wiki:least_squares} and study of the empirical covariance dates at least to Pearson's invention of principal component analysis \cite{pearson1901liii}.

However, data cannot assumed to be Gaussian in every situation.
In this paper we only assume boundedness conditions on a small number of moments of a random vector $X$ (generally $8$th moments).
Under such assumptions, the error rates of the empirical covariance and ordinary least squares grow polynomially in $1/\delta$, while optimal error rates are logarithmic in $1/\delta$.
Beyond allowing us to address basic questions about which error rates are achievable in polynomial time, working under weak assumptions makes our algorithms potentially useful in a variety of settings where classical estimators break down.

First, our algorithms are useful in statistical settings involving heavy-tailed data -- data drawn from distributions with only a finite number of bounded moments.
Large networks, for instance, are well known to generate heavy-tailed data, often following a power law distribution.
Other common heavy-tailed distributions in statistics include the \emph{Student's $t$} distribution, and the Log-Normal distribution -- the  latter describes a number of real-world phenomena, such as the distribution of English sentence lengths, the distribution of elements in the Earth's crust, the distribution of species' abundances, and more \cite{wiki:lognormal}.
Even when data are not known to follow a particular heavy-tailed distribution, the conservative statistician may wish to avoid a Gaussian assumption if also lacking good reason to believe that the underlying population is Gaussian-distributed.

Second, it is often convenient to use algorithmic primitives for basic tasks like covariance estimation and regression as parts of more sophisticated algorithms.
Algorithms for the complicated high-dimensional statistics problems often studied in theoretical machine learning can have many moving parts.
In such situations, the samples $X_1,\ldots,X_n$ may themselves be the output of a complex random process or another ``upstream'' algorithm.
This can make it it difficult or impossible to guarantee that $X_1,\ldots,X_n$ satisfy sub-Gaussian concentration properties, but it can be much easier to establish that the outputs of such upstream algorithms satisfy the kind of weak finite-moment bounds required by our algorithms.
Indeed, one of the first uses of the \emph{median of means} technique we employ here (for estimation of frequency moments in a streaming setting) was for exactly this purpose \cite{DBLP:journals/jcss/AlonMS99}.

\subsection{Results}
\label[section]{sec:results}

\paragraph{``Nice'' distributions}
Since the main goal of our work is to achieve Gaussian-style error bounds while avoiding Gaussian assumptions in high-dimensional parameter estimation, before we lay out our results, we must describe the class of distributions to which they apply.
Obtaining Gaussian-style error rates does require some assumptions on the underlying random variables, for information-theoretic reasons -- typically the existence of $2$nd moments is a minimal requirement \cite{catoni2012challenging}.
(For covariance estimation this becomes $4$th moments of a random vector $X$, which are the $2$nd moments of the random matrix $XX^\top$.)

In this paper we make an assumption called \emph{certifiable hypercontractivity:}
we assume that \emph{as a polynomial in variables $u = (u_1,\ldots,u_d)$},
\[
O(1) \cdot (\E \iprod{X,u}^2)^4 - \E \iprod{X,u}^8\mper
\]
is a sum of squares.\footnote{Our algorithms also work if instead the inequality $\E \iprod{X,u}^8 \leq (\E \iprod{X,u}^2)^4$ has an SoS proof of higher degree, at a commensurate cost in running time to allow for higher-degree SoS relaxations.}
This in particular implies the more standard $8$th moment bound $\E \iprod{X,u}^2 \leq O(\E \iprod{X,u}^8)^{1/4}$.
We often call $(2,8)$ certifiably-hypercontractive distributions \emph{nice}.
We emphasize that niceness is an ``infinite-sample'' assumption: it concerns population moments $\E X^{\tensor 8}$.

Certifiable hypercontractivity holds for numerous interesting heavy-tailed distributions for which previous polynomial-time algorithms could not have achieved Gaussian-style error guarantees.
For instance, any product of univariate distributions with bounded $8$-th moments, and any linear transformation thereof (in particular for multivariate $t$-distributions) is certifiably hypercontractive.
In fact, the certifiable hypercontractivity assumption has been shown to hold for any distribution whose $8$-th moments match those of some strongly log-concave distribution \cite{kothari2018robust} (even if, say, $9$th moments do not exist).
The certifiable hypercontractivity assumption also underlies recent results in on polynomial-time high-dimensional clustering of mixture models and several robust parameter estimation problems \cite{kothari2018robust,hopkins2018mixture,klivans2018efficient}.
%\footnote{\emph{Robust} parameter estimation is the problem of learning a distribution from a set of samples for which a small fraction have been \emph{adversarially} corrupted. While recent literature on robust parameter estimation also leverages the SoS approach to high-dimensional parameter estimation, that setting is incomparable to the ones we consider here.} 

% \Snote{state our assumptions and the sos-boundedness assumption. make clear this is an assumption on the population and not on the samples. previous work has no trouble getting from samples to population (hopkins-li, kothari-steinhardt-steurer) because they can draw poly(d) extra samples. we are focused on the sample complexity so we have to use power of high-degree sos to get from samples to population.}
% \Snote{note that sos boundedness holds for products, linear transformations of products, any distribution with 4th moments matching those of a strongly-log-concave distribution, multivariate $t$}

\subsubsection{Covariance Estimation}
\label[section]{sec:results-cov}

Covariance estimation is the following simple problem.
Given samples $X_1,\ldots,X_n$ from a $d$-dimensional random vector with covariance $\Sigma$, find $\hat{\Sigma}$ with the smallest possible spectral norm error $\|\hat{\Sigma} - \Sigma\|_2$.
For simplicity, let us focus for now on the setting that $\Tr \Sigma \leq O(d)$ and $\|\Sigma\|_2 \leq O(1)$ and $\delta = 2^{-d}$.
(Our main theorem for covariance estimation handles the case of general $\Sigma$ and $\delta$.)
We also assume throughout that $\E X = 0$; otherwise one may replace $X$ with $(X - X')/\sqrt{2}$ for pairs of independent samples $X,X'$ without affecting the covariance and losing only a factor of $2$ in the sample complexity.

Consider the Gaussian setting $X \sim \cN(0,\Sigma)$.
In this case, classical results offer the following type of concentration bound for the empirical covariance $\overline{\Sigma} = \tfrac 1  n \sum_{i \leq n} X_i X_i^\top$ of $n$ independent samples: for a universal constant $C$,
\begin{align}
  \Pr \Paren{ \Norm{\overline{\Sigma} - \Sigma}_2 \geq C\Paren{\sqrt{\frac dn } + t}} \leq \exp(-t^2 n)\mper \label{eq:gauss-cov}
\end{align}
(This bound becomes meaningful only when $n \geq d$.)
Note that by \cref{eq:gauss-cov}, $\Normt{ \overline{\Sigma} - \Sigma} \leq O(\sqrt{d/n})$ with probability $1-2^{-d}$.

Recent work by Mendelson and Zhivotovskiy \cite{mendelson2018robust}, building on earlier works by Lugosi and Mendelson \cite{lugosi2018near} shows that there is an estimator $\hat{\Sigma}$ for the covariance $\Sigma$ which matches this error guarantee under only the assumption that $X$ has hypercontractive $4$-th moments.
(In all the following informal theorem statements we assume $\Tr \Sigma \leq O(d),\|\Sigma\|_2 \leq O(1)$.)

\begin{theorem}[\cite{mendelson2018robust}]\label{thm:mendelson-intro}
  There is an estimator $\hat{\Sigma} = \hat{\Sigma}(X_1,\ldots,X_n)$ which given $n$ independent samples from a random variable $X$ with covariance $\Sigma$ and which is $(2,4)$-hypercontractive has the guarantee
  \[
  \Norm{\hat{\Sigma} - \Sigma}_2 \leq O\Paren{\sqrt{\frac {d \log d}{n}}} \text{ with probability at least } 1 - 2^{-d}\mper
  \]
\end{theorem}
\noindent

Up to logarithmic factors, this rate of error is information-theoretically optimal, but no algorithm is known which achieves this guarantee in polynomial time.
Prior to this work, the strongest result known for polynomial-time algorithms was weaker by a $\poly(d)$ factor:
\begin{theorem}[\cite{minsker2018robust}]\label[theorem]{thm:minsker-intro}
  Under the same hypotheses as \cref{thm:mendelson-intro} there is a polynomial-time algorithm which finds $\hat{\Sigma}$ such that $\|\hat{\Sigma} - \Sigma\|_2 \leq O(d/\sqrt{n})$ with probability at least $1-2^{-d}$.
\end{theorem}

Our main result for covariance estimation in the setting $\Tr \Sigma \approx d, \|\Sigma \| \approx 1$ is the following.
(For other parameter regimes see \cref{thm:covest}.)

\begin{theorem} [Main theorem on covariance estimation, informal -- see \cref{thm:covest}]
\label[theorem]{thm:covest-inform}
    There is an algorithm with running time $\poly(n,d)$ which when given $n$ i.i.d. samples $X_1,\ldots,X_n$ from a nice random vector $X$ in $d$ dimensions returns an estimate $\hat{\Sigma}$ of the covariance $\Sigma$ of $X$ such that
    \[
      \Norm{\hat{\Sigma} - \Sigma}_2 \leq \tO \left( \frac {d^{3/4}} {\sqrt n} \right) \text{ with probability at least } 1 - 2^{-d}\mper
    \]
    Here $\tO(\cdot)$ hides logarithmic factors in the dimension $d$.
%    Let $\{ \cD \}$ be the class of "nice" distributions (in the sense of \cref{def:nice-dist}) with covariance $\Sigma$ satisfying $\normt{\Sigma}=O(1), \Tr(\Sigma)=O(d), L=O(1)$. Then there exists an algorithm with complexity $\poly(n, d)$, when given access to $n$ i.i.d. samples $v_i \sim \cD \in \{\cD \}$, that provides an estimate of $\cD$'s covariance, $\hat{\Sigma}$, satisfying,
%    \begin{align*}
%        \normt{\hat{\Sigma}-\Sigma} \leq \tO \left( \frac{1}{\sqrt{n}}  \left( \frac{d^{5/4}}{n^{1/4}} +  d^{3/4} \right) \right)
%    \end{align*}
%    with probability at least $1-2^{-\Omega(d)}$ \mper
\end{theorem}
%\PRfnote{is the terminology "nice" or "$O(1)$-nice"}
%\noindent We note that our algorithm achieves spectral norm error $\tO({d^{3/4}}/{\sqrt n})$; up to logarithmic factors this is $d^{1/4}$ smaller than in \cref{thm:minsker-intro}.
%In fact, for any $\e > 0$, if $n \gg d^{1+\eps}$ we achieve error smaller by a factor $d^{\Omega(\eps)}$ than previous polynomial-time algorithms.
%This is the main interesting regime for covariance estimation, since to obtain any nontrivial estimate of a rank-$d$ covariance matrix $\Sigma$ requires $n \gg d$ samples.

The general statement of our main theorem (\cref{thm:covest}) obtains an error rate which avoids explicit dependence on the ambient dimension $d$ (except for logarithmic factors); instead, it depends only on the "effective" rank $\sr = \frac{\Tr(\Sigma)}{\normt{\Sigma}} \leq d$ and the operator norm $\normt{\Sigma}$.
Thus if $X$ lies in or near a low-dimensional subspace, our algorithm exploits this additional structure to estimate $\Sigma$ with fewer samples.

Finally, we note that our algorithm assumes access to a small number of additional parameters: bounds on $\Tr \Sigma, \Normt{\Sigma}$, and (as with all the algorithms described in this paper beyond empirical averages) in the case of general confidence levels $1-\delta$ it depends on the value of $\delta$.
The latter dependence is intrinsic: it is not information-theoretically possible to obtain Gaussian-style error rates in the heavy-tailed setting with estimators which do not depend on $\delta$ \cite{catoni2012challenging}.
We expect that techniques similar to those of \cite{mendelson2018robust} can avoid the dependence on $\Tr \Sigma, \Normt{\Sigma}$ by estimating them from samples.

The improvement from $d/\sqrt{n}$ to $d^{3/4}/\sqrt{n}$ moves the algorithmic state of the art for covariance estimation closer to information-theoretic optimality.
Of course the possibility of an information-theoretically optimal covariance estimation algorithm is tantalizing, but just as interesting from a complexity viewpoint is the possibility that $d^{3/4}/\sqrt{n}$ cannot be improved upon in polynomial time.
In \cref{sec:results-roadblock} we discuss evidence in this direction.

\subsubsection{Linear Regression}
\label[section]{sec:results-reg}

%\Snote{Update technical section for linear regression -- degree 8 SoS, $d^(3/4)$ bound, update main theorem, update abstract, update introduction}

We study the following classical linear regression problem.
Let $f^* \, : \, \R^d \rightarrow \R$ be a linear function -- that is $f^*(x) = \iprod{f^*,x}$ for some vector $f^* \in \R^d$.
Let $X$ be a $d$-dimensional mean-zero random vector, and let $\e$ be an $\R$-valued random variable with $\E \e = 0$.
To avoid a preponderance of parameters, in this paper we focus on the case that $\E XX^\top = \Id$ and $\E \e^2 = 1$.\footnote{It is trivial to show that our results also work if $\E \e^2 = \sigma^2$, with appropriate dependence of the error rates on $\sigma$. We also believe that our techniques will be useful in designing algorithms which achieve small error $\E (\hat{f}(X) - f^*(X))^2$ when $\E XX^\top = \Sigma$ for general $\Sigma$, but we defer this challenge to future work. If $X$ is not mean zero then it can be replaced with $X - X'$ for pairs of samples $X,X'$, so this assumption is without loss of generality.}

The goal is to take $n$ independent samples of the form $(X_i,Y_i)$, where $Y_i = f^*(X_i) + \e_i$, and find a linear function $\hat{f}$ such that $\|f^* - \hat{f}\|$ is as small as possible.
Here the norm $\|f^* - \hat{f}\|$ is the $2$-norm induced by $X$; that is, $( \E (f^*(X) - \hat{f}(X))^2)^{1/2}$.
However, since we assume $\E XX^\top = \Id$, this is identical to the Euclidean norm of $f^* - \hat{f}$ considered as a vector of coefficients.

In most respects the situation for linear regression is similar to that for covariance estimation.
The classical algorithm is empirical risk minimization, also known in this setting as ordinary least squares regression (OLS).
The algorithm is simple: given $(X_1,Y_1),\ldots,(X_n,Y_n)$, output $\hat{f}$ which minimizes the empirical loss $\E_{i \sim [n]} (f(X_i) - Y_i)^2$.
This minimization problem is convex, so $\hat{f}$ can be obtained in polynomial time; it also admits a closed-form linear-algebraic solution.

Analogously to the empirical covariance in the previous section, when $X$ and $\e$ are Gaussian, OLS achieves small error with high probability.
Concretely, one has the following:\footnote{It is traditional here to state bounds on $\|\hat{f} - f\|^2$ rather than $\|\hat{f} - f\|$; note that the bound $O(d/n)$ represents the so-called \emph{fast rate} for regression -- in this paper we are exclusively concerned with fast rates, rather than the \emph{slow rate} $O(\sqrt{d/n})$.)}
\[
  \|\hat{f}_{\text{OLS}}  - f^*\|^2 \leq O \Paren{\frac d n } \text{ with probability } 1 - 2^{-d} \text{ so long as } n \gg d \mper
\]
We focus for now on the setting of regression with confidence $1-2^{-d}$: this regime provides a useful litmus test because it is the highest probability for which the $O(d/n)$ guarantee holds for OLS.
When $X$ or $\e$ has only a finite number of bounded moments, the error bound on $\|\hat{f}_{\text{OLS}} - f^*\|$ degrades badly, becoming exponential in $d$ for confidence $1-2^{-d}$.

Recent work by Lugosi and Mendelson \cite{lugosi2016risk} shows that a guarantee matching that of OLS in the Gaussian setting is possible without Gaussian assumptions.
Concretely we have the following:
\begin{theorem}[\cite{lugosi2016risk}, informal]
  \label[theorem]{thm:lm-reg-intro}
  There exists an (exponential-time) estimator $\hat{f}$ which given $n$ independent samples $(X_1,Y_1),\ldots,(X_n,Y_n)$ where $Y = f^*(X) + \e$, $\E XX^\top = \Id$, $X$ is $(2,4)$-hypercontractive, and $\E \e^2 = 1$, has\footnote{The results of \cite{lugosi2016risk} apply to a wide variety of convex function classes rather than just linear regression; we state here the special case for linear regression.}
  \[
    \Norm{\hat{f} - f}^2 \leq O \Paren{\frac d n} \text{ with probability } 1 - 2^{-d} \text{ so long as } n \gg d\mper
  \]
\end{theorem}

Once again, the state of the art for polynomial-time algorithms is somewhat worse (though still far better than OLS).
Until this paper, the polynomial-time algorithm with smallest error guarantees in the $1-2^{-d}$ probability regime were achieved by an algorithm of \cite{hsu2016loss}.

\begin{theorem}[\cite{hsu2016loss}, informal]
  \label[theorem]{thm:hs-intro}
  There is a polynomial-time algorithm which computes an estimator $\hat{f}$ which given $n$ i.i.d. samples $(X_i,Y_i)$ where $X$ is $(2,4+\delta)$-hypercontractive for some $\delta > 0$ and $Y = f^*(X) + \e$ for some linear function $f^*$ for a random variable $\e$ with $\E \e = 0$ and $\E \e^2 = 1$ achieves
  \[
    \Norm{\hat{f} - f}^2 \leq O\Paren{ \frac{d^2}{n} } \text{ with probability } 1 - 2^{-d} \text{ so long as } n \gg d^2\mper
  \]
\end{theorem}

Note that the error guarantees of \cref{thm:hs-intro} are weaker than what is information-theoretically possible (\cref{thm:lm-reg-intro}) in two key ways: first of all, the error scales with $d^2$ rather than with $d$, and second, the error rate does not kick in until $n \gg d^2$.
Our main theorem on regression completely fixes the first problem and partially fixes on the second (but does not reach information-theoretic optimality), for nice $X$.

\begin{theorem}[Main theorem on linear regression, informal -- see \cref{thm:regression-main}]
\label[theorem]{thm:regression-intro}
  There is an algorithm with running time $\poly(n,d)$ with the following guarantees.
  Suppose $X$ is nice, $\e$ is a univariate random variable with $\E \e^2 = 1$ and $\E \e = 0$, and $f^*$ is a linear function.
  Given $n$ i.i.d. samples $(X_i,Y_i)$ of the form $Y_i = f^*(X_i) + \e_i$, the algorithm finds a linear function $\hat{f}$ such that
  \[
    \Norm{\hat{f} - f^*}^2 \leq O\Paren{\frac d n} \text{ with probability } 1 - 2^{-d} \text{ so long as } n \gg d^{3/2} \cdot (\log d)^{O(1)}\mper
  \]
\end{theorem}

Our main result (and all the prior work) gracefully tolerates confidence levels other than $1-2^{-d}$; see \cref{thm:regression-main} for details.

% \Snote{also here do comparison to Hsu-Sabato and to bucket-wise ERM}

\subsubsection{Faster algorithms for mean estimation in general norms}
\label[section]{sec:results-mean}

Our final algorithmic result concerns the problem of estimating the mean of a random vector $X$ on $\R^d$ with respect to an arbitrary norm $\| \cdot \|$.
%we consider the question: \emph{what is the right baseline brute-force running time for estimators which achieve Gaussian-style performance under finite-moment assumptions?}
Our starting point is the following theorem of Lugosi and Mendelson which constructs an estimator of the mean with respect to any norm $\| \cdot \|$ on $\R^d$.
In such a general setting the question of information-theoretic optimality is somewhat murky.
Nonetheless, for many natural norms ($\ell_2$ and spectral norm, for instance) one may see that the guarantees of their estimator match those of the empirical mean in the Gaussian setting.
We refer the reader to \cite{lugosi2018near} for further interpretation of the guarantees of the following theorem.

\begin{theorem}[\cite{lugosi2018near}, informal, $\Id$-covariance case]
  \label[theorem]{thm:lm-general-intro}
  For every $n,d \in \N$ and $\delta > 2^{-n}$ and norm $\| \cdot \|$ on $\R^d$ there is an estimator with the following guarantee.
  Given $n$ i.i.d. samples $X_1,\ldots,X_n$ of a random vector $X$ with mean $\mu$ and covariance $\Id$, it finds $\hat{\mu}$ such that
  \[
  \Norm{\mu - \hat{\mu}} \leq \frac 1 {\sqrt n} \cdot O\Paren{ \E \Norm{\sum_{i \leq n} \sigma_i X_i } + R \sqrt{\log(1/\delta)} }  \text{ with probability at least } 1 - \delta
  \]
  where $\sigma_1,\ldots,\sigma_n \sim \{ \pm 1\}$ are independent signs and $R = \sup_{\|x\|_* = 1} \|x\|_2$ is the norm-equivalence constant between the dual norm $\|\cdot\|_*$ and $\ell_2$.
  Note that the first term is essentially the \emph{expected} error achieved by the empirical mean for the norm $\| \cdot \|$, and in particular is independent of $\delta$, while the second term determines the decay of the bound as $\delta$ becomes small.\footnote{In \cite{lugosi2018near} this theorem is stated with an extra term in the error guarantee (which is typically dominated by the first term); we provide a simplified proof which also shows that the additional term is unnecessary.}
\end{theorem}

The naive algorithm to compute the estimator $\hat{\mu}$ from \cref{thm:lm-general-intro} requires brute-force search for a point in a non-convex set in $d$ dimensions, taking $\exp(\Omega(d))$ time.
We slightly modify the estimator from \cref{thm:lm-general-intro} and show that subject to a mild computational assumption on the norm $\| \cdot \|$ it can be computed by an algorithm whose running time is exponential only in $\log(1/\delta)$ rather than in $d$.

\begin{theorem}[Informal, $\Id$-covariance case, see \cref{thm:general-norms}]
  With the same setting and guarantees as \cref{thm:lm-general-intro}, under the additional assumption that there is a polynomial-time separation oracle for the dual ball of $\|\cdot \|$, there is an algorithm to compute $\hat{\mu}$ in time $\poly(n,d,1/\delta)$.
\end{theorem}

% \Snote{the ``right'' brute-force running time}

\subsubsection{Roadblock to Improved Error Rates: Single-Spike Block Mixtures}
\label[section]{sec:results-roadblock}

Our main results on covariance estimation and linear regression (\cref{thm:covest-inform,thm:regression-intro}) push the state of the art in terms of error rates achievable for heavy-tailed statistics in polynomial time, but they do not achieve information-theoretic optimality.
Our covariance estimation algorithm in the setting of $\Tr \Sigma \leq O(d), \Normt{\Sigma} \leq O(1)$ achieves error $\|\hat{\Sigma} - \Sigma\|_2 \leq \tO(d^{3/4}/\sqrt{n})$, while in exponential time it is possible to achieve $\tO(\sqrt{d/n})$.
(Similarly, our linear regression algorithm requires $n \gg d^{3/2}$ rather than $n \gg d$.)

It is a fascinating open problem to understand whether these gaps can be closed.
We offer here some evidence that this is unlikely to be possible with techniques in the present paper.
We focus on covariance estimation -- the relation to linear regression is more subtle (see \cref{sec:roadblock}).
The key subroutine in our covariance estimation algorithm is an algorithm for the following problem:

\begin{problem}[Find high-variance direction]
\label[problem]{prob:high-var}
Given $\Sigma_1,\ldots,\Sigma_d \in \R^{d \times d}$, with $\Sigma_i \succeq 0$, find a unit vector $x \in \R^d$ such that $\iprod{x,\Sigma_i x} \geq r$ for at least $d/4$ matrices $\Sigma_i$, or certify that none exists.
\end{problem}

In fact, \cref{prob:high-var} must be solved when $\Sigma_1,\ldots,\Sigma_d$ are empirical covariance matrices by any algorithm performing covariance estimation using the \emph{median-of-means} framework, which is the dominant approach in constructing high-dimensional estimators with optimal $r(n,d,\delta)$ (even ignoring running time considerations).
It will have to wait until \cref{sec:techniques} to see in more detail why an algorithm solving \cref{prob:high-var} is useful for covariance estimation.
For now, let us note that our subroutine solves \cref{prob:high-var} when $\Sigma_i$ is the empirical covariance of $n/d$ samples from the heavy-tailed distribution whose covariance we are estimating, and the $\Sigma_i$'s are all independent.
\cref{prob:high-var} gets easier as $r$ gets larger, but it turns out that the value of $r$ for which we can solve it translates directly to the error rate of our covariance estimation algorithm.
\emph{Summarizing: in the case of estimating the covariance $\Sigma$ of a random variable $X$ with $\Tr \Sigma \approx d, \|\Sigma\|_2 \approx 1$, our key subroutine solves \cref{prob:high-var} with $\Sigma_i$ being the empirical covariance of $n/d$ of the samples $X_1,\ldots,X_n$ and $r \leq \tO(d^{3/4}/\sqrt{n}$).}

Improving the error rates of our algorithm (or any other median-of-means-based algorithm) would thus seem to require solving \cref{prob:high-var} with smaller $r$.
To investigate whether this may be possible in polynomial time, we consider an easier variant, which we call the \emph{single-spike block mixtures} problem.
It is easier in two respects: it is a decision problem rather than a search problem, and the underlying random variable $X$ is distributed in a known, Gaussian fashion.
(Note that it appears no longer relevant that we were initially interested in heavy-tailed random vectors -- we believe computational hardness for \cref{prob:high-var} appears even when $\Sigma_i$'s are empirical covariances formed from Gaussian samples.)

\begin{definition}[Single-Spike Block Mixtures]
  Let $d,m \in \N$ and $1 > \lambda > 0$.
  In the \emph{single-spike block mixtures testing problem} the goal is to distinguish, given vectors $y_1,\ldots,y_{md} \in \R^d$, between the following two cases:
  \begin{itemize}
  \item[\textsc{null}:] $y_1,\ldots,y_{md} \sim \cN(0,\Id)$ i.i.d.
  \item[\textsc{planted}:]
  First $x \sim \{ \pm 1/\sqrt{d} \}^d$ and $s_1,\ldots,s_d \sim \{\pm 1\}$.
  Then, $y_1,\ldots,y_m \sim \cN(0,\Id + s_1 \lambda xx^\top)$ and $y_{m+1},\ldots,y_{2m} \sim \cN(0, + s_2 \lambda xx^\top)$, and so forth.
  That is, each \emph{block} of vectors $y_{im},\ldots,y_{(i+1)m - 1}$ has either slightly larger variance in the $x$ direction (if $s_i = 1$) or slightly lesser variance (if $s_i = -1$) than they would in the null case.
  \end{itemize}
\end{definition}

\noindent It turns out that so long as $\lambda \gg 1/\sqrt{m} = \sqrt{d/n}$ (where $n = md$) it is possible to distinguish \textsc{null} from \textsc{planted} in exponential time.
(This is closely related to the fact that heavy-tailed mean estimation can be solved with error rate $\tO(\sqrt{d/n})$.)
But what about polynomial time?
A consequence of our main subroutine is the following theorem:

\begin{theorem}[Informal]
  If $\lambda \geq (d^{3/4}/\sqrt{n}) \poly \log(d,m)$ then there is a polynomial-time algorithm which distinguishes \textsc{null} from \textsc{planted} with high probability.
\end{theorem}

\noindent We make the following conjecture regarding optimality of this algorithm.

\begin{conjecture}
\label[conjecture]{conj:opt}
If $\lambda \leq d^{3/4 - \Omega(1)}/\sqrt{n}$ then no polynomial time algorithm solves the single-spike block mixture problem.
\end{conjecture}

In support of \cref{conj:opt}, we prove a lower bound against a certain class of restricted algorithms, called \emph{low degree tests}.
A degree-$D$ test is a function $f \, : \, \R^{d \times md} \rightarrow \R$ such that as a polynomial $\deg f \leq D$ and $\E_{Y = y_1,\ldots,y_{md} \sim \textsc{null}} f(Y) = 0$.
We say the test is \emph{successful} if $\E_{Y \sim \textsc{planted}} f(Y) / (\E_{Y \sim \textsc{null}} f(Y)^2)^{1/2} \rightarrow \infty$ as $d,m \rightarrow \infty$.

While such low degree tests (for $D$ relatively small -- say at most $(md)^{o(1)}$) would seem to be a quite restrictive model compared to the class of all polynomial time algorithms, it turns out that the existence of a successful low degree test solving a hypothesis testing problem is a remarkably accurate predictor for the existence of any polynomial time algorithm.
For instance, successful low degree tests (of logarithmic degree) appear exactly at the predicted \emph{computational thresholds} for the planted clique problem (clique size $\Omega(\sqrt{n})$), the random $3$-SAT problem ($(\text{number of variables})^{3/2})$ clauses), the $k$-community stochastic block model (the \emph{Kesten-Stigum threshold}), the sparse PCA problem (the $k^2$ sample threshold) and beyond.
Lower bounds on low degree tests are technically distinct from but conceptually similar to statistical query lower bounds.
They are also closely related to the \emph{pseudocalibration} technique for proving lower bounds against SoS algorithms.
For further discussion, see \cite{hopkins2018statistical,kunisky2019notes}.

We rule out the existence of successful low degree tests for $D = (md)^{o(1)}$ when $\lambda \leq d^{3/4 - \Omega(1)}/\sqrt{n}$.
Obtaining an impossibility result for such large $D$ is relatively strong: in this low-degree test model the typical proxy for polynomial time is $D$ of degree logarithmic in the input size (in this case $md^2$).

\begin{theorem}[Informal, see \cref{thm:roadblock-main}]
  If $\lambda \leq d^{3/4 - \Omega(1)}/\sqrt{n}$ then there is no successful degree $(md)^{o(1)}$ test for the single-spike block mixtures problem.
\end{theorem}

\subsection{Techniques}
\label[section]{sec:techniques}

For purposes of this technical overview, we focus on covariance estimation.
Our algorithm for linear regression employs broadly similar ideas.

\paragraph{The median of means framework}

%\Snote{TODO rewrite techniques section!!!!!!}

%
%\PRfnote{the description in terms of $Q$ feels a little abstract.  Perhaps, explicitly describing the median of means algorithm for mean estimation would help?  We could also mention that, while distribution is heavy-tailed, independence of samples is still present.  Median of means exploits independence of samples to get concentration as follows.. divide into buckets...}

Let us first explain the basic median-of-means trick in one dimension.
Consider the problem of estimating the mean $\mu \in \R$ of a one-dimensional random variable $X$ from independent samples, and suppose $\E (X - \mu)^2 \leq 1$, but make no further assumptions on $X$.
In this setting, the empirical mean $\overline{\mu} = \sum_{i=1}^n X_i$ of $n$ independent samples has $\Pr(|\overline{\mu} - \mu| > t) \leq 1/t^2 n$ by Chebyshev's inequality, and no tighter bound is possible.
By contrast, if $X$ were Gaussian, we would have the exponentially-better bound $\Pr(|\overline{\mu} - \mu| > t)\leq \exp(-t^2 n)$.

The simplest median-of-means trick offers a family of estimators $\hat{\mu}_\delta$ for each $\delta \geq 2^{-0.01 n}$ such that $\Pr(|\hat{\mu}_\delta - \mu| > 100 \sqrt{\log(1/\delta) / n}) \leq \delta$.
First we place $X_1,\ldots,X_n$ into $\Theta(\log(1/\delta))$ equal-size buckets.
In each bucket $i \leq \Theta(\log(1/\delta))$ we let $Z_i$ be the average of the samples in bucket $i$.
Then we let $\overline{\mu}_\delta$ be the median of $Z_1,\ldots,Z_{\Theta(\log(1/\delta))}$.

The analysis is a straightforward use of Chebyshev's inequality to show that each $Z_i$ has $|Z_i - \mu| \leq O(\sqrt{\log(1/\delta)/n})$ with probability at least $0.9$, followed by a binomial tail bound ensuring that with probability at least $1-\delta$ at least a $0.7$ fraction of the $Z_i$'s satisfy this inequality.
Then the key step: if more than half of $Z_1,\ldots,Z_k$ have distance at most $r$ to $\mu$, then so does their median.

\paragraph{Medians in High Dimensions}
Extending this idea to high dimensional settings requires surmounting several hurdles.
The first one is to design an appropriate high-dimensional notion of median.
In the last few years, however, the techniques to do this have become relatively well understood in statistics \cite{lugosi2019mean}.
For example, the key notion in recent heavy-tailed estimators of the mean of a random vector in $d$ dimensions with respect to Euclidean distance is the following: for a set of points $Z_1,\ldots,Z_k \in \R^d$ and $r > 0$, $x \in \R^d$ is an $r$-median if for every unit direction $u$ we have $|\iprod{Z_i,u} - \iprod{x,u}| \leq r$ for at least a $0.51$-fraction of $Z_1,\ldots,Z_k$.
It turns out that using the median of means trick with this notion of median leads to an information-theoretically optimal estimator of the mean in $d$ dimensions assuming only that the underlying random vector has finite covariance.

For covariance estimation the appropriate notion of median was first defined in \cite{lugosi2018near} and fully analyzed in \cite{mendelson2018robust}.
We will call $M$ an $r$-median for matrices $Z_1,\ldots,Z_k$ if for all unit $x \in \R^d$ it holds that $|\iprod{Z_i,xx^\top} - \iprod{M,xx^\top}| \leq r$ for at least a $0.51$-fraction of $Z_1,\ldots,Z_k$.
Then (ignoring some technical details regarding truncation of large samples) one may design a nearly information-theoretically optimal covariance estimator for random vectors $X$ with bounded $4$th moments as follows.
Given samples $X_1,\ldots,X_n$, as before, place them in $\approx \log(1/\delta)$ buckets.
Let $\Sigma_i$ be the empirical covariance in bucket $i$, and output an $r$-median of $\Sigma_1,\ldots,\Sigma_{\Theta(\log(1/\delta))}$ for the least $r$ for which such an $r$-median exists.

\paragraph{How to Compute a Median in High Dimensions}
The next hurdle is computational: naive algorithms to compute the medians described above would seem to require exponential time in $n$ or $d$.
Hopkins \cite{hopkins2018sub} uses the sum of squares method to compute the relevant median for mean estimation in $\ell_2$.
Our main technical contribution for covariance estimation is an algoirthm to compute the relevant median \emph{for values of $r$ somewhat larger (hence making finding the median easier) than information-theoretically optimal (but exponential time) algorithms would do}.
We are able to analyze our algorithm only in the average-case setting that $Z_1,\ldots,Z_k$ whose median we wish to find are the empirical covariances of bucketed independent samples $X_1,\ldots,X_n$.

The key difficulty in computing a median is knowing when we have found one.
We first aim to solve a simpler \emph{certification} problem.
Suppose given $\Sigma_1,\ldots,\Sigma_k$ which are the empirical covariances of independent bucketed copies $X_1,\ldots,X_n$ of a random vector $X$ with covariance $\Sigma$, \emph{and suppose also given $\Sigma$}.
How can we \emph{certify}, for as small a value of $r$ as possible, that $\Sigma$ is an $r$-median of $\Sigma_1,\ldots,\Sigma_k$?
That is, we aim to find a certificate that for all unit directions $u$ we have $|\iprod{\Sigma_i,uu^\top} - \iprod{\Sigma,uu^\top}| \leq r$ for at least a $0.51$-fraction of $\Sigma_1,\ldots,\Sigma_k$.
To leverage the power of the median-of-means trick to obtain estimators whose error is small with high probability, we need to successfully find such a certificate with high probability, $1-2^{-k}$.
(This need for a high-probability guarantee will play the same role in the algorithmic and high-dimensional context as the simple binomial concentration bound does in the one-dimensional median-of-means estimator.)

To certify that $\Sigma$ is an $r$-median for $\Sigma_1,\ldots,\Sigma_k$ we start by setting up an optimization problem in variables $b_1,\ldots,b_k \in \{0,1\}^k$ and $u \in \R^d$ with $\|u\|^2 = 1$.
\begin{align}\label{eq:opt-tech}
\max \sum_{i \leq k} b_i \text{ s.t. } b_i \iprod{\Sigma_i - \Sigma, uu^\top} \geq b_i r, \|u\|^2 = 1, b_i^2 = b_i\mper
\end{align}
Notice that a feasible solution of value $0.52k$ to the above problem corresponds to a subset of $0.52k$ of $\Sigma_1,\ldots,\Sigma_k$ and a unit direction $u$ such that for all $\Sigma_i$ in the subset, $|\iprod{\Sigma_i, uu^\top} - \iprod{\Sigma, uu^\top}| \geq r$.
\emph{Ruling out} such solutions (i.e. placing an upper bound on the value of the optimization problem) would thus certify that $\Sigma$ is an $r$-median (ignoring some small technical issues about the sign of $\iprod{\Sigma_i - \Sigma, uu^\top}$).

We will pass to an efficiently-computable convex relaxation of the optimization problem above.
In particular, we use the degree-$8$ Sum of Squares (SoS) semidefinite programming relaxation of \cref{eq:opt-tech}.
Sum of Squares semidefinite programs are convex relaxations of polynomial optimization problems -- they have seen extensive recent use in algorithm design for high-dimensional statistics. (See e.g. \cite{raghavendra2018high,hopkins2018statistical}.)
Roughly speaking, to show that SoS SDPs can efficiently certify a bound on the optimum of the above optimization problem, we need to prove such an upper bound \emph{using only arguments involving low-degree polynomials in $u,b_i$}.
Now we sketch that proof, which is the technical heart of our algorithm for covariance estimation.

First, we show by applying a bounded-differences concentration inequality to the value of the SoS SDP that the optimum value of the relaxation of \cref{eq:opt-tech} concentrates around its expectation with high probability $(1-2^{-k})$.
(This bounded-differences step appears in the non-algorithmic context in \cite{LM18} and in the algorithmic context in \cite{hopkins2018sub}.)
Then we bound the expected value of the above problem via
\[
  \sum_{i\leq k} b_i \leq \frac 1 r \sum_{i \leq k} b_i \iprod{\Sigma_i - \Sigma, uu^\top} \leq \frac 1 r \cdot \sqrt{k} \cdot \Paren{\sum_{i \leq k} \iprod{\Sigma_i - \Sigma, uu^\top}^2}^{1/2}\mcom
\]
where we have used Cauchy-Schwarz.

The polynomial on the right-hand side is a degree-$4$ polynomial in $u$ with random coefficients; the goal is to upper bound its expected maximum on the unit sphere (via an argument which applies also to the SoS relaxation, which rules out standard approaches using $\e$-nets).
In fact, since we need the bound $0.51k$ on the $\sum_{i \leq k} b_i$, we will eventually take $r$ large enough to compensate for whatever is our bound on $\sum_{i \leq k} \iprod{\Sigma_i - \Sigma,uu^\top}^2$.
We want to keep $r$ small, so we want the tighest bound possible.

Note that $\sum_{i \leq k} \iprod{\Sigma_i - \Sigma,uu^\top}^2$ is a sum of i.i.d. random polynomials.
A standard approach to analyze the performance of SoS for such random polynomials is to first ``unfold'' the polynomial to a matrix (in this case $\sum_{i \leq k} (\Sigma_i - \Sigma)^{\tensor 2}$) and then use matrix concentration inequalities to analyze the maximum eigenvalue of this random matrix.
Such eigenvalue bounds will also apply to the SoS relaxation we work with in the end.

We use a similar approach, with a key technical twist: in previous applications of this idea, it was usually necessary to have an explicit expression for $\E M$, where $M$ is the random matrix analogous to $(\Sigma_i - \Sigma)^{\tensor 2}$, and typically also for its inverse, in order to correctly ``precondition'' the random matrix before analyzing its top eigenvalue.
Such an explicit representation would be easily accessible if the underlying data $X$ were Gaussian or had independent coordinates, for example, which was the case in previous applications of SoS to random degree-$4$ polynomials.
We do not have this luxury, since we only make the niceness assumption on the underlying random vector $X$.

Nonetheless, we are able to carry out the preconditioning strategy (which removes spurious large eigenvalues of $(\Sigma_i - \Sigma)^{\tensor 2}$) for any nice random variable $X$.
Along the way we prove a new (albeit simple) SoS Bernstein inequality which may be of independent use (and in particular allows for simplifed proofs of some previous applications of SoS to random degree-$4$ polynomials -- e.g. that of \cite{BarakBHKSZ12}).
See \cref{thm:sosmatbern} for the SoS Bernstein inequality and \cref{lem:p-main} for our application to the random polynomial $\sum_{i \leq k} \iprod{\Sigma_i - \Sigma, uu^\top}^2$.

\paragraph{Certification to Search} Using similar techniques as \cite{cherapanamjeri2019fast} developed for the case of $\ell_2$ mean estimation, we turn our certification into an algorithm to \emph{find} an $r$-median.
Suppose that instead of knowing the true covariance $\Sigma$ as above, in its place we have some guess $M \in \R^{d \times d}$.
If the certification algorithm certifies that $M$ is a median, then we can output $M$ as our estimator for $\Sigma$.
If not, we show that by rounding the above SoS relaxation we can instead update $M$ to make it closer to $\Sigma$ -- we can replace it with $M + \Delta$ such that $\|M+\Delta - \Sigma\| \ll \|M-\Sigma\|$.

\subsection{Related Work}

\paragraph{Robust Statistics}
The questions we address here are distinct from those addressed by a recent flurry of algorithmic work in \emph{robust} statistics \cite{DBLP:conf/focs/DiakonikolasKK016, DBLP:conf/focs/LaiRV16} (see also \cite{li2018principled, steinhardt2018robust} for further references).
In the latter setting, one studies statistics when the list of samples $X_1,\ldots,X_n$ contains a small constant fraction $\e$ of adversarially-chosen outliers, and the primary focus is on achieving statistical error nearly as small as would be achieved by the classical estimators when $\e = 0$.
By contrast, \emph{our goal is to beat the error rate of the classical estimators when Gaussianity is violated}.
One consequence is that we give estimators which come with small confidence intervals even for error probabilities as low as $2^{-d}$; this high-probability regime is not addressed by the adversarial corruptions model.\footnote{One recent work, \cite{lecue2019robust}, shows that while the adversarial robustness model and the ones we consider here are incomparable, under some circumstances the same algorithm can give information-theoretically optimal estimates in both models. This work, however, does not address covariance estimation or linear regression -- it is an interesting direction to understand to what extent algorithms for covariance estimation and linear regression can perform well across different models.}

\paragraph{Median of Means}
In heavy-tailed (constantly-many moments exist) settings, estimators based on empirical averages typically have poor statistical performance, because they are sensitive to large outliers.
Our work falls in a long line which develop the \emph{median of means} technique for high-probability estimators in the face of heavy tails.
The median of means framework was first developed to estimate univariate heavy-tailed random variables \cite{NemYud83, jerrum1986random, alon1999space}.
Recent extensions to the multivariate case typically have two flavors: they are polynomial-time computable (e.g. \cite{HsuSab16,lerasle2011robust,Min18}) but statistically suboptimal, or statistically optimal (\cite{lugosi2017sub, lugosi2018near, lugosi2016risk}) but apparently require exponential computation time.
%using a natural generalization of the one-dimensional median -- the geometric median. In $d$ dimensions this median-of-means estimator improves over the empirical mean by improving the dependence on the confidence level $k$, but is still worse then the rate achievable for estimating the mean of a Gaussian random vector. The question of whether it was statistically feasible to obtain Gaussian-like
%confidence intervals for the problem of $d$-dimensional mean estimation was finally resolved
%by \cite{lugosi2017sub} using a median-of-means "tournament". Although statistically-optimal, the median-of-means "tournament" is computationally intractable\footnote{see techniques section for a review of these concepts.}.
The first major exceptions to this rule came in 2018, starting with a polynomial-time statistically-optimal algorithm for mean estimation in $\ell_2$ \cite{hopkins2018sub}.
Because of reliance on high-degree sum of squares semidefinite programs, this algorithm has an enormous polynomial running time.
The subsequent work \cite{cherapanamjeri2019fast} brought the running time much closer to practicality by replacing some of the sum of squares tools with a gradient-descent style algorithm.
(\cite{lei2019fast,lecue2019robust} brought the running times down even further.)
The present work builds substantially on ideas from both these papers.

\paragraph{Covariance Estimation}

There is a long and rich literature on the problem of covariance estimation (see \cite{fan2016overview} for an expository review). However, strong high-confidence guarantees for many such estimators rely on the assumption that the samples are drawn from a sub-Gaussian distribution.
%or in other words from distributions that have \textit{light} tails.
The problem of robustly estimating covariance only assuming boundedness of low-order moments on the underlying distribution has also received attention; however many rigorous theoretical results in this vein are either asymptotic (i.e. concern only the $n\rightarrow \infty$ limit for fixed dimensions $d$) and/or often impose strong parametric assumptions on the underlying distribution (i.e. requiring elliptical symmetry). See \cite{tyler1987distribution, fan2016overview} for example, for a coverage of several such results.

The state-of-the-art results for the problem we consider here have been recently achieved in the works of \cite{minsker2018robust} and \cite{mendelson2018robust}. These results have come in two flavors, paralleling recent work in the problem of heavy-tailed mean estimation: \cite{minsker2018robust}[Corollary 4.1] provides computationally-efficient but information-theoretically suboptimal estimators while \cite{mendelson2018robust}[Theorem 1.9] provides statistically-optimal estimators that require exponential time (in $n, d$) to compute. 
%To understand this gap, let us consider a distribution $\cD$ which satisfies an bounded kurtosis condition (see \cref{def:nice-dist} for example) with covariance $\Sigma$. For $n \geq \tO(d)$ with probability at least $1-2^{-\Omega(k)}$\footnote{Here we assume that $\normt{\Sigma} = O(1)$ and $\Tr(\Sigma) = O(d)$ to simplify the presentation of the results.},  \cite{minsker2018robust} provides a polynomial-time estimator $\hat{\Sigma}_1$ satisfying $\normt{\hat{\Sigma}_1-\Sigma} \leq \tO \left( \sqrt{\frac{d k}{n}} \right)$ while \cite{mendelson2018robust} provides an estimator (for which there is no known polynomial-time algorithm) $\hat{\Sigma}_2$ satisfying $\normt{\hat{\Sigma}_2-\Sigma} \leq \tilde{O} \left(\sqrt{\frac{d}{n}} + \sqrt{\frac{k}{n}} \right)$ . Up to logarithmic factors, it can be shown the result in \cite{mendelson2018robust} is essentially statistically-optimal. 

% These results raise the question of whether it is possible to break this computational-statistical gap by constructing a polynomial-time algorithm that avoids a strong coupling between dimension ($d$) and the desired confidence level ($k$). In short, can the $\sqrt{dk}$ dependence of \cite{minsker2018robust} be improved under any weak-moment conditions, with a polynomial-time estimator?

\paragraph{Linear Regression}
Like covariance estimation, linear regression is an old and well-studied topic and a thorough survey is out of the scope of this paper.
Regression in the heavy-tailed and high-dimensional setting has been studied via the median-of-means framework in \cite{lugosi2016risk,hsu2016loss,lugosi2017regularization}.
There are also efficient outlier-robust algorithms for linear regression which use techniques besides median-of-means estimation -- for instance, the iterative methods of \cite{suggala2019adaptive} -- but none are yet known to achieve information-theoretically optimal error. 
In particular we are not aware of any which improve on the guarantees of \cite{hsu2016loss} in our setting, while our algorithms offer $\poly(d)$ improvements on the error rates of \cite{hsu2016loss}.

% \Snote{TODO: do this please}

\paragraph{Sum of Squares Algorithms for High-Dimensional Statistics}
There has been a significant amount of recent work using the sum of squares (SoS) semidefinite programming hierarchy to design computationally efficient algorithms for unsupervised learning problems (see \cite{RSS18} for a survey).
By now, SoS algorithms are the only ones known which gives state-of-the-art statistical performance among polynomial-time algorithms for a wide range of problems: dictionary learning, tensor decomposition, high-dimensional clustering, robust parameter estimation and regression, and more \cite{DBLP:conf/stoc/BarakKS15,hopkins2018mixture,kothari2018robust,DBLP:conf/focs/MaSS16,klivans2018efficient,DBLP:conf/colt/BarakM16}.
%which achieve often leads to the only algorithms which achieve the best known statistical and computational guarantees. Inherently, SoS based relaxations for statistics problems work by constructing identifiability certificates, i.e., by finding a proof which says that given a set of samples, an \textit{estimator} of the hidden parameter of interest (say $\theta$) given these samples is close to the true value of the hidden parameter $\theta^*$, then any sufficiently large set of samples uniquely (up to potentially some error) determines $\theta^*$. \Nnote{this sentence should be rephrased}

We note that one of our techniques for exploiting $8$-th moments is inspired by a certain approach to using the Cauchy-Schwarz inequality in SoS proofs for bounding degree-$3$ random polynomials by degree-$4$ random polynomials.
This technique is in turn inspired by refutation algorithms for random constraint satisfaction problems \cite{DBLP:journals/toc/FeigeO07}, and has been used in the design of SoS algorithms for several learning problems \cite{ge2015decomposing,DBLP:conf/colt/HopkinsSS15,DBLP:conf/colt/BarakM16}.
We also note that the certify-or-gradient paradigm used by our algorithms, where gradients are furnished by solving SDPs, has previously appeared in robust and heavy-tailed mean estimation \cite{cheng2019high,cherapanamjeri2019fast}; these works do not combine this technique with SoS SDPs of degree greater than $2$.

Our algorithms using the SoS hierarchy run in polynomial time, but because of their reliance on solving large semindefinite programs they are impractical.
However, numerous slow-but-polynomial-time SoS algorithms for high-dimensional statistics have led to algorithms with practical nearly-linear running times \cite{schramm2017fast,dong2019quantum,DBLP:conf/stoc/HopkinsSSS16,lecue2019robust,hopkins2019robust,cherapanamjeri2019fast}.
We therefore hope that additional investigation can lead to SoS-inspired and practical algorithms with improved guarantees for heavy-tailed covariance estimation and regression.

%The existence of such certificates can usually be framed as polynomial optimization problems. Consequently, one designs a Sum-of-squares relaxation of this polynomial optimization problem and prove, by devising a certificate to the non-negativity of a polynomial which can be written only as a sum-of-squares. In many cases, the polynomial optimization problem is a polynomial of constant degree (often degree 3 or 4) and a constant degree (again usually degree 4) Sum-of-squares relaxation to the problem suffices. A large class of problems in this space include random 3-SAT \cite{DBLP:journals/toc/FeigeO07}, tensor completion \cite{DBLP:conf/colt/BarakM16}

\paragraph{Certifiable Hypercontractivity}
Our algorithms for covariance estimation and linear regression assume the underlying random vector $X$ is $(2,8)$ certifiably hypercontractive.
The certifiable hypercontractivity assumption was introduced in \cite{kothari2018robust,hopkins2018mixture} where it was used in designing algorithms for robust estimation and mixture model clustering.
It has been used in the context of regression by \cite{klivans2018efficient}.
Previous work using certifiable hypercontractivity assumptions (for example in clustering mixture models) typically assumed the presence of a $\poly(d)$-factor more samples than information-theoretically necessary in order to ensure the convergence of empirical moments to these population averages.
Since we are interested in fine-grained questions about the number of samples required to achieve certain rates of statistical error, a major portion of the technical work in our paper is to show that SoS algorithms can exploit structure in the population moments even with relatively few samples.
\cite{hopkins2019hard,DBLP:conf/stoc/BarakBHKSZ12} investigate computational hardness questions surrounding certifiable hypercontractivity.

\section{Preliminaries}

We write $\|M\|_2$ for the spectral norm of a matrix $M$, $\|M\|_1$ for its nuclear norm, and $\|M\|_F$ for its Frobenius norm.
The notation $\iprod{v,w}$ always indicates the Euclidean inner product of vectors, and for matrices $A,B$ we write $\iprod{A,B} = \Tr A B^\top$.

\subsection{SoS Basics}
We refer the reader to \cite{sos-notes-general} for most basic SoS definitions and discussion; we note here just a few pieces of notation.
If $p,q$ are polynomials, we write $p \preceq q$ to denote that $q - p$ is a sum of squares.
If $\cA$ is a set of polynomial inequalities, we write $\cA \proves_t p \geq 0$ to indicate that there is a degree-$t$ SoS proof that $p \geq 0$ using axioms $\cA$.

\subsection{Certifiably Hypercontractive Distributions}
Our algorithms for regression and covariance estimation will work for a class of distributions for which SoS can certify upper bounds on low-order moments.
In particular, we make the following definition:

\begin{definition}[Certifiable $(2,8)$ Hypercontractivity]
\label[definition]{def:cert-hyper}
  Let $X$ be a mean-zero random variable on $\R^d$.
  We say that $X$ is $L$-certifiably $(2,8)$-hypercontractive for some number $L > 0$ if
  \[
  \E \iprod{X,u}^8 \preceq L^2 \cdot (\E \iprod{X,u}^2)^4
  \]
  where left and right-hand sides are polynomials in $u$.
  We make the analogous definition for certifiable $(2,4)$-hypercontractivity.
\end{definition}

Our covariance estimation and regression algorithms will concern the following class of distributions.

\begin{definition}[Nice Distributions]
\label[definition]{def:nice-dist}
  We say a mean-zero random vector $X$ on $\R^d$ is $L$-nice if it has $\E X = 0$ and it is $L$-certifiably $(2,8)$-hypercontractive and $L$-certifiably $(2,4)$-hypercontractive.
\end{definition}
%\Nnote{Say something about which distributions satisfy this.}
\subsection{Random Matrices}
% We will use the matrix Bernstein inequality to prove a generalization we refer to as the SoS matrix Bernstein inequality which is sharper when applied to certifiably hypercontractive distributions. 
We state the matrix Bernstein inequality which we repeatedly use throughout,

\begin{lemma}[Matrix Bernstein -- see \cite{DBLP:journals/focm/Tropp12}]
\label[lemma]{lem:matrix-bernstein}
  Let $S_1,\ldots,S_n$ be independent symmetric random $d \times d$ matrices.
  Suppose that each has $\normt{S_k - \E S_k} \leq R$ with probability $1$.
  Let $Z = \sum_{i=1}^k S_i$, and let
  \[
  \sigma^2 = \normt{\E Z^2 - (\E Z)^2 }\mper
  \]
  Then
  \[
  \E \normt{Z - \E Z} \leq \sqrt{2 \sigma^2 \log(2d)} + \frac 1 3 R \log(2d)\mper
  \]
\end{lemma}

\section{Degree-$8$ SoS for Certifiable Distributions}\label{sec:random_poly}

In this section we state and prove a key bound on the expected maximum value of degree-$8$ SoS relaxations of certain random polynomial optimization problems on the unit sphere.
We will use this bound to control the expected error of both our covariance and linear regression estimators.

\paragraph{Setup and truncation}
Let $v$ be a $d$-dimensional $L$-nice mean-zero random vector with covariance $\Sigma$.
For $\alpha > 0$, we define $\tv = v \cdot \Ind(\|v\| \leq \alpha)$ as the $\alpha$-truncation of $v$.
Let $\tSigma$ be the second moment matrix of $\tv$, $\tSigma = \E[\tv \tv^\top]$.

Let $v_1,\ldots,v_n$ be i.i.d. copies of $v$.
Let $B_1,\ldots,B_k$ partition $[n]$ into $k$ equal-sized buckets, and for each $i \leq k$ let
\[
  Z_i = \frac{1}{m} \sum_{j \in B_i} \tv_j \tv_j^\top - \tSigma\mcom
\]
where $m=n/k$.
For $\tau \in \R$, let $\tZ_i = Z_i \cdot \Ind(\|Z_i\|_2 \leq \tau)$ be the $\tau$-truncation of $Z_i$.

\begin{definition}
We study the following polynomial $p_{v_1,\ldots,v_n}(u)$ in variables $u = (u_1,\ldots,u_d)$:
\begin{align}
  \label{eq:p-def}
  p_{v_1,\ldots,v_n}(u) = \sum_{i \leq k} \iprod{\tZ_i \tensor \tZ_i, uu^\top \tensor uu^\top} = \sum_{i \leq k} \iprod{u, \tZ_i u}^2\mper
\end{align}
Notice that $\iprod{u,\tZ_i u} = \frac{1}{m} \sum_{j \in B_i} \iprod{u,\tv_j}^2 - \E \iprod{u,\tv}^2$, so the polynomial $p$ measures the squared deviations of the quadratic forms of $\tZ_1,\ldots,\tZ_k$ in the direction $u$.
\end{definition}

\noindent Our main lemma gives an upper bound on $p$, and shows furthermore that this bound can be certified by a degree-$8$ SoS SDP.

% \begin{lemma}
%   \label[lemma]{lem:p-main}
%   Let $v_1,\ldots,v_n \sim \cD$ be i.i.d and assume the distribution $\cD$ satisfies \cref{def:nice-dist}.
%   Then for $p$ as in \eqref{eq:p-def},
%   \begin{align*}
%   & \E_{v_1,\ldots,v_n} \Brac{ \max_{\pE_u} \pE_u p(u) } \\
%   & \leq O\Paren{\frac 1 n \cdot k^{3/2} \cdot L \cdot \Tr \Sigma \cdot \Brac{\|\Sigma\|_2 + \sqrt{\frac k n} \cdot \|\Sigma\|_F} \cdot \sqrt{\log d} } + O \Paren{ \frac{k^2}{n} \cdot L \cdot \|\Sigma\|_2^2} + O \Paren{\tau^2 \cdot \log d}\mper
%   \end{align*}
%   where the maximum is over all degree-$4$ pseudodistributions in variables $u = (u_1,\ldots,u_d)$ satisfying $\{\|u\|^2 = 1\}$.
% \end{lemma}

\begin{lemma}
  \label[lemma]{lem:p-main}
  Let $v_1,\ldots,v_n \sim \cD$ be i.i.d and assume the distribution $\cD$ satisfies \cref{def:nice-dist}.
  Then for $p$ as in \eqref{eq:p-def},
  \begin{align*}
  & \E_{v_1,\ldots,v_n} \Brac{ \max_{\pE_u} \pE_u p(u) } \\
  & \leq O\Paren{\frac 1 n \cdot k^{3/2} \cdot L \cdot \Tr \Sigma \cdot \|\Sigma\|_2 \cdot \sqrt{\log d} } + O \Paren{ \frac{k^2}{n} \cdot L \cdot \|\Sigma\|_2^2} + O \Paren{\tau^2 \cdot \log d}\mper
  \end{align*}
  where the maximum is over all degree-$8$ pseudodistributions in variables $u = (u_1,\ldots,u_d)$ satisfying $\{\|u\|^2 = 1\}$.
\end{lemma}

Our key technical tool to obtain a sharp bound on the random fluctuations of the polynomial $p$ in the proof of \cref{lem:p-main} is a sum-of-squares generalization of the matrix Bernstein inequality.
We think this inequality may be of independent interest.

\begin{theorem}[SoS Matrix Bernstein]
    \label{thm:sosmatbern}
  Let $M_i \in \R^{d^r \times d^r}$ be a sequence of i.i.d., mean-zero, random symmetric matrices satisfying the conditions,
  \begin{itemize}
      \item $\norm{M_i} \leq R$ almost surely,
      \item $\E[\langle u^{\otimes r}, M_i^2 u^{\otimes r} \rangle] \preceq \sigma^2 \norm{u}^{2r}$ where the left and right-hand sides are polynomials in $u$.
  \end{itemize}
  Then,
  \begin{align}
      \E \Brac{ \max_{\tE} \tE \Brac{ \sum_{i \leq k} \langle u^{\otimes r}, M_i u^{\otimes r} \rangle } } \leq \left( \frac{2 (\log(2) + r \log d)}{3} \cdot R + 2\sqrt{2k (\log(2) + r \log d)} \cdot \sigma \right)
  \end{align}
  where the maximum is taken over all degree-$2r$ pseudoexpectations satisfying the polynomial inequality $\{ \norm{u}^2=1 \}$.
\end{theorem}

Results in \cite{BarakBHKSZ12, hopkins2016fast} anticipate the result of \cref{thm:sosmatbern}. However, these results strongly exploit the Gaussianity of the underlying random matrices while \cref{thm:sosmatbern} applies to a broader class of  random matrices. The core idea in the proof of the SoS matrix Bernstein inequality is to apply the standard Matrix Bernstein inequality to an appropriately preconditioned version of the matrix $M$.
Crucially, this allows us to bound the fluctuations of the of the (random) polynomial in terms of the ``SoS norm'' of the matrix variance term $\sum M_i^2$ as opposed to the spectral norm of $\sum M_i^2$, which would yield a much cruder bound.
% Effectively, this allows a sharper (SoS) bound on the matrix variance-like term, while a direct application of the standard matrix Bernstein inequality resorts to a cruder spectral norm bound on the corresponding term \Ynote{Maybe say here that this crucially allows us to bound the expected value of the polynomial by the SoS norm of the matrix variance as opposed to its spectral norm which is much cruder? The current phrasing is a bit confusing.}.

Our proof of \cref{lem:p-main} has two steps: an expectation step in which we will control $\pE \E_{v_1,\ldots,v_n} p(u)$, and a deviation step in which we control the remaining random fluctuations which critically uses \cref{thm:sosmatbern}. These are captured by the following two lemmas, the proofs of which we present below.

The first lemma exploits the rank-one structure of the tensor $uu^\top \tensor uu^\top$ inside $\pE \E_{v_1,\ldots,v_n} p(u)$:
\begin{lemma}[Expectation of $p$]
\label[lemma]{lem:p-expectation}
    Assume the distribution $\cD$ satisfies \cref{def:nice-dist} and that $v_i \sim \mathcal{D}$ i.i.d. Then for all degree-$8$ pseudoexpectations $\pE_u$ that satisfy the polynomial equation $\norm{u}^2=1$,
  \[
   \pE_u \E_{v} p(u) = \sum_{i=1}^{k} \pE_u \E_v \iprod{\tZ_i \tensor \tZ_i, uu^\top \tensor uu^\top} = \sum_{i=1}^k \pE_u \E_v \iprod{u, \tZ_i u}^2 \leq O \left(\frac{L k}{m} \right) \normt{\Sigma}^2 = O \Paren{\frac{k^2}{n} \cdot L \cdot \normt{\Sigma}^2}\mper
  \]
\end{lemma}

\noindent The deviations of $p$ are controlled by the following result.
% \begin{lemma}[Expectation of $p$]
%   \label[lemma]{lem:p-expectation}
%   Let $X_1,\ldots,X_n,p$ be as above.
%   Then $\E_{X_1,\ldots,X_n} p(u) \preceq O(k^2/n + 1) \cdot \|\Sigma\|^2 \cdot \|u\|^4$.
% \end{lemma}

% \begin{proof}[Proof of \cref{lem:p-expectation}]
%   Recall that $Z_1,\ldots,Z_k$ are i.i.d. copies of a random variable $Z = \frac kn \sum_{i=1}^{n/k} \tX_i \tX_i^\top - \Sigma$ where $X_1,\ldots,X_{n/k}$ are i.i.d. copies of $X$.
%   Also since for every fixed $Z$ we have $\iprod{u,Zu}^2$ is a square, $\E \iprod{u,\tZ u}^2 \preceq \E \iprod{u, Zu}^2$.
%   Now we compute that:
%   \begin{align*}
%   \E \iprod{u,Zu}^2 & = \E \Paren{ \Iprod{u, \frac kn \Paren{\sum_{i \leq n/k} \tX_i \tX_i^\top} u} - \iprod{u,\Sigma u}}^2\\
%   & \preceq 2 \E \Iprod{u, \frac kn \Paren{\sum_{i \leq n/k} \tX_i \tX_i^\top} u}^2 + 2 \iprod{u,\Sigma u}^2 \text{ by SoS triangle inquality} \\
%   & \preceq 2 \E \Iprod{u, \frac kn \Paren{\sum_{i \leq n/k} \tX_i \tX_i^\top} u}^2 + 2 \|\Sigma\|^2 \|u\|^4 \text{ by $\iprod{u,\Sigma u} \preceq \|u\|^2 \|\Sigma\|$} \\
%   & = \frac{2 k^2}{n^2} \E \sum_{i,j \leq k/n} \iprod{X_i,u}^2 \iprod{X_j,u}^2 + 2 \|\Sigma\|^2 \|u\|^4 \text{ expanding and dropping truncation} \\
%   & \preceq 2 L \cdot \frac k n \|\Sigma\|^2 \|u\|^4 + 4 \|\Sigma\|^4 \|u\|^4 \text{ by certifiable $(2,4)$-hypercontractivity } 
%   \end{align*}
%   The lemma follows by adding $k$ copies of $\E \iprod{u,Z u}^2$.
% \end{proof}

\begin{lemma}[Deviations of $p$]
\label[lemma]{lem:p-deviation}
  Assume the distribution $\cD$ satisfies \cref{def:nice-dist} and that $v_1,\ldots,v_n \sim \mathcal{D}$ are i.i.d. Recalling that for $\tau > 0$ we defined $\tilde{Z_i} = Z_i \cdot \Ind[\|Z_i\|_2 \leq \tau]$,
  \begin{align*}
  \E_v \max_{\pE_u} \pE_{u} \left[ p(u) - \E_{v} p(u) \right] \leq O(\log(d)\tau^2) + O(\sqrt{k \log d}) \cdot \sqrt{ O \Paren{\frac{1}{m^2}} L^2 (\Tr \Sigma)^2 \normt{\Sigma}^2} \mcom
  \end{align*}
  where the maximum is taken over degree-8 pseudodistributions $\pE_u$ satisfying $\{\norm{u}^2=1\}$.
\end{lemma}

\noindent Now we are prepared to prove \cref{lem:p-main}.

\begin{proof}[Proof of \cref{lem:p-main}]
\noindent We can simply center the random polynomial and apply \cref{lem:p-expectation,lem:p-deviation} to the first and second terms:
\begin{align*}
    &  \E_{v} \Brac{ \max_{\pE_u} \pE_u  p(u) } \leq \E_v \max_{\pE_{u}} \pE_{u} \left[p(u) - \E_v p(u) \right] + \max_{\pE_u} \pE_u \E_v p(u) 
\end{align*}
to conclude \cref{lem:p-main}.
\end{proof}

\subsection{Proofs of \cref{lem:p-expectation,lem:p-deviation}}

We turn to the proofs of \cref{lem:p-expectation,lem:p-deviation}, starting with the former.

\begin{proof}[Proof of \cref{lem:p-expectation}]
  Recall that $\tZ = (\frac kn \sum_{i=1}^{n/k} \tv_i \tv_i^\top - \tilde{\Sigma}) \Ind(\normt{Z} \leq \tau)$ where $v_1,\ldots,v_{n/k}$ are i.i.d. samples $v_i \sim \cD$.
  For any $\tZ_i$ we have that,
  \[ \pE_u \E_v \iprod{u, \tZ_i u}^2 = \pE_u \E_v \iprod{u, Z_i u}^2 - \pE_u \E_v  [\Ind(\normt{Z_i} \geq \tau) \iprod{u, Z_i u}^2]   \leq \pE_u \E_v \iprod{u, Z_i u}^2 \]
  noting that $\pE_u \E_v [\Ind(\normt{Z_i} \geq \tau) \iprod{u, Z_i u}^2] \geq 0$ to discard the second term.
  Continuing by expanding $Z_i$, we have for any $\pE_u$,
  \begin{align*}
  \pE_u \E_v \iprod{u, Z_i u}^2 & \leq \frac k n \left( \pE_u \E[ \iprod{v_i,u}^4] - \E [\iprod{v_i, u}^4 \Ind[\normt{v} \geq \alpha]] - \pE_u \iprod{u, \tSigma u}^2 \right)\\
  & \leq \frac{k}{n} \sqrt{\pE_u \E[\langle v, u \rangle^8]} \\
  & \leq L \cdot \frac k n \cdot \sqrt{ \pE_u (\E [\iprod{v,u}^2])^4}\\
  & \leq L \cdot \frac{k}{n} \normt{\Sigma}^2
  \end{align*}
  noting in the first inequality $\pE_u[\iprod{v, u}^4 \Ind[\normt{v} \geq \alpha]] \geq 0$ and $\pE_u \iprod{\tSigma \tensor \tSigma, uu^{\top} \tensor uu^{\top}} = \pE_u \iprod{\tSigma, uu^\top}^2 \geq 0$ are squares so the intermediate terms can be discarded. The second inequality appeals to pseudoexpectation Cauchy-Schwarz, the third to certifiable L$8$-L$2$ hypercontractivity, and the final an SoS spectral bound. Assembling, we obtain
  \[
  \sum_{i=1}^{k} \pE_u \E_Z \iprod{\tZ_i \tensor \tZ_i, uu^{\top} \tensor uu^{\top}} \leq \frac{L k}{m} \normt{\Sigma}^2 \mper
  \]
  for any $\pE_u$.
\end{proof}

It remains to prove \cref{lem:p-deviation}.
To do so we need one further lemma, which we prove at the end of this section.
The approach to prove \cref{lem:p-deviation} is to use an application of the SoS-version of the matrix Bernstein inequality
The latter requires a careful computation of the SoS \emph{matrix variance}, which will leverage the certifiably, hypercontractive properties of the nice distribution $\cD$.
This is captured by the following lemma.
% \begin{lemma}\label[lemma]{lem:mat-var-term}
%   Assume the distribution $\cD$ satisfies \cref{def:nice-dist} and that $v_i \sim \mathcal{D}$ i.i.d. Let $Z = \frac 1m \sum_{i \leq m} \tv_i \tv_i^\top - \tilde{\Sigma}$.
%   Then
%   \begin{align*}
%   \normt{ \E [\tZ^2 \tensor \tZ^2]} \leq \normt{ \E [Z^2 \tensor Z^2]} \leq  O \left(\frac{1}{m^3} L^2 \Tr(\Sigma)^2 \norm{\Sigma}_F^2 + \frac{1}{m^2} L^2 \normt{\Sigma}^2 \Tr(\Sigma)^2 \right) \mper
%   \end{align*}
% \end{lemma}

\begin{lemma}
    \label[lemma]{lem:sosmvbound}
    Consider the random matrix, $Z = \frac{1}{m} \sum_{i = 1}^m (\tv_i \tv_i^\top - \tilde{\Sigma})$ where $\tv_i = v_i \Ind[\norm{v_i} \leq \alpha]$ and $v_i$ are drawn i.i.d. from a distribution $\mathcal{D}$ satisfying \cref{def:nice-dist} with covariance matrix $\Sigma$. Let $\tZ = Z \Ind[\norm{Z} \leq \tau]$.  Then, we have that for any pseudoexpectation, $\pE$, satisfying $\norm{u}^2 = 1$:
    \begin{equation*}
        \pE \iprod{\E [\tZ^2 \tensor \tZ^2] - (\E [\tZ \tensor \tZ])^2, u^{\tensor 4}} \leq O \Paren{\frac{1}{m^2}} L^2 (\Tr \Sigma)^2 \normt{\Sigma}^2.
    \end{equation*}
\end{lemma}

\noindent Using \cref{lem:sosmvbound} we can complete the proof of \cref{lem:p-deviation}:

\begin{proof}[Proof of \cref{lem:p-deviation}]
Directly applying the SoS matrix Bernstein inequality
$\E_v \max_{\pE_u} \pE_{u} \left[ p(u) - \E_{v} p(u) \right]$, along with the fact that $\norm{\tZ} \leq \tau$ almost surely and the sharp computation of the SoS matrix variance term from \cref{lem:sosmvbound}, shows,
  \begin{align*}
  \E_v \max_{\pE_u} \pE_{u} \left[ p(u) - \E_{v} p(u) \right] \leq O(\log(d)\tau^2) + O(\sqrt{k \log d}) \cdot \sqrt{ O \Paren{\frac{1}{m^2}} L^2 (\Tr \Sigma)^2 \normt{\Sigma}^2} \mcom
  \end{align*}
\end{proof}

We conclude by presenting the proof of \cref{lem:sosmvbound}  which uses the niceness of $\cD$.% \begin{proof}[Proof of \cref{lem:sosmvbound}]

\begin{proof}[Proof of \cref{lem:sosmvbound}]
Firstly, we have:
\begin{equation*}
    \pE \iprod{(\E [\tZ \tensor \tZ])^2, u^{\tensor 4}} = \pE \iprod{(u^{\tensor 2})^\top(\E [\tZ \tensor \tZ])^2, u^{\tensor 2}} \geq 0.
\end{equation*}
Therefore, we can ignore the second term in the inner product. From this, we get from the fact that $Z^2 \tensor Z^2$ is positive semidefinite:
\begin{equation*}
    \pE \iprod{\E [\tZ^2 \tensor \tZ^2] - (\E [\tZ \tensor \tZ])^2, u^{\tensor 4}} \leq \pE \iprod{\E [\tZ^2 \tensor \tZ^2], u^{\tensor 4}} = \pE \iprod{\E [Z^2 \tensor Z^2 \Ind[\norm{Z} \leq \tau]], u^{\tensor 4}} \leq \pE \iprod{\E [Z^2 \tensor Z^2], u^{\tensor 4}}.
\end{equation*}
We now expand the right hand side as follows:
\begin{align*}
    \pE \iprod{\E [Z^2 \tensor Z^2], u^{\tensor 4}} &= \frac{1}{m^3} \pE \Iprod{\E [((\tv\tv^\top - \tSigma)^2)^{\tensor 2}], u^{\tensor 4}} + O \Paren{\frac{1}{m^2}} \pE \Iprod{\E [((\tv\tv^\top - \tSigma)^2) \tensor ((\tw\tw^\top - \tSigma)^2)], u^{\tensor 4}}\\
    &+ O \Paren{\frac{1}{m^2}} \pE \Iprod{\E [((\tv\tv^\top - \tSigma) (\tw\tw^\top - \tSigma)) \tensor ((\tw\tw^\top - \tSigma) (\tv\tv^\top - \tSigma)))], u^{\tensor 4}} \\
    &+ O \Paren{\frac{1}{m^2}} \pE \Iprod{\E [((\tv\tv^\top - \tSigma) (\tw\tw^\top - \tSigma))^{\tensor 2})], u^{\tensor 4}}.
\end{align*}
Note, that all other terms that arise vanish under expectations. For the first term, we have:
\begin{align*}
    \pE \Iprod{\E [((\tv\tv^\top - \tSigma)^2)^{\tensor 2}], u^{\tensor 4}} &= \pE \E [(u^\top (\tv\tv^\top - \tSigma)^2 u)^2] = \pE \E [(\norm{\tv}^2 \iprod{u, \tv}^2 -2 \iprod{\tv, u} \tv^\top \tSigma \tv + u^\top \tSigma^2 u)^2] \\
    &\leq \pE \E [(2\norm{\tv}^2 \iprod{u, \tv}^2 + 2 u^\top \tSigma^2 u)^2] \leq 4 \pE \E [2(\norm{\tv}^2 \iprod{u, \tv}^2)^2 + 2 (u^\top \tSigma^2 u)^2] \\
    &\leq 8 \pE \E [\norm{\tv}^4 \iprod{u, \tv}^4 + (u^\top \tSigma^2 u)^2] \leq 8 \pE \E [\norm{\tv}^4 \iprod{u, \tv}^4] + 8\normt{\tSigma}^4 \\
    &\leq 8 (\E [\norm{\tv}^8])^{1 / 2} (\pE \E [\iprod{u, \tv}^8])^{1 / 2} + 8\normt{\tSigma}^4.
\end{align*}
For the second term, we have:
\begin{align*}
    &\pE \Iprod{\E [((\tv\tv^\top - \tSigma)^2) \tensor ((\tw\tw^\top - \tSigma)^2)], u^{\tensor 4}} = \pE \E [(u^\top(\tv\tv^\top - \tSigma)^2u) (u^\top(\tw\tw^\top - \tSigma)^2 u)] \\
    &\qquad \leq \frac{1}{2} \cdot \pE \E [(u^\top(\tv\tv^\top - \tSigma)^2u)^2 + (u^\top(\tw\tw^\top - \tSigma)^2 u)^2] = \pE \E [(u^\top(\tv\tv^\top - \tSigma)^2u)^2] \\
    &\qquad \leq 8 (\E [\norm{\tv}^8])^{1 / 2} (\pE \E [\iprod{u, \tv}^8])^{1 / 2} + 8\normt{\tSigma}^4.
\end{align*}
Similarly, for the third term, we get:
\begin{align*}
    &\pE \Iprod{\E [((\tv\tv^\top - \tSigma) (\tw\tw^\top - \tSigma)) \tensor ((\tw\tw^\top - \tSigma) (\tv\tv^\top - \tSigma)))], u^{\tensor 4}} \\
    &\qquad = \pE \E [(u^\top (\tv\tv^\top - \tSigma) (\tw\tw^\top - \tSigma)u) (u^\top (\tw\tw^\top - \tSigma) (\tv\tv^\top - \tSigma)u))] \\
    &\qquad = \pE \E [(u^\top (\tv\tv^\top - \tSigma) (\tw\tw^\top - \tSigma)u)^2] \leq \pE \E [(u^\top (\tv\tv^\top - \tSigma)^2u) \cdot (u^\top (\tw\tw^\top - \tSigma)^2 u)] \\
    &\qquad \leq \frac{1}{2} \cdot \pE \E [(u^\top (\tv\tv^\top - \tSigma)^2u)^2 + (u^\top (\tw\tw^\top - \tSigma)^2u)^2] \leq 8 (\E [\norm{\tv}^8])^{1 / 2} (\pE \E [\iprod{u, \tv}^8])^{1 / 2} + 8\normt{\tSigma}^4.
\end{align*}
Finally, for the last term, we notice that:
\begin{equation*}
    \pE \Iprod{\E [((\tv\tv^\top - \tSigma) (\tw\tw^\top - \tSigma))^{\tensor 2})], u^{\tensor 4}} = \pE \E [(u^\top (\tv\tv^\top - \tSigma) (\tw\tw^\top - \tSigma)u)^2].
\end{equation*}
Finally, we have:
\begin{equation*}
    \E [\norm{\tv}^8] = \E [\norm{v}^8 \Ind[\norm{v} \leq \alpha]] \leq \E [\norm{v}^8] \text{ and } \normt{\tSigma} = \max_{\norm{u} = 1} \E [\iprod{u, v}^2 \Ind[\norm{v} \leq \alpha]] \leq \max_{\norm{u} = 1} \E [\iprod{u, v}^2] = \normt{\Sigma}.
\end{equation*}
Similarly, we also have from the fact that $\iprod{v, u}^8$ is a square polynomial:
\begin{equation*}
    \pE \E [\iprod{u, \tv}^8] = \pE \E [\iprod{u, v}^8 \Ind[\norm{v} \leq \alpha]] \leq \pE \E [\iprod{u, v}^8].
\end{equation*}
From \cref{lem:l8l2vec} and the certifiable hypercontractivity of $\mathcal{D}$, this concludes the proof of the lemma.
\end{proof}

\subsection{Proof of \cref{thm:sosmatbern}}
We now provide the proof of the SoS matrix Bernstein inequality.

\begin{proof}[Proof of \cref{thm:sosmatbern}]
We apply the standard matrix Bernstein inequality to the sum of random matrices $\sum_{i \leq k} M_i$ which has been preconditioned with the p.s.d. operator $A = \frac{1}{\sigma} \E[M^2] + \sigma \cdot I$, 
\begin{align*}
    \E[\normt{A^{-1/2} \sum_{i \leq k} M_i A^{-1/2}}] \leq \frac{1}{3} \log(2d^r) \max_i \normt{A^{-1/2} M_i A^{-1/2}} + \sqrt{2k \log(2d^r) \normt{\E[A^{-1/2} M A^{-1} M A^{-1/2}]}}.
\end{align*}
The fact that $A \succeq \sigma \cdot I \implies A^{-1} \preceq \frac{1}{\sigma} \cdot I \implies \norm{A^{-1/2}} \leq \frac{1}{\sqrt{\sigma}}$ and submultiplicativity of the operator norm show,
\begin{align*}
    \max_i \normt{A^{-1/2} M_i A^{-1/2}} \leq \max_i \normt{M_i} \normt{A^{-1/2}}^2 \leq  \frac{R}{\sigma}.
\end{align*}
We now bound the matrix variance term using the variational characterization of the operator norm and the Cauchy-Schwarz inequality,
\begin{align*}
    & \normt{\E[A^{-1/2} M A^{-1} M A^{-1/2}]} = \sup_{x : \norm{x}_2=1} \E[x^\top A^{-1/2} M A^{-1} M A^{-1/2} x] \leq \sup_{x : \norm{x}_2=1} \E[ \normt{x^\top A^{-1/2} M}   \normt{A^{-1}  M A^{-1/2}x}]  \\ & \leq \normt{A^{-1}} \sup_{x : \norm{x}_2=1} \E[ \normt{MA^{-1/2 }x}^2 ] \leq \normt{A^{-1}} \sup_{x : \norm{x}=1} x^\top A^{-1/2} \E[M^2] A^{-1/2} x \leq \frac{1}{\sigma} \normt{A^{-1/2} \E[M^2] A^{-1/2}}.
\end{align*}
However note by construction that $\E[M^2] \preceq \sigma A \implies A^{-1/2} \E[M^2] A^{-1/2} \preceq \sigma I$, so $\normt{A^{-1/2} \E[M^2] A^{-1/2}} \leq \sigma$. Thus $\normt{\E[A^{-1/2} M A^{-1} M A^{-1/2}]} \leq 1$. Assembling, we conclude,
\begin{align*}
    \E \left[\normt{A^{-1/2} \sum_{i \leq k} M_i A^{-1/2}} \right] \leq \frac{\log(2d^r)}{3} \frac{R}{\sigma} + \sqrt{2k \log(2d^r)}
\end{align*}

To conclude our final result, note that the following inequality holds deterministically in the p.s.d. order,
\begin{align*}
    A^{-1/2} \sum_{i \leq k} M_i A^{-1/2} \preceq \normt{A^{-1/2} \sum_{i \leq k} M_i A^{-1/2}} \cdot I \implies \sum_{i \leq k} M_i \preceq A \normt{A^{-1/2} \sum_{i \leq k} M_i A^{-1/2}} 
\end{align*}
Noting $A \preceq B \implies \langle A, u^{\otimes k} \rangle \preceq \langle B, u^{\otimes k} \rangle$ (where the first $\preceq$ is in the semi-definite ordering and the second $\preceq$ corresponds to an SoS proof) applying the pseudoexpectation and expectation operators gives,
\begin{align*}
    & \max_{\tE} \tE \langle u^{\otimes r}, \sum_{i \leq k} M_i u^{\otimes r} \rangle \leq \normt{A^{-1/2} \sum_{i \leq k} M_i A^{-1/2}} \cdot \max_{\tE} \tE \langle u^{\otimes r}, \sum_{i \leq k} A u^{\otimes r} \rangle \implies \\
    & \E[\max_{\tE} \tE \langle u^{\otimes r}, \sum_{i \leq k} M_i u^{\otimes r} \rangle] \leq \E[\normt{A^{-1/2} \sum_{i \leq k} M_i A^{-1/2}}] \cdot \max_{\tE} \tE \langle u^{\otimes r}, \sum_{i \leq k} A u^{\otimes r} \rangle \\
    & \leq \left(\frac{2 (\log(2) + r \log d))}{3} R + 2\sqrt{2k (\log(2) + r \log d)} \sigma \right) \max_{\tE} \tE[\normt{u}^{2k}]
\end{align*}
using the result of the previous matrix Bernstein bound and the SoS hypercontractivity assumption. Finally note $\max_{\tE} \tE[\normt{u}^{2k}]=1$ under the conditions of the theorem.
\end{proof}
% \begin{lemma}
%   \label[lemma]{lem:p-main-sharp}
%   Let $v_1,\ldots,v_n \sim \cD$ be i.i.d and assume the distribution $\cD$ satisfies \cref{def:nice-dist}.
%   Then for $p$ as in \eqref{eq:p-def},
%   \begin{align*}
%   & \E_{v_1,\ldots,v_n} \Brac{ \max_{\pE_u} \pE_u p(u) } \\
%   & \leq O\Paren{\frac 1 n \cdot k^{3/2} \cdot L \cdot \Tr \Sigma \cdot \Brac{\|\Sigma\|_2 + \sqrt{\frac k n} \cdot \|\Sigma\|_F} \cdot \sqrt{\log d} } + O \Paren{ \frac{k^2}{n} \cdot L \cdot \|\Sigma\|_2^2} + O \Paren{\tau^2 \cdot \log d}\mper
%   \end{align*}
%   where the maximum is over all degree-$4$ pseudodistributions in variables $u = (u_1,\ldots,u_d)$ satisfying $\{\|u\|^2 = 1\}$.
% \end{lemma}

\section{Covariance estimation}\label[section]{sec:covariance}
% Let $v_1,\ldots,v_n$ be i.i.d. samples from a distribution $v \sim \cD$ supported on $\R^d$ that is mean-zero ($\E[v_i]=0$). We aim to estimate the covariance of $\cD$, $\Sigma$ from the samples $v_i$. 

% These assumptions hold for a large class of interesting multivariate, heavy-tailed distributions such as products of 1-D heavy-tailed distributions, multivariate T distributions, etc.....

In this section, we will state and prove our result on estimating covariance matrices of $L$-nice distributions. Specifically, we will prove the following theorem:

\begin{theorem}
    \label{thm:covest}
    Given $\epsilon > 0$, $k>0$, the truncation parameter $\alpha = \left( \frac{L \cdot \normt{\Sigma}}{\Tr(\Sigma)} \cdot \frac{n}{\sqrt{k}} \right)^{1/4} \cdot \sqrt{\Tr(\Sigma)}$ and a sequence of $\{v_i\}_{i = 1}^n$ i.i.d. samples from a $O(1)$-nice distribution $\cD$ in $d$ dimensions (in the sense of \cref{def:nice-dist}), Algorithm~\ref{alg:covest} returns an estimate $\sigalg$ satisfying:
    \begin{gather*}
        \normt{\sigalg - \Sigma} \leq \cest \max(\covrate, \epsilon) \text{ where } \\
        r = O \left( \frac{(\log d)^{3/2} \normt{\Sigma}}{\sqrt{n}} \left( 
        % + \frac{\sr^{3/4} \sqrt{k}}{n^{1/4}} 
        k^{1/4} \sr^{1/2} + \sqrt{k} \right) \right)
    \end{gather*}
    with probability at least $1 - 2^{-\Omega(k)}$ where $\sr = \frac{\Tr(\Sigma)}{\normt{\Sigma}}$. Furthermore the run-time of Algorithm~\ref{alg:covest} is at most $\tO\Paren{\log(\frac{1}{\epsilon}) \cdot (d + k)^{17} n }$.
\end{theorem}

The proof of the theorem will proceed through two main steps:
\begin{enumerate}
    \item First, in Subsection~\ref{ssec:certpopcov} we show that there exists an efficient polynomial-time algorithm to certify if a candidate matrix is close to the true covariance (see \cref{lem:cert-main}). Our certification algorithm is based on a sum-of-squares relaxation of the polynomial optimization problem obtained via the analysis of the median-of-means tournament estimator for the same problem. While the polynomial optimization problem itself is intractable, crucially we show that a bounded degree sum-of-squares relaxation of the above problem suffices to perform such a test. 
    \item Subsequently in Subsection~\ref{ssec:gdcov}, we show that in the event that the above certification concludes that a candidate matrix is far from the true covariance matrix, such a certification algorithm also furnishes a suitable descent direction to improve our candidate estimate of the true covariance (see \cref{lem:covgdest}). We integrate this insight into a gradient-descent style algorithm for covariance matrix estimation (see \cref{alg:covest}).
\end{enumerate}

\subsection{Certifying the Population Covariance}
\label{ssec:certpopcov}
We begin by constructing a SDP that can efficiently test if a candidate matrix is close to the true covariance, using the samples $v_i$.

A key component of our high-probability certificate requires obtaining a sharp bound on the expected value of the SoS program. The expectation bound will rest heavily on our results from \cref{sec:random_poly} controlling the expected values of random degree-4 polynomials.
Concentration around the mean will follow from a bounded-differences inequality argument.

To bound the expected value of the SoS program, we let $\alpha > 0$ be a truncation level to be set in the sequel and
let $\tilde{v}_i = v_i \cdot \Ind[\|v_i\| \leq \alpha]$ be a truncated version of $v_i$.
Let $\tilde{\Sigma} = \E \tilde{v} \tilde{v}^\top$.
We will eventually choose $\alpha$ such that $\normt{\tilde{\Sigma} - \Sigma} \leq O(\covrate)$, so the bias from the truncation does not introduce too much error to our estimator. Fix $k \in \mathbb{N}$ (later we will take $k = \Theta(\log(1/\delta))$) and
let $B_1,\ldots,B_k \subseteq [n]$ partition $n$ evenly into $k$ buckets.
For $i \leq k$, let
\[
Z_i = \frac 1 m \sum_{j \in B_i} \tv_j\tv_j^\top - \tilde{\Sigma} \mper
\]
For convenience we will write $m = n/k$ (the number of samples in each bucket).

% Our goal is to find an SoS upper bound of $0.001$ on the following polynomial system for the smallest $r$ we can (this $r$ will be our error rate) \footnote{\Ynote{The language here seems a little confusing. Shouldn't we say something like Our goal is to prove an upper bound of 0.001 on the SoS relaxation of the following optimization problem? Also, it might make sense to move the comment about $r$ to after the program has been introduced. Shouldn't we also mention the precise degree of the SoS algorithm that we use here?}}:

Having partitioned our data into $k$ buckets, solutions to the following polynomial optimization problem can be used to certify with high probability over the draw of the data points that $\tSigma$ is the covariance matrix of the distribution up to a radius of $r$ in spectral norm.
\begin{align*}
  & \max_{b,u} \frac 1k \sum_{i=1}^k b_i \text{ such that}\\
  & b_i^2 = b_i \\
  & \|u\|^2 = 1 \\
  & b_i^4 \norm{u}^2 \iprod{Z_i, uu^\top} \geq r b_i\mper \tag{\textbf{Test-Cov}} \label{eq:tstco}
\end{align*}
Note that in our gradient descent algorithm, we actually use two nearly identical testing programs \cref{eq:tstcov} and \cref{eq:tstcovn}, which differ by a sign in one of the constraints in \cref{eq:tstco}. We defer discussion of this technicality until later.

As the above optimization problem is non-convex and it is unclear whether it can be solved in polynomial time, we work with its SoS relaxation instead. In particular, we will use the degree-$8$ SoS relaxation of the above problem. The reason we use the last constraint instead of the conceptually simpler constraint, $b_i \iprod{Z_i, uu^\top} \geq r b_i$, is that this avoids enforcing additional polynomial constraints of the form $q(x,b) (b_i \iprod{Z_i, uu^\top} - r b_i) \geq 0$ which would complicate subsequent analysis. Informally, the smallest value $r$ for which the value of the SoS program is small will determine our error rate. Notationally, we will refer to the value of this program using $\ref{eq:tstco}(r, \bZ)$ where $\bZ = \{ Z_i \}_{i=1}^k$.
The following is the key lemma for certifying the population covariance.

\begin{lemma}\label[lemma]{lem:cert-main} Assume that the $v_i$ are $L$-nice,
fix $k \in \mathbb{N}$ and
let $B_1,\ldots,B_k \subseteq [n]$ partition $n$ evenly into $k$ buckets. 
For $i \leq k$, let $Z_i = \frac k n \sum_{j \in B_i} \tv_j\tv_j^\top - \tilde{\Sigma}$ where $\tv_i = v_i \Ind[\normt{v_i} \leq \alpha]$.
If the truncation level $\alpha = \left( \frac{L \normt{\Sigma} n}{\Tr(\Sigma) \sqrt{k}} \right)^{1/4} \cdot \sqrt{\Tr(\Sigma)}$ and
  if $r \geq C_3 (\log d)^{3/2}  \left( \frac{L \normt{\Sigma}}{\sqrt{n}} \left(k^{1/4} \sr^{1/2} + \sqrt{k} \right) \right)$ for large-enough constant $C_3$ then,
  \[
  \max_{\pE} \frac 1k \pE \sum_{i=1}^k b_i \text{ s.t. $\pE$ satisfies \ref{eq:tstco}$( r, \bZ)$ and degree $\pE = 8$}
  \]
  is at most $0.001$ with probability at least $1-2^{-\Omega(k)}$.
\end{lemma}

To prove \cref{lem:cert-main}, we note, firstly, that the random variable we wish to bound satisfies a bounded differences inequality with respect to the random matrices, $Z_i$ (\cref{lem:sos-conc}). Therefore, it suffices to bound its expected value. That is, it suffices to control the following quantity:
\[
\E_v \max_{\pE} \frac 1k \pE \sum_{i=1}^k b_i \text{ s.t. $\pE$ satisfies above equations}
\]

Subsequently, the bulk of the proof of \cref{lem:cert-main} rests on assembling some supporting results to control the expected value of the aforementioned SoS program in \cref{eq:tstco}. We, first, establish that this program obeys a bounded-differences condition:

\begin{lemma} \label[lemma]{lem:sos-conc}
  Let $Y = (Y_1, \hdots, Y_k)$ by any set of $k$, $d$-dimensional p.s.d. matrices and let $Y' = (Y_1, \hdots, Y_i', \hdots, Y_k)$ be an identical set except with the $i$th matrix replaced with $Y_i'$. Now let $m$ denote the optimal value of \ref{eq:tstco}$(r, \bY)$ and $m'$ the optimal value of \ref{eq:tstco}$(r, \bY')$. Then $\abs{m-m'} \leq 1/k$.
\end{lemma}

\begin{proof}
    The argument proceeds similarly to an analogous argument in \cite{hopkins2018sub} and \cite{cherapanamjeri2019fast}. Let $\pE$ denote any degree-8, feasible pseudoexpectation for the \ref{eq:tstco} evaluated on $Y$. Then, define a new pseudoexpectation functional $\pE'$ which satisfies $\pE'[p(u,b_1,\ldots,b_{i-1},b_i,b_{i+1},\ldots,b_k)] = \pE[p(u,b_1,\ldots,b_{i-1},0,b_{i+1},\ldots,b_k)]$ for any degree-8 polynomial in $u, b_1, \hdots, b_k$. Such a $\pE'$ is a feasible pseudoexpectation for the latter program \ref{eq:tstco} evaluated on $\bm{Y}'$. Finally, $\frac{1}{k} \pE'[\sum_{j=1}^{k} b_j] = \frac{1}{k} \pE[\sum_{j=1}^{k} b_j] - \frac{1}{k}\pE[b_i] \geq \frac{1}{k} \pE[\sum_{j=1}^{k} b_j]-1/k$ since $0 \leq \pE[b_j] \leq 1, 0 \leq \pE'[b_j] \leq 1$ for all $j \in [k]$. Since the aforementioned argument holds for all feasible $\pE$ it follows that $m' \geq m-\frac{1}{k}$. A symmetric argument shows that $m \geq m'-1/k$.
\end{proof}

With this concentration result in hand, we now turn to controlling the expected value and begin by establishing a bound on the bias introduced by the $\alpha$-level truncation of the vectors $v_i$. Indeed, our truncation level $\alpha$ will have to be chosen sufficiently large so that this bias does not dominate our desired estimation accuracy. 
\begin{lemma}\label[lemma]{lem:bias} 
Assume the distribution $\cD$ satisfies \cref{def:nice-dist}, $v \sim \cD$ and $\tv = v \Ind[\normt{v} \leq \alpha]$. Then
 \begin{align*}
& \normt{\tilde{\Sigma}-\Sigma} \leq \frac{ L \normt{\Sigma} \Tr \Sigma}{\alpha^2}
\end{align*}
where $\tSigma = \E[\tv \tv^\top]$.
\end{lemma}
\begin{proof}
  For unit-norm $x$ we have that, 
\begin{align*}
      & \|\tilde{\Sigma}-\Sigma\| = \sup_{x} | \E \iprod{v, x}^2 -\E \iprod{\tv, x}^2 | = \sup_x | \E[(1-\Ind[\|v\| \leq \alpha]) \cdot \iprod{v, x}^2] | \\
      & \leq \sup_x (\E \iprod{v,x}^8)^{1/4} \cdot \sqrt{\Pr[\normt{v} > \alpha]} \leq \sqrt{L} \sup_x (\E \iprod{v, x}^2) \cdot \sqrt{ \frac{\sqrt{\E[\normt{v}^8]}}{\alpha^4}} \leq \frac{L \normt{\Sigma } \Tr(\Sigma)}{\alpha^2}
\end{align*}
where we have used the Cauchy-Schwarz/Jensen inequalities, Markov's inequality (in the form $\Pr[\normt{v}^4 \geq \alpha^4] \leq \frac{\E[\normt{v}^4]}{\alpha^4}$) and appealed to L8-L2 hypercontractivity. The result $\E[\normt{v}^8] \leq L^2 (\Tr \Sigma)^4$ follows from \cref{lem:l8l2vec}.
\end{proof}
Moreover we will also use another of level of truncation of the $Z_i's$ at level $\tau$ (appearing simply in the analysis of \cref{lem:cert-main} not in our algorithm). Appropriately choosing $\tau$ requires a bound on $\E[\normt{Z}]$,
\begin{lemma}\label[lemma]{lem:EZub}
Assume the distribution $\cD$ satisfies \cref{def:nice-dist} and that $v_i \sim \mathcal{D}$ i.i.d. Then, letting $Z = \frac{1}{m} \sum_{i \leq m} \tv_i \tv_i^\top - \tSigma$,
\begin{align*}
    \E[\normt{Z}] \leq O(\log d) \frac{\alpha^2}{m} + O \left(\sqrt{\log (d)} \sqrt{\frac{L \Tr \Sigma \normt{\Sigma}}{m}} \right)
\end{align*}
where $\tv = v \Ind[\normt{v} \leq \alpha]$ and $\tSigma = \E[\tv \tv^\top]$.
\end{lemma}

\begin{proof}
  This follows from an application of the matrix Bernstein inequality:
  \begin{align*}
      \E[\normt{Z}] \leq O(\log d) \cdot \frac{\alpha^2}{m} + O(\sqrt{\log d}) \normt{ \E Z^2 }^{1/2}.
  \end{align*}
  We can use \cref{lem:linalg1} along with  \cref{lem:l4l2norm} to bound the matrix variance as
  \begin{align*}
    & \normt{ \E Z^2 } \leq O \left(\Normt{ \frac{1}{m^2} \sum_{i \leq m} \E[\normt{\tv_i}^2 \tv_i \tv_i^\top] } \right) = \frac{1}{m} O(\normt{\E[\snormt{\tv} \tv \tv^\top}) \leq O(\frac{1}{m} L \Tr \Sigma \normt{\Sigma})
  \end{align*}
  Combining terms yields the result. Recall $\normt{\tSigma} \leq \normt{\Sigma}$ and $\Tr(\tSigma) \leq \Tr(\Sigma)$.
\end{proof}

With these supplementary lemmas we are finally prepared to embark on the proof of \cref{lem:cert-main}.

\begin{proof}[Proof of \cref{lem:cert-main}]

The primary technical challenge is bounding the expected value of the (random) SoS program in \ref{eq:tstco} since concentration around this expected value follows by a bounded differences argument. Accordingly we begin by bounding the expected value $\E_v \max_{\pE} \frac 1k \pE \sum_{i=1}^k b_i$.

To begin, simply for the purposes of our analysis, we partition the set of $k$ indices corresponding to buckets $B_1, \hdots, B_k$, into two sets, $\cG = \{i : \normt{Z_i} \leq \tau\}$ and $\cB = \{i : \normt{Z_i} > \tau\}$. Then, for any degree-8 $\pE$ satisfying the constraints in \ref{eq:tstco}, we have that
\begin{align*}
\E_v \max_{\pE} \frac 1k \pE \sum_{i=1}^k b_i \leq \E_v \max_{\pE} \frac 1k \pE \sum_{i \in \cG} b_i + \E_v \max_{\pE} \frac 1k \pE \sum_{i \in \cB} b_i.
\end{align*}
Now, since $\E \normt{Z} \leq  O(\log d) \frac{\alpha^2}{m} + O(\sqrt{\log d/m}) \sqrt{L \Tr \Sigma \normt{\Sigma}}$ (which follows from \cref{lem:EZub}), taking $\tau = C_1 \left( \log(d) \frac{\alpha^2}{m} + \sqrt{\log (d)} \sqrt{L \Tr \Sigma \normt{\Sigma}/m} \right)$ for some large-enough constant $C_1$, we can ensure that
\begin{align*}
    \E_v \max_{\pE} \frac 1k \pE \sum_{i \in \cB} b_i \leq \E_v \frac{1}{k} \sum_{i \leq k} \Ind[\normt{Z_i} \geq \tau] \leq \frac{1}{k} \frac{\sum_{i \leq k} \E \normt{Z_i}}{\tau} \leq .00025
\end{align*}
since $\pE[b_i] \leq 1$ and by appealing to Markov's inequality. For the remainder of the argument we fix this truncation level $\tau$. Now to control $\E_v \max_{\pE} \frac 1k \pE \sum_{i \in \cG} b_i$ we have that, 
\begin{align*}
    \E_v \max_{\pE} \frac 1k \pE \sum_{i \in \cG} b_i \leq \frac{1}{r} \E_v \max_{\pE} \frac 1k \pE \sum_{i \in \cG} b_i^4 \iprod{Z_i, uu^\top} = \frac{1}{r} \E_v \max_{\pE} \frac 1k \pE \sum_{i \in \cG} b_i \iprod{\tZ_i, uu^\top}
\end{align*}
using the constraints in \ref{eq:tstco} and defining $\tZ_i = Z_i \Ind[\normt{Z_i} \leq \tau]$. Now by an application of the SoS Cauchy-Schwarz inequality, we have
  \[
  \frac 1 k \E_v \max_{\pE} \pE \sum_{i \in \cG} b_i \iprod{\tZ_i, uu^\top} \leq \frac 1 {\sqrt k} \Paren{\E_v \max_{\pE} \pE \sum_{i=1}^k \iprod{\tZ_i \tensor \tZ_i, uu^\top \tensor uu^\top}}^{1/2}\mper
  \]
  Now using sub-additivity of the $\sqrt{\cdot}$,
  \begin{align*}
  & \frac 1 {\sqrt k}  \Paren{\E_v \max_{\pE} \pE \sum_{i=1}^k \iprod{\tZ_i \tensor \tZ_i, uu^\top \tensor uu^\top}}^{1/2} \leq \\
  & \frac{1}{\sqrt{k}}  \Paren{ O(\log(d)\tau^2) + O(\sqrt{k \log d}) \cdot \sqrt{O \left( \frac{1}{m^2} L^2 \normt{\Sigma}^2 \Tr(\Sigma)^2 \right)} }^{1/2} + \\
  & \frac{1}{\sqrt{k}} \left(O(\frac{L k}{m}) \cdot \normt{\Sigma}^2 \right)^{1/2}.
  \end{align*}
  for any degree-8 $\pE$ satisfying the conditions of the theorem by \cref{lem:p-main}.
  Note that $\tau$ is a function of the truncation level $\alpha$ on the vectors $v_i$. We now choose $\alpha$ appropriate to balance the bias from truncation and the previous upper bound. For convenience we reparametrize as $\alpha = \beta \sqrt{\Tr \Sigma}$. Recall by \cref{lem:bias} we have that our bias is $\normt{\tSigma-\Sigma} \leq \frac{L \normt{\Sigma } \Tr(\Sigma)}{\alpha^2} = \frac{L \normt{\Sigma}}{\beta^2} \equiv B(\beta)$. Similarly, after algebraic simplifications we obtain that aforementioned upper bound is at most,
  \begin{align*}
      & (\log d)^{3/2} \cdot O \Big( \sqrt{k} \frac{\Tr(\Sigma) \beta^2}{n} + \frac{1}{\sqrt{n}} \sqrt{L \Tr(\Sigma) \normt{\Sigma}} + k^{1/4} \frac{\sqrt{L \normt{\Sigma} \Tr(\Sigma)}}{\sqrt{n}} + \sqrt{\frac{kL}{n}} \normt{\Sigma} \Big).
  \end{align*}
  where we have used that $m=n/k$ and define $$R(\beta) = \sqrt{k} \frac{\Tr(\Sigma) \beta^2}{n} + \frac{1}{\sqrt{n}} \sqrt{L \Tr(\Sigma) \normt{\Sigma}} 
%   +
%   \sqrt{k} \cdot \frac{\sqrt{L \Tr(\Sigma) \norm{\Sigma}_F}}{n^{3/4}} 
  +
  k^{1/4} \frac{\sqrt{L \normt{\Sigma} \Tr(\Sigma)}}{\sqrt{n}} + \sqrt{\frac{kL}{n}} \normt{\Sigma}$$ for convenience.
  Since our final estimation error for $\Sigma$ will be upper bounded by $O((\log d)^{3/2} R(\beta)+B(\beta)) \leq C_2 (\log d)^{3/2} (R(\beta)+B(\beta))$ for some universal constant $C_2$, we choose $\beta$ to minimize $R(\beta)+B(\beta)$. A short computation choosing $\beta$ to balance the first term in $R(\beta)$ and $B(\beta)$ shows the optimal $\beta^* = \left (\frac{\normt{\Sigma} \cdot L }{\Tr(\Sigma)} \cdot \frac{n}{\sqrt{k}} \right)^{1/4}$. Hence we have the final value can be upper bounded, 
  \begin{align*}
    O \left( (\log d)^{3/2} \cdot \frac{\sqrt{L} \normt{\Sigma}}{\sqrt{n}} \left(
    % + \frac{\sr^{3/4} \sqrt{k}}{n^{1/4}} + 
    k^{1/4} \sr^{1/2} + \sqrt{k} \right) \right).
  \end{align*}
 By assembling this bound with our previous results and choosing $r \geq C_3 (\log d)^{3/2}  \left( \frac{\sqrt{L} \normt{\Sigma}}{\sqrt{n}} \left(k^{1/4} \sr^{1/2} + \sqrt{k} \right) \right)$, for large-enough constant $C_3$, we can guarantee that 
  \begin{align*}
  \E_v \max_{\pE} \frac 1k \pE \sum_{i=1}^k b_i \leq .0005
  \end{align*}
  where $\max_{\pE}$ is taken over any degree-8 pseudoexpectation $\pE$ satisfying the constraints in \ref{eq:tstco}. 
  
  Finally, since by \cref{lem:sos-conc}, \ref{eq:tstco} obeys a bounded-difference condition, simply applying the bounded differences inequality shows that under the previous conditions,
  \begin{align*}
      \max_{\pE} \frac 1k \pE \sum_{i=1}^k b_i \leq .001
  \end{align*}
  with probability at least $1-2^{-\Omega(k)}$.
\end{proof}

\subsection{Gradient Descent for Covariance Estimation}
\label{ssec:gdcov}

We now show how the certification algorithm described in the previous subsection can be leveraged to obtain an efficient algorithm for estimating covariance matrices. We start by proving a lemma which will help us relate a solution to the testing semidefinite program centered at an \textit{arbitrary} matrix, $x$, to a solution at the mean. To start with, we formally describe  the certification SDPs below:

\begin{gather*}
    \max_{\pE} \pE \left[\sum_{i = 1}^k b_i\right] \\
    \text{ such that $\pE$ satisfies: } b_i^2 = b_i \\
    \norm{u}^2 = 1 \\
    b_i^4 \norm{u}^2 \iprod{uu^\top, Z_i - x} \geq b_i r \tag{\textbf{Test-Cov-Pos}} \label{eq:tstcov}
\end{gather*}

In addition, to account for cases where the largest eigenvalue (in magnitude) of $\tSigma - x$ is negative, we will use the following semidefinite program:
% \Snote{singular values are always nonnegative. eigenvalue maybe?}
\begin{gather*}
    \max_{\pE} \pE \left[\sum_{i = 1}^k b_i\right] \\
    \text{ such that $\pE$ satisfies: } b_i^2 = b_i \\
    \norm{u}^2 = 1 \\
    -b_i^4 \norm{u}^2 \iprod{uu^\top, Z_i - x} \geq b_i r \tag{\textbf{Test-Cov-Neg}} \label{eq:tstcovn}
\end{gather*}

We will use $(\pE, v) = \ref{eq:tstcov}(\bm{Z}, X, r)$ to denote the optimal solution, value pair of the semidefinite program \ref{eq:tstcov} instantiated with $\bm{Z}, X$ and $r$ and analogous notation for \ref{eq:tstcovn}. Through the rest of this subsection, we will assume the following deterministic condition on the random matrices, $Z_i$.

\begin{condition}
    \label{cnd:covconc}
    Given $\bm{Z} = \{Z_i\}_{i = 1}^k$, for:
    \begin{equation*}
        \covrate = C_3 (\log d)^{3/2} \left( \frac{L \normt{\Sigma}}{\sqrt{n}} \left( 
        % + \frac{\sr^{3/4} \sqrt{k}}{n^{1/4}} 
        k^{1/4} \sr^{1/2} + \sqrt{k} \right) \right)
    \end{equation*}
    for large-enough $C_3$,
    the solutions $(\pE_p, v_p) = \ref{eq:tstcov} (\bm{Z}, \tSigma, \covrate)$ and $(\pE_n, v_n) = \ref{eq:tstcovn} (\bm{Z}, \tSigma, \covrate)$ satisfy, $v_p, v_n \leq 0.001k$.
\end{condition}

We will also make use of the fact that due to the pseudoexpectations satisfying $b_i^2 = b_i$ and $\norm{u}^2 = 1$, the last constraints in \ref{eq:tstcov} and \ref{eq:tstcovn} are equivalent to:

\begin{equation*}
    \pE[b_i \iprod{uu^\top, Z_i - x}] \geq \pE [b_i r] \text{ and } \pE[-b_i \iprod{uu^\top, Z_i - x}] \geq \pE[b_i r] \text{ respectively.}
\end{equation*}

\begin{lemma}
  \label{lem:sdpxtm}
  Assume Condition~\ref{cnd:covconc}. Let $\pE$ be a pseudo-distribution over variables $b_i$ and $v_j$ satisfying $\norm{v}^2 = 1$ and $b_i^2 = b_i$. Suppose further that $\pE$ satisfies $\pE[\sum_{i = 1}^k b_i] \geq 0.999k$. Then, there are sets of $0.998k$ indices $\mathcal{S}_p$ and $\mathcal{S}_n$ such that for all $l_p \in \mathcal{S}_p$ and $l_n \in \mathcal{S}_n$, we have:
  \begin{equation*}
      \pE[b_{l_p}\iprod{uu^\top, Z_{l_p} - \tSigma}] \leq \pE[b_{l_p}] \covrate \text{ and } \pE[-b_{l_n}\iprod{uu^\top, Z_{l_n} - \tSigma}] \leq \pE[b_{l_n}] \covrate
  \end{equation*}
  Furthermore, there is a subset $\mathcal{R} \subseteq [k]$ such that $\abs{\mathcal{R}} \geq 0.98k$ for any $i \in \mathcal{R}$, we have $\pE[b_i] \geq 0.95$.
\end{lemma}
\begin{proof}
  Under Condition~\ref{cnd:covconc},  the optimal value of \ref{eq:tstcov} at the true mean $\tSigma$ and the radius set to $\covrate$ is at most $\frac{k}{1000}$. Furthermore, the only constraints of \ref{eq:tstcov}$(\bm{Z}, \tSigma, \covrate)$ that are violated by $\pE$ are constraints which involve the polynomial equation $b_i^4 \normt{u}^2 \iprod{uu^\top, Z_i - \tSigma} \geq b_i \covrate$. However, note that since the polynomial constraint is of degree $8$ and we are optimizing over degree-$8$ pseudoexpectations, the only constraints enforced by this inequality are constraints of the form:
  \begin{equation*}
      \pE [b_l^4 \normt{u}^2 \iprod{uu^\top, Z_l - \tSigma}] \geq \pE[b_l] \covrate
  \end{equation*}
  Let the set of indices which violate the above inequality be denoted by $\mathcal{S}$. By setting to $0$ the $b_l$ corresponding to all the indices in $\mathcal{S}$, we obtain a feasible solution for \ref{eq:tstcov} $(\tZ, \tSigma, \covrate)$. However, note that setting a particular $b_l$ to $0$, only decreases the value of $\pE[\sum_{i = 1}^k b_i]$ by at most $1$. Since, we reduce the value of this quantity by at least $0.998k$, we conclude that the size of $\mathcal{S}$ is at least $0.998k$. Analogous results for \ref{eq:tstcovn} prove the existence of $\mathcal{S}_n$.
  
  For the second claim of the lemma, let $\mathcal{R} = \{i: \pE[b_i] \geq 0.95\}$. We have:
  \begin{equation*}
    0.999k \leq \pE[\sum_{i = 1}^k b_i] = \pE[\sum_{i \in \mathcal{R}} b_i] + \pE[\sum_{i \notin \mathcal{R}} b_i] \leq \abs{\mathcal{R}} + 0.95 (k - \abs{\mathcal{R}}) \implies \abs{\mathcal{R}} \geq 0.98k
  \end{equation*}
\end{proof}

In the next lemma, we show that for any point, $x$, we will able to accurately estimate the distance from $x$ to the mean.

\begin{lemma}
  \label{lem:dest}
  Assume Condition~\ref{cnd:covconc}. Let $x \in \mathbb{R}^{d \times d}$ be a symmetric matrix. Then, we have the following guarantee for the distance estimation step run on $x$ where $d_x = \distest(Z, x)$:
  
  \begin{equation*}
      \abs{d_x - \normt{x - \tSigma}} \leq \max(\cdest \covrate, 1 / 4 \normt{x - \tSigma})
  \end{equation*}
\end{lemma}

\begin{proof}
  We first consider the case where $\normt{\tSigma - x} \geq \cdest \covrate$. For the lower bound of the distance estimation step, consider the singular vector $u$ corresponding to the largest singular value of $\tSigma - x$. We will assume without loss of generality that the eigenvalue corresponding to this eigenvector is positive. Under Condition~\ref{cnd:covconc}, we know that for at least $0.999k$ of the $Z_i$, we have $\iprod{Z_i - \tSigma, -uu^\top} \leq \covrate$ as otherwise, one can construct a feasible solution for \ref{eq:tstcov}$(\bm{Z}, \tSigma, \covrate)$ with value at least $0.001k$. For any $i$ satisfying the previous property:
  
  \begin{equation*}
      \iprod{Z_i - x, uu^\top} = \iprod{Z_i - \tSigma, uu^\top} + \iprod{\tSigma - x, uu^\top} \implies \iprod{Z_i - x, uu^\top} - \normt{\tSigma - x} \geq -\covrate
  \end{equation*}
  
  This proves the lower bound in the lemma. For the upper bound, suppose that $\pE$ is a solution which attains the optimal value of \ref{eq:tstcov}$(\bm{Z}, x, 1.25\norm{\tSigma - x})$ greater than $0.999k$. Note that for any $Z_i$, we have the following:
  \begin{align*}
      \pE [b_i\iprod{Z_i - x, uu^\top}] &= \pE [b_i\iprod{Z_i - \tSigma, uu^\top}] + \pE [b_i\iprod{\tSigma - x, uu^\top}] \\
      &\leq \pE [b_i\iprod{Z_i - \tSigma, uu^\top}] + \Paren{\pE \Brac{b_i^2}}^{1 / 2} \Paren{\pE \Brac{\iprod{uu^\top, \tSigma - x}^2}}^{1 / 2} \\
      &\leq \pE [b_i\iprod{Z_i - \tSigma, uu^\top}] + \normt{\tSigma - x}
  \end{align*}
  
  where the first inequality follows from SoS Cauchy-Schwarz. We have from the second claim of Lemma~\ref{lem:sdpxtm} that there is a subset of at least $0.98k$ elements, $\mathcal{R}$, such that for all $i \in \mathcal{R}$:
  \begin{equation*}
      \pE \Brac{b_i \iprod{Z_i - x, uu^\top}} \geq \pE[b_i] (1.25\normt{\tSigma - x}) \geq 1.0625 \normt{\tSigma - x}
  \end{equation*}
  Therefore, we have from the previous two equations that for the $i \in \mathcal{R}$:
  \begin{equation*}
      \pE [b_i\iprod{Z_i - \tSigma, uu^\top}] \geq 0.0625 \normt{\tSigma - x} > \covrate
  \end{equation*}
  By constructing a pseudo-expectation, $\pE'$ such that $\pE'$ is identical to $\pE$ for monomials not involving the $b_i$ for $i \notin \mathcal{R}$ and $0$ otherwise, we obtain a feasible solution for \ref{eq:tstcov}$(\bm{Z}, \tSigma, \covrate)$ with objective value at least $0.95\abs{\mathcal{R}} \geq 0.5k$ as for each $i \in \mathcal{R}$, we have $\pE[b_i] \geq 0.95$. This is a contradiction to Condition~\ref{cnd:covconc}. A similar proof for \ref{eq:tstcovn} proves the upper bound in the case where $\normt{\tSigma - x} \geq \cdest \covrate$.
  
  Now, consider the alternate case where $\normt{\tSigma - x} \leq \cdest \covrate$. In this case, the lower bound is trivially true. For the upper bound, we proceed similarly to the upper bound for the previous case. Let $\pE$ be a solution to \ref{eq:tstcov}$(\tZ, x, \normt{\tSigma - x} + \cdest\covrate)$ which obtains objective value at least $0.999k$. We define the set, $\mathcal{R}$ similar to the previous case and we have for all $i \in \mathcal{R}$:
  \begin{equation*}
      \pE \Brac{b_i \iprod{Z_i - x, uu^\top}} \geq 0.95 (\normt{\tSigma - x} + \cdest\covrate) \implies \pE \Brac{b_i \iprod{Z_i - \tSigma, uu^\top}} \geq 19\covrate - 0.05\normt{\tSigma - x} \geq 18\covrate
  \end{equation*}
  As before, we may construct as before a new pseudo-expectation, $\pE'$, which is identical to $\pE$ on polynomials not involving $b_i$ for $i \notin \mathcal{R}$ and $0$ otherwise. $\pE'$ is a feasible solution for \ref{eq:tstcov}$(\bm{Z}, \tSigma, 18\covrate)$ with optimal value at least $0.5k$ which is a contradiction to Condition~\ref{cnd:covconc}. A similar proof for \ref{eq:tstcovn} proves the upper bound in the case where $\normt{\tSigma - x} \leq \cdest \covrate$. This concludes the proof of the lemma.
\end{proof}

In the next lemma, we show that we may accurately estimate a gradient from the solution of the testing problem. In the following, we use $\normo{X}$ to denote the trace norm of the matrix $X$ which the sum of the singular values of $X$.

\begin{lemma}
  \label[lemma]{lem:covgdest}
  Assume Condition~\ref{cnd:covconc}.
  Let $x \in \mathbb{R}^{d\times d}$ be a symmetric matrix satisfying $\normt{\tSigma - x} \geq \cgd \covrate$. The matrix $G = \gest(\tZ, x)$ satisfies:
  \begin{equation*}
      \normo{G} = 1, \qquad \qquad \iprod{G, \tSigma - x} \geq 0.5 \normt{\tSigma - x}.
  \end{equation*}
\end{lemma}

\begin{proof}
  Let $d_x = \distest(\tZ, x)$ and that $\pE$ is the solution to \ref{eq:tstcov}$(\tZ, x, d_x)$ or \ref{eq:tstcovn}$(\tZ, x, d_x)$ satisfying $\pE [\sum b_i] \geq 0.999k$. Without loss of generality assume that $\pE$ is the solution to \ref{eq:tstcov} as the proof for the alternate case is similar. In this case, note that $G = \pE [uu^\top]$. Furthermore, from Lemma~\ref{lem:sdpxtm}, we have that there exists a set $\mathcal{S}$ of size at least $0.998k$ indices such that for all $i \in \mathcal{S}$, we have:
  \begin{equation*}
      \pE \Brac{b_i \iprod{uu^\top, \tZ_i - \tSigma}} \leq \pE [b_i] \covrate
  \end{equation*}
  In addition, there exists a set $\mathcal{R}$ of size at least $0.98k$ such that for all $i \in \mathcal{R}$, we have:
  \begin{equation*}
      \pE \Brac{b_i \iprod{uu^\top, \tZ_i - x}} \geq \pE [b_i] d_x \geq 0.95 d_x \geq 0.95 \cdot 0.75 \normt{\tSigma - x} \geq 0.7125 \normt{\tSigma - x}
  \end{equation*}
  Now, consider the set $\mathcal{T} = \mathcal{R} \cap \mathcal{S}$. Note that the size of $\mathcal{T}$ is at least $0.975k$. We now get the following inequality:
  \begin{align*}
      0.695k \norm{\tSigma - x} &\leq 0.7125 \abs{\mathcal{T}} \norm{\tSigma - x} \leq \pE \Brac{\sum_{i \in \mathcal{T}} b_i \iprod{uu^\top, \tZ_i - x}} \\ 
      &= \pE \Brac{\sum_{i \in \mathcal{T}} b_i \iprod{\tZ_i - \tSigma, uu^\top}} + \pE \Brac{\sum_{i \in \mathcal{T}} b_i \iprod{\tSigma - x, uu^\top}} \\
      &\leq \sum_{i \in \mathcal{T}} \pE[b_i]\covrate + k\iprod{\tSigma - x, \pE [uu^\top]} -\Iprod{\tSigma - x, \pE \Brac{\Paren{k - \sum_{i \in \mathcal{T}} b_i} uu^\top}} \\
      &\leq k\covrate + k\iprod{\tSigma - x, \pE uu^\top} + \normt{\tSigma - x}\Normo{\pE \Brac{\Paren{k - \sum_{i \in \mathcal{T}} b_i} uu^\top}}
  \end{align*}
  where the last inequality follows by an application of the matrix-\Holder inequality. Note that the matrix $\pE \Brac{(k - \sum_{i \in \mathcal{T}}) uu^\top}$ is positive semidefinite as for all all $v\in \mathbb{R}^{d}$, we have:
  \begin{equation*}
      v^\top \pE \Brac{\Paren{k - \sum_{i \in \mathcal{T}}} uu^\top} v = \pE \Brac{\Paren{k - \sum_{i \in \mathcal{T}} b_i} \iprod{u, v}^2} \geq 0
  \end{equation*}
  where the last inequality follows because $\pE$ satisfies the polynomial inequality $b_i \leq 1$ and therefore the inequality $(k - \sum_{i \in \mathcal{T}} b_i) \geq 0$. We bound the second term in the above inequality as follows:
  
  \begin{equation*}
      \Tr \pE \Brac{\Paren{k - \sum_{i \in \mathcal{T}} b_i} uu^\top} = \pE \Brac{\Paren{k - \sum_{i \in \mathcal{T}} b_i} \Tr uu^\top} = \pE \Brac{k - \sum_{i \in \mathcal{T}} b_i} \leq k - \abs{\mathcal{T}} \cdot 0.95 \leq 0.1k
  \end{equation*}
  By substituting the above bound, we get:
  \begin{equation*}
      0.695k \normt{\tSigma - x} \leq kr + k \iprod{\tSigma - x, \pE uu^\top} + 0.1k \normt{\tSigma - x}
  \end{equation*}
  
  Noting that $\norm{\tSigma - x} \geq \cgd \covrate$, we get:
  \begin{equation*}
      \iprod{\tSigma - x, \pE uu^\top} \geq 0.5 \normt{\tSigma - x}
  \end{equation*}
  Finally, note that $\pE[uu^\top]$ is also a psd matrix and therefore, we have that:
  \begin{equation*}
      \normo{\pE [uu^\top]} = \Tr \pE [uu^\top] = 1
  \end{equation*}
\end{proof}

We now conclude that our gradient descent algorithm returns a good solution.

\begin{proof}[Proof of Theorem~\ref{thm:covest}]
  Let $\Sigma_{t}$ be the sequence of iterates obtained in the algorithm. We define the set, $\mathcal{G} = \{X \in \mathbb{R}^{d\times d}: \normt{X - \tSigma} \leq \cgd \covrate\}$. We will prove the theorem under two cases:
  \paragraph{Case 1:} One of the iterates belongs to the set, $\mathcal{G}$. Suppose that $\Sigma_t$ be an iterate in $\mathcal{G}$. Therefore, we have from Lemma~\ref{lem:dest} applied to the iterate, $\Sigma_t$:
  \begin{equation*}
    \dalg \leq d_t \leq 120 \covrate
  \end{equation*}
  Finally, if $\sigalg$ already belongs to the set $\mathcal{G}$, we are already done. Otherwise, via an application of Lemma~\ref{lem:dest} to $\sigalg$, we get:
  \begin{equation*}
      \normt{\sigalg - \tSigma} \leq \frac{1}{0.75} \dalg \leq 160 \covrate
  \end{equation*}
  This proves the lemma in this case.
  \paragraph{Case 2:} In the alternate case where none of the iterates belong $\mathcal{G}$, we have via the following inequality and Lemmas~\ref{lem:covgdest} and \ref{lem:dest}:
  \begin{align*}
      \frob{\Sigma_{t + 1} - \tSigma}^2 &= \Frob{\Sigma_{t} - \frac{d_t}{4} G_t - \tSigma}^2 = \frob{\Sigma_{t} - \tSigma}^2 - 2 \frac{d_t}{4} \Iprod{\Sigma_t - \tSigma, G_t} + \frac{d_t^2}{16} \frob{G_t}^2 \\
      &\leq \frob{\Sigma_{t} - \tSigma}^2 - \frac{d_t}{4} \normt{\Sigma_t - \tSigma} + \frac{d_t^2}{16} \leq \frob{\Sigma_{t} - \tSigma}^2 - \frac{3}{16} \normt{\Sigma_t - \tSigma}^2 + \frac{25}{256} \normt{\Sigma_t - \tSigma}^2 \\
      &\leq \frob{\Sigma_{t} - \tSigma}^2 - \frac{1}{16} \normt{\Sigma_t - \tSigma}^2 \leq \frob{\Sigma_{t} - \tSigma}^2 - \frac{1}{16d} \frob{\Sigma_t - \tSigma}^2 \\
      &= \Paren{1 - \frac{1}{16d}} \frob{\Sigma_t - \tSigma}^2 \leq e^{-\frac{T}{16d}}\frob{\Sigma_0 - \tSigma}^2
  \end{align*}
  where we have used the fact that $\frob{G_t} \leq \normo{G_t} = 1$ and the fact that $\normt{X} \geq \frac{1}{\sqrt{d}} \frob{X}$. The accuracy of our algorithm follows from recursively applying the above inequality. 
  
  By applying Lemma~\ref{lem:cert-main} to the the semidefinite optimization problems, \ref{eq:tstcov} and $\ref{eq:tstcovn}$, we see that Condition~\ref{cnd:covconc} holds with probability at least $1 - 2^{-\Omega(k)}$. Note that technically the output of \cref{alg:covest}, $\Sigma^*$, is an estimate of the "truncated" second-moment matrix $\tSigma$. However as noted in the proof of \cref{lem:cert-main}, the truncation parameter $\alpha$ is chosen to precisely balance the bias term (due to truncation) with the estimation error, so the overall error $\normt{\Sigma^*-\Sigma} \leq \normt{\Sigma^*-\tSigma} + \normt{\tSigma-\Sigma}$ achieves the stated convergence rate.
  
  Our algorithm consists of $O(d\log 1 / \epsilon)$ iterations where in each iteration, we solve an semidefinite optimization problem with $O((d + k)^8)$ variables. Assuming standard runtimes of the Ellipsoid algorithm, the total runtime of our algorithm can be upper bounded by $O((d + k)^{17}\log 1 / \epsilon)$. This concludes the proof of the theorem.
  
\end{proof}
\begin{algorithm}[H]
\caption{Distance Estimation}
\label{alg:covdest}
\begin{algorithmic}[1]
\State \textbf{Input: } Set of sample covariance matrices, $\bm{Z} = \{Z_i\}_{i = 1}^k$, Current estimate $X$
\State $d^* \gets \sup \{r > 0: \ref{eq:tstcov} (\bm{Z}, X, r) \geq 0.999k \text{ or } \ref{eq:tstcovn} (\bm{Z}, X, r) \geq 0.999k\} $
\State \textbf{Return: } $d^*$
\end{algorithmic}
\end{algorithm}

\begin{algorithm}[H]
\caption{Gradient Estimation}
\label{alg:covgest}
\begin{algorithmic}[1]
\State \textbf{Input: } Set of sample covariance matrices, $\bm{Z} = \{Z_i\}_{i = 1}^k$, Current estimate $X$
\State $d^* \gets \distest (\bm{Z}, X)$
\State $\pE_p \gets \text{\ref{eq:tstcov}} (\bm{Z}, X, d^*),\ v_p \gets \pE_p[\sum_{i = 1}^k b_i]$
\State $\pE_n \gets \text{\ref{eq:tstcovn}} (\bm{Z}, X, d^*),\ v_n \gets \pE_n[\sum_{i = 1}^k b_i]$
\If {$v_p \geq 0.999k$}
\State $G \gets \pE_p[uu^\top]$
\Else 
\State $G \gets -\pE_n[uu^\top]$
\EndIf
\State \textbf{Return: } $G$
\end{algorithmic}
\end{algorithm}
\begin{algorithm}[H]
\caption{Estimate Covariance}
\label{alg:covest}
\begin{algorithmic}[1]
\State \textbf{Input: } Set of sample points, $\{v_i\}_{i = 1}^n$, Error Tolerance $\epsilon$, Success Probability $\delta$
\State $nit \gets \nit$
\State $k \gets \nbuck$
\State $\tv_i \gets v_i \bm{1} \{\norm{v_i} \leq \alpha\}$
\State Split data into $k$ buckets, $\mathcal{B}_j = \left\{\tv_{\frac{(j-1)n}{k} + 1}, \dots, \tv_{\frac{jn}{k}}\right\}$ for $j = 1, \cdots k$
\State $Z_j \gets \frac{k}{n} \sum_{\tv \in \mathcal{B}_j} \tv\tv^\top$
\State $\bm{Z} = \{Z_1, \dots, Z_k\}$
\State $\Sigma_0 \gets 0, \Sigma^* \gets 0, d^* \gets \infty$
\For {$t = 0:nit$}
    \State $d_{t} \gets \distest (\bm{Z}, \Sigma_t)$
    \If {$d_t \leq d^*$}
        \State $d^* \gets d_t,\ \Sigma^* \gets \Sigma_t$
    \EndIf
    \State $G_t \gets \gest (\bm{Z}, \Sigma_t)$
    \State $\Sigma_{t + 1} \gets \Sigma_t - \frac{d_t}{4} G_t$
\EndFor
\State \textbf{Return: }$\Sigma^*$
\end{algorithmic}
\end{algorithm}

\section{Regression}

In this section we prove the following main theorem.

\begin{theorem}\label[theorem]{thm:regression-main}
  There is a polynomial-time algorithm and a universal constant $C > 0$ with the following guarantees.
  For any $d$-dimensional $O(1)$-nice random variable $X$ with $\E X = 0$ and $\E XX^\top = \Id$ and any $\R$-valued random variable $\e$ with $\E \e = 0$ and $\E \e^2 = 1$ and any linear function $f^* \, : \, \R^d \rightarrow \R$, given $n$ i.i.d. samples $(X_1,f^*(X_1) + \e_1),\ldots,(X_n,f^*(X_n) + \e_n)$ the algorithm produces $\hat{f}$ such that with probability at least $1-\delta$ it holds that 
  \[
  \E_X (f(X) - f^*(X))^2 \leq C \cdot \Paren{ \frac d n + \frac {\log(1/\delta)}{n} }
  \]
  so long as $n \geq \max(d (\log(1/\delta))^{1/2} \cdot (\log d)^{C}, C \log(1/\delta))$.
\end{theorem}

We will prove \cref{thm:regression-main} from main lemmas in the next subsection, but before we do so we need to set up some notation.

\paragraph{Notation}
Suppose that $S \subseteq \R^d$ is a finite set of vectors.
They induce an inner product on functions $f \, : \, \R^d \rightarrow \R$ by $\iprod{f,g}_S = \E_{x \sim S} f(x) g(x)$.
We also write $\|f\|_S = \iprod{f,f}_S^{1/2}$.

Similarly, if $(X,Y)$ is a random variable which is clear from context and $f,g$ are functions of $X,Y$ then we let $\iprod{f,g} = \E_{X,Y} f(X,Y)g(X,Y)$ and similarly for $\|f\|$.
If $u \in \R^d$ we denote by $\|u\|$ its Euclidean norm; since in this section the random variable $X$ will always have $\E X = 0$ and $\E XX^\top = \Id$ this is the same as $\|f\|$ for the linear function $f(X,Y) = \iprod{u,X}$.

We work with the square-loss function.
If $X$ is a random vector which is clear from context, and $Y$ is an $\R$-valued random variable, then $L(f) = \E_X (f(X) - Y)^2 = \|f(X) - Y\|^2$.
If $S = \{(X_1,Y_1),\ldots,(X_m,Y_m)\}$, we denote the empirical loss on $S$ by $L_S(f) = \E_{(X,Y) \sim S} (f(X) - Y)^2 = \|f(X) - Y\|_S^2$.

We often have a set $\{(X_1,Y_1),\ldots,(X_n,Y_n)\}$ which we split into $k$ buckets $B_1,\ldots,B_k$ of equal size.
In this case, we shorten the notation $\iprod{f,g}_{B_i}$ to $\iprod{f,g}_i$ and similarly for $\|f\|_i$ and $L_i(f)$.

\subsection{Proof of \cref{thm:regression-main}}

\newcommand{\cAloss}{\cA_{\text{bucket-loss}}}
\newcommand{\cAnormupper}{\cA_{\text{norm-upper}}}
\newcommand{\cAnormlower}{\cA_{\text{norm-lower}}}
\newcommand{\cAnoise}{\cA_{\text{noise}}}

To set up for the proof of \cref{thm:regression-main}, we need to describe the key certifiability properties that our main algorithm exploits.
%We capture several properties of a collection of samples in the following polynomial systems.
Throughout, let $X_1,\ldots,X_n \in \R^d$ and $Y_1,\ldots,Y_n \in \R$.
In the background of all the definitions which follow there is a fixed partition $B_1,\ldots,B_k$ of $[n]$ into $k$ equal parts.
The polynomial systems defining the SoS SDPs are all in variables $f = (f_1,\ldots,f_d)$ which represents a linear function on $\R^d$ by its coefficients and $b_1,\ldots,b_k$ which we think of as $0/1$-indicators corresponding to the buckets $B_1,\ldots,B_k$.

\begin{definition}[Noise correlation SDP]
    For a number $r > 0$ and a linear function $g$, let $\cAnoise$ be the polynomial system
    \begin{align*}
        b_i \iprod{Y - g, f}_i & \geq r b_i \text{ for } i \in [k]\\
        \|f\|^2 & = 1 \\
        b_i^2 & = b_i \text{ for } i \in [k]\mper
    \end{align*}
    A feasible solution $(b,f)$ is a linear function $f$ together with the indicator of a subset of the buckets on which its empirical correlation with $Y - g$ is at least $r$.
    We define the noise correlation SDP (parameterized by $g,r$) as
    \[
    \max \pE \sum_{i \leq k} b_i \text{ such that } \deg \pE = 4, \pE \text{ satisfies } \cAnoise
    \]
\end{definition}

We will also employ the following semidefinite programs, which certify bounds on the deviations between per-bucket norms and Euclidean norms.

\begin{definition}[Norm upper bound SDP]
  For $C > 0$ we define the following SDP over degree-$4$ pseudodistributions in variables $b_1,\ldots,b_k,f_1,\ldots,f_d$.
  \begin{align*}
  & \max \pE \sum_{i \leq k} b_i \|f\|^2 \text{ such that }\\
  & \pE \|f\|^4 \leq 1 \\
  & \pE \text{ satisfies } b_i^2 = b_i \\
  & \pE b_i \|f\|_i^2  \geq C \pE b_i \|f\|^2\\
  \end{align*} 
\end{definition}

\begin{definition}[Norm lower bound SDP]
  For $c > 0$ we define the following SDP over degree-$4$ pseudodistributions in variables $b_1,\ldots,b_k,f_1,\ldots,f_d$.
  \begin{align*}
  & \max \pE \sum_{i \leq k} b_i \|f\|^2 \text{ such that }\\
  & \pE \|f\|^4 \leq 1 \\
  & \pE \text{ satisfies } b_i^2 = b_i \\
  & \pE b_i \|f\|_i^2  \leq c \pE b_i \|f\|^2\\
  \end{align*} 
\end{definition}

Our main algorithm will succeed under the following deterministic condition.

\newcommand{\ttX}{\tilde{X}}
\newcommand{\tY}{\tilde{Y}}

\begin{definition}[Regression Deterministic Conditions]\label[definition]{def:reg-det}
Let $X,f^*,\e$ be as in \cref{thm:regression-main} and let $Y_i = f^*(X_i) + \e_i$. 
Let $(\ttX_i,\tY_i) = (X_i,Y_i) \cdot \Ind(\|X\| \leq \alpha)$, where $\alpha = C_0 \sqrt{d}$ for a large-enough constant $C_0$.
Let $k = \Theta(\log(1/\delta))$ and let $B_1,\ldots,B_k$ be a fixed partition of $[n]$ into $k$ buckets as usual.
For some $r^2 = O(d/n + \log(1/\delta)/n)$, our deterministic conditions are
\begin{align*}
%  & \cAloss \proves_4 \sum_{i \leq k} b_i \leq 0.001k   \tag{\text{loss}} \label{eq:loss}\\
  & \text{noise correlation SDP}\leq 0.001k \tag{\text{noise}} \label{eq:noise}\\
  & \text{norm upper bound SDP} \leq 0.001k \tag{\text{norm-upper}} \label{eq:norm-upper}\\
  & \text{norm lower bound SDP} \leq 0.001k \tag{\text{norm-lower}} \label{eq:norm-lower}\mper
%   & \cAnormupper \proves_4 \sum_{i \leq k} b_i \|f\|^2 \leq 0.001k \tag{\text{norm-upper}} \label{eq:norm-upper}\\
%   & \max_{\substack{\deg \pE = 4 \\ \pE \|f\|^2 = 1 \\ \pE b_i \|f\|_i^2 \geq 1.01 \pE b_i \|f\|^2 \\ \pE \text{ satisfies } b_i^2 = b_i }} \pE \sum_{i \leq k} b_i \leq 0.001k \tag{\text{norm-lower}} \label{eq:norm-lower}\mper
%   & \cAnormlower \proves_4 \sum_{i \leq k} b_i \leq 0.001k \tag{\tex{norm-lower}} \label{eq:norm-lower}\mper
\end{align*}
where the SDPs are instantiated with $\tilde{X_i},\tilde{Y_i}, r,f^*$ and $C = 1.01, c = 0.99$ and truncated samples $(\ttX_i,\tY_i)$.
\end{definition}

We will prove \cref{thm:regression-main} from two main lemmas.
The first says that the deterministic conditions above hold with high probability.

\begin{lemma}\label[lemma]{lem:regression-deterministic}
  Let $X,f^*,\e,\delta$ be as in \cref{thm:regression-main}.
  For $k = C \log(1/\delta)$ for a large-enough constant $C$, let $B_1,\ldots,B_k$ partition $[n]$ into equal-size parts.
  Let $f^*$ be a linear function, and let $r^2 = C' (d/n + \log(1/\delta)/n)$ for some universal $C' > 0$.
  Let $X_1,\ldots,X_n,\e_1,\ldots,\e_n$ be i.i.d. copies of $X,\e$ respectively and let $Y_i = f^*(X_i) + \e_i$.
  Suppose $n \geq  \log(1/\delta)^{1/2} d (\log d)^{C''}$ and $\delta \geq 2^{-n/C''}$, for a large-enough constant $C''$.
  Then \cref{eq:noise,eq:norm-upper,eq:norm-lower} all hold with probability at least $1-\delta$.
\end{lemma}

The second lemma says that when the deterministic conditions are satisfied it is possible to estimate $f^*$ in polynomial time.

\begin{lemma}\label[lemma]{lem:regression-descent-main}
  There is a polynomial-time algorithm with the following guarantees.
  Given $(X_1,Y_1),\ldots,(X_n,Y_n)$ and a partition $B_1,\ldots,B_k$ of $[n]$ into $k$ buckets and $r$ for which \cref{eq:noise,eq:norm-upper,eq:norm-lower} all hold for some linear function $f^*$, and given some linear function $f_0$ such that $\|f^* - f_0\| \leq \exp(\poly(n,d)) \cdot r$, the algorithm outputs $\hat{f}$ such that $\|\hat{f} - f^*\| \leq O(r)$.
\end{lemma}

\begin{proof}[Proof of \cref{thm:regression-main}]
  The theorem follows almost immediately from \cref{lem:regression-descent-main} and \cref{lem:regression-deterministic}.
  The only missing ingredient is the initialization $f_0$ for the algorithm of \cref{lem:regression-descent-main}.
  This may be obtained by running the classical ordinary least squares algorithm: even with probability $\delta = 2^{-n}$ it offers an estimator with $\|\hat{f} - f^*\| \leq \exp(\poly(n,d))$.
%  \Snote{post-deadline do a better job of describing initialization; work out sample complexity more explicitly}.
\end{proof}

\subsection{Gradient Descent for Linear Regression -- Proof of \cref{lem:regression-descent-main} }
\label[section]{sec:regression-descent}

In this section we describe and analyze our main gradient descent method for linear regression, proving \cref{lem:regression-descent-main}.
The algorithm will produce a series of iterates $f_0 = g_0,g_1,\ldots,g_T = \hat{f}$.
The key step is a subroutine to make progress: that is, if $\|f^* - g_t\| \gg r$, we need to produce $g_{t+1}$ such that $\|f^* - g_{t+1} \| \leq 0.99 \|f^* - g_t\|$.

\begin{lemma}\label[lemma]{lem:regression-descent-progress}
  There is a polynomial-time algorithm and a constant $C > 0$ with the following guarantees.
  Given $(X_1,Y_1),\ldots,(X_n,Y_n)$ for which \cref{eq:noise,eq:norm-upper,eq:norm-lower} all hold (for some $r > 0$ and a linear function $f^*$) and a linear function $g$ such that $\|g - f^*\| \geq Cr$, the algorithm produces $g'$ such that
  \[
  \|g' - f^*\| \leq 0.999 \|g - f^*\|\mper
  \]
  Additionally, if $\|g - f^*\| \leq Cr$, the algorithm outputs \textsc{certify}.
\end{lemma}

\cref{lem:regression-descent-main} follows immediately from \cref{lem:regression-descent-progress}: given $f_0$, the algorithm will iterate the procedure from \cref{lem:regression-descent-progress} until it finds $g$ for which it outputs \textsc{certify}.
Then this $g$ is output as $\hat{f}$.
Since $\|g' - f\| \leq 0.999 \|g - f^*\|$ at each iteration, by our assumptions on $f_0$ only $\poly(n,d)$ iterations are required, so the overall algorithm runs in polynomial time.
We focus now on proving \cref{lem:regression-descent-progress}.

\subsubsection{Main Lemmas and Proof of \cref{lem:regression-descent-progress}}
In this subsection we accumulate the main lemmas needed to prove \cref{lem:regression-descent-progress} and prove the latter.
The first lemma states that if $\|g - f^*\| \gg r$ then it is possible to find a pseudodistribution on functions $f$ such that the loss of $f$ is noticeably less than the loss of $g$ on most buckets.

\begin{lemma}\label[lemma]{lem:regression-pdist-feasible}
    Suppose given $(X_1,Y_1),\ldots,(X_n,Y_n)$ and $r > 0$ for which \cref{eq:noise,eq:norm-upper,eq:norm-lower} all hold for some linear function $f^*$, and also given a linear function $g$ such that $\|g - f^*\| >Cr$ for a large-enough constant $C$.
    Then it is possible to find in polynomial time a degree-$4$ pseudodistribution on variables $f_1,\ldots,f_d$ (representing a linear function $f$) and $b_1,\ldots,b_k$ which has the following properties, for some number $s \geq 0.99 \|g - f^*\|$.
    \begin{align*}
        & \pE \text{ satisfies } b_i^2 = b_i \\
        & \pE \text{ satisfies } \sum_{i \leq k} b_i = 0.998k\\
        & \pE \text{ satisfies } \|f - g\|^2 = s^2 \\
        & \pE b_i L_i(f) \leq \pE b_i(L_i(g) - 0.97s^2) \\
    \end{align*}
\end{lemma}

The second lemma shows that the pseudodistribution found in \cref{lem:regression-descent-progress} makes progress towards $f^*$.

\begin{lemma}\label[lemma]{lem:regression-pdist-prog}
  Let $(X_1,Y_1),\ldots,(X_n,Y_n)$ satisfy the same conditions as in \cref{lem:regression-pdist-feasible} with respect to some linear function $f^*$.
  Suppose $g$ is another linear function with $\|g - f^*\| > Cr$ for a large-enough constant $C$, and that $\pE$ is a pseudodistribution of degree-$4$ on variables $f,b$ with the properties from the conclusion of \cref{lem:regression-pdist-feasible}.
  Then
  \[
  \pE \|f - f^*\|^2 \leq 0.999 \|g - f^*\|^2\mper
  \]
\end{lemma}

Our last lemma shows that to detect whether a linear function $g$ has $\|g - f^*\| \leq O(r)$ it is enough to check for the existence of a certain SoS proof.

\begin{lemma}
  \label[lemma]{lem:regression-certify-done}
  For some $(X_1,Y_1),\ldots,(X_n,Y_n)$ and $r > 0$ and linear function $f^*$ on $\R^d$, suppose that \cref{eq:noise,eq:norm-upper,eq:norm-lower} all hold.
  Let $g$ be another linear function.
  Then if
  \begin{align*}
  & \max \pE \sum_{i \leq k} b_i \text{ such that } \\
  & \pE \text{ satisfies } b_i^2 = b_i \\
  & \pE \text{ satisfies } b_i \iprod{Y-g,f}_i \geq b_i \cdot C \cdot r \\
  & \pE \text{ satisfies } \|f\|^2 = 1\\
  & \deg \pE = 4
  \end{align*}
  is less than $0.1k$ then $\|f^* - g\| \leq 2Cr$, while if $\|f^*-g\| \leq Cr/2$ then the quantity above is at most $0.1k$, for any large-enough $C$.
\end{lemma}

Now we can prove \cref{lem:regression-descent-progress}.

\begin{proof}[Proof of \cref{lem:regression-descent-progress}]
  The algorithm to obtain $g'$ with $\|g' - f^*\| \leq 0.999 \|g - f^*\|$ follows immediately from \cref{lem:regression-pdist-feasible} and \cref{lem:regression-pdist-prog}.
  By those lemmas, it is possible in polynomial time to obtain $\pE$ with $\pE \|f - f^*\|^2 \leq 0.999 \pE \|g - f^*\|^2$, and it suffices to output the linear function $\pE f$.
  This is because by Cauchy-Schwarz, $\| \pE f - f^*\|^2 \leq \pE \|f - f^*\|^2$.
  To decide whether to output \textsc{certify} it suffices to solve the SDP from \cref{lem:regression-certify-done} for appropriate choice of $C$ and output \textsc{certify} if the optimum value is at most $0.1k$.
\end{proof}

\subsubsection{Remaining proofs}

It remains to prove \cref{lem:regression-pdist-feasible,lem:regression-pdist-prog,lem:regression-certify-done}.
We prove them in order.

\begin{proof}[Proof of \cref{lem:regression-pdist-feasible}]
  The polynomial-time algorithm is to binary search on the parameter $s$; for each candidate value of $s$ solving the degree-$4$ SoS SDP in the lemma statement.

  First we show that with the value $s^2 = \|f^* - g\|^2$, the SoS program is feasible.
  The feasible point we exhibit is given by $f = f^*$.
  We expand $L_i(g) - L_i(f^*)$.
  \begin{align*}
      L_i(g) - L_i(f^*) & = \|g - y\|_i^2 - \|f^* - y\|_i^2 \\
      & = \|g - f^*\|_i^2 + 2 \iprod{g - f^*, f^* - y}_i\mper
  \end{align*}
  By \cref{eq:noise}, for at least $0.999k$ buckets we have $\iprod{g - f^*, f^* - y}_i \geq -sr$.
  And by \cref{eq:norm-lower}, for at least $0.999k$ buckets $\|g-f^*\|_i^2 \geq 0.99\|g - f^*\|^2$.
  So for at least $0.998k$ indices $i$ we have $L_i(g) - L_i(f^*) \geq 0.99s^2 - sr$.
  Thus, the degree-$4$ SoS program in the lemma statement is feasible (using $f = f^*$) with the correct choice of $s$ and a big-enough constant $C$.

  By a similar argument, the program in the lemma statement remains feasible for any choice of $s$ in the range $[0.99\|g-f^*\|, 1.01 \|g-f^*\|]$, by choosing $f = \e g + (1-\e)f^*$ for $\e = \e(s)$ such that $\|f - g\|^2 = s^2$.
  Such $s$ can be found by binary search.
\end{proof}

We turn to \cref{lem:regression-pdist-prog}.

\begin{proof}[Proof of \cref{lem:regression-pdist-prog}]
  We start by rearranging $\pE b_i L_i(f) \leq \pE b_i(L_i(g) - 0.97s^2)$ as
  \[
    \pE b_i (\|f - f^*\|_i^2 + 2 \iprod{f - f^*, f^* - y}_i + \|f^* - y\|_i^2 ) \leq \pE b_i (\|g - f^*\|_i^2 + 2 \iprod{g-f^*, f^* - y}_i + \|f^* - y\|_i^2 - 0.97s^2)
  \]
  which after grouping terms becomes
  \begin{align}
  \pE b_i \|f - f^*\|_i^2 \leq \pE b_i(\|g - f^*\|_i^2 + 2\iprod{g-f,f^*-y}_i - 0.97s^2)\mper \label{eq:loss-rearranged}
  \end{align}
  Now we define a set $G \subseteq [k]$ of \emph{good} buckets.
  Let $i \in G$ if the following all hold:
  \begin{align}
      \pE b_i \|f - f^*\|_i^2 & \geq 0.99 \pE b_i \|f - f^*\|^2 \label{eq:good-1} \\
      \pE b_i \|g - f^*\|_i^2 & \leq 1.01 \pE b_i \|g - f^*\|^2  \label{eq:good-2} \\
      \pE b_i \iprod{g-f,f^*-y}_i & \leq r s \label{eq:good-3}\mper
  \end{align}
  We claim that the contribution of buckets $B_i$ for $i \notin G$ to the sum $\sum_{i \leq k} b_i \|f - f^*\|^2$ is small in the following sense.
  Let $B \subseteq [k]$ be those buckets violating \eqref{eq:good-1}, $B' \subseteq [k]$ those violating \eqref{eq:good-2}, and $B'' \subseteq [k]$ those violating \eqref{eq:good-3}.
  Below, we will show that
  \begin{align}
    \sum_{i \in B \cup B' } \pE b_i \|f - f^*\|^2 \leq 0.016 k \pE \|f - f^*\|^2 \text{ and } \pE \sum_{i \in B''} b_i \leq 0.006k \mper \label{eq:bad-small}
  \end{align}
  Let us first see that this is enough to complete the argument.

  By definition of good buckets, if $i \in G$ then it follows from \cref{eq:loss-rearranged} that
  \[
  \pE b_i \|f - f^*\|^2 \leq 1.02 \pE b_i(1.01 \|g-f^*\|^2 - 0.96s^2)\mper
  \]
  (Here we also used that for a big-enough constant $C$ it holds that $s \geq Cr$.)
  Since $\pE$ satisfies $\sum_{i \leq k} b_i = 0.998k$, we have
  \begin{align*}
    & 0.998k \pE \|f - f^*\|^2 \\
    & = \pE \sum_{i \leq k} b_i \|f-f^*\|^2 \\
    & = \pE \sum_{i \in G} b_i \|f - f^*\|^2 + \pE \sum_{i \notin G}  b_i \|f-f^*\|^2 \\
    & \leq 1.02 \pE \sum_{i \in G} (1.01 \|g-f^*\|^2 - 0.96s^2) + \pE \sum_{i \notin G} b_i \|f - f^*\|^2 \\
    & \leq 0.9 k \|g-f^*\|^2 - 0.4ks^2 + \pE \sum_{i \notin G} b_i \|f -f^*\|^2 \text{ since $s^2 \geq 0.99 \|g-f^*\|^2$} \mper
  \end{align*}
  If we now use that $\sum_{i \in B \cup B'} b_i \|f - f^*\|^2 \leq 0.016 k \pE \|f - f^*\|^2$ from \eqref{eq:bad-small}, we can rearrange to obtain
  \begin{align*}
  & 0.96 k \pE \|f-f^*\|^2 \\
    & \leq 0.9k \|g-f^*\|^2 - 0.4ks^2 + \pE \sum_{i \in B''} b_i \|f -f^*\|^2\\
    &  \leq 0.9k \|g-f^*\|^2 - 0.4ks^2 + \Paren{\pE \Paren{\sum_{i \in B''} b_i}^2}^{1/2} (\pE \|f-f^*\|^4)^{1/2} \text{ by pseudodist. Cauchy-Schwarz} \\
    & \leq 0.9k \|g-f^*\|^2 - 0.4ks^2 +  0.1k (\pE \|f-f^*\|^4)^{1/2} \text{ by $\pE \sum_{i \in G} b_i \geq 0.992k$ } \\
    & \leq 0.9k \|g-f^*\|^2 - 0.4ks^2 +  0.2k (\pE \|f-g\|^4 + \|g-f^*\|^4)^{1/2} \text{ by pseudodist. triangle inequality} \\
    & \leq 0.9k \|g-f^*\|^2 - 0.4ks^2 + 0.35k s^2 \text{ since $s^2 \geq 0.99\|f-g^*\|^2$ and $\pE$ satisfies $\|f-g\|^2 = s^2$}\\
    & \leq 0.9k \|g-f^*\|^2\mper
  \end{align*}
  The lemma follows.

  \subparagraph{Proof of \cref{eq:bad-small}}
  It remains to establish \cref{eq:bad-small}.
  First let us establish that we may assume $\pE \|f - f^*\|^2 \geq s^2/10$.
  Otherwise,
  \[
  \pE \|f-f^*\|^2 \leq \frac{s^2}{10} = \frac 1 {10} \pE \|f-g\|^2 \leq \frac 1 5 (\pE \|f-f^*\|^2 + \pE \|g-f^*\|^2)
  \]
  which rearranges to imply $\pE \|f-f^*\|^2 \leq 0.75 \pE \|f-g^*\|^2$, so the lemma would follow.

  Continuing with the proof of \cref{eq:bad-small}, let $B \subseteq [k]$ be the set of indices where \eqref{eq:good-1} fails.
  If $\pE \sum_{i \in B} b_i \|f-f^*\|^2 \geq 0.008k \pE \|f-f^*\|^2 $, we claim that by modifying $\pE$ to set $b_i$ to zero for $i \notin B$ and replacing $f-f^*$ with $h = \tfrac{f-f^*}{2s}$ we would obtain a pseudodistribution which violates \eqref{eq:norm-lower}.
  To see this we must check feasibility for the norm lower bound SDP.
  The main constraint to check is $\pE \|h\|^4 \leq 1$.
  For this we observe that
  \[
  \pE \|f-f^*\|^4 \leq 2 \pE \|f-g\|^4 + 2 \pE \|g - f^*\|^4 \leq 6s^4
  \]
  by pseudoexpectation triangle inequality.
  We conclude that $\pE \sum_{i \in B} b_i \|f-f^*\|^2 \leq 0.008k$.

  By an analogous argument, this time violating \eqref{eq:norm-upper}, if $B'$ is the set of indices where \eqref{eq:good-2} fails then $\pE \sum_{i \in B'} b_i \|f-f^*\|^2 \leq 0.008k \pE \|f-f^*\|^2$.

  Lastly we establish the second part of \eqref{eq:bad-small}.
  If $B''$ is the set of indices where \eqref{eq:good-3} fails, if $\pE \sum_{i \in B''} b_i \geq 0.006k$ then by modifying $\pE$ by setting $b_i$ to zero for $i \notin B''$ and replacing $\pE (f-g)$ by $\pE (f-g)/s$ we obtain a pseudodistribution which violates \cref{eq:noise}.
\end{proof}

Finally we turn to the proof of \cref{lem:regression-certify-done}.

\begin{proof}[Proof of \cref{lem:regression-certify-done}]
  We start with the second implication.
  Suppose $\|f^*-g\| \leq Cr/2$, and suppose $\pE$ satisfies the constraints of \cref{lem:regression-certify-done}.
  Then we see that $\pE$ also satisfies
  \[
  b_i \iprod{Y-f^*,f}_i \geq b_i C r - b_i \iprod{f^* - g,f}_i \geq b_i C r - b_i \|f^*-g\|_i\mper
  \]
  Since \cref{eq:norm-upper} holds, this is at least $b_i C r - 1.01 b_i \|f^* - g\| \geq b_i C r /3$.
  For large-enough $C$, $\pE$ is feasible for the noise correlation SDP.
  Since \cref{eq:noise} holds, we must have $\pE \sum b_i \leq 0.001k$.

  Now we tackle the first implication.
  By \cref{eq:noise} and hypothesis on $g$ and pigeonhole principle, there exist at least $0.95k$ indices $i \leq k$ such that $\iprod{Y-g,\tfrac{f^*-g}{\|f^*-g\|}}_i \leq Cr$ and $\iprod{Y-f^*, \tfrac{g- f^*}{\|g-f^*\|}}_i \leq r \leq Cr$.
  The former rearranges to $\iprod{g-Y, \tfrac{g-f^*}{\|f^*-g\|}}_i \leq Cr$.
  Adding, we find $\|g-f^*\|_i^2 / \|g-f^*\| \leq (C+1)r$.
  Since \cref{eq:norm-lower} holds, for one $i$ for which $\|g-f^*\|_i^2 / \|g-f^*\| \leq (C+1)r$ we also have $\|g-f^*\|_i^2 \geq 0.99 \|g-f^*\|^2$.
  Putting these together, we find $\|g-f^*\| \leq 2Cr$ (for large-enough $C$), which proves the lemma.
\end{proof}

\subsection{Certification -- Proof of \cref{lem:regression-deterministic}}

In this section we show that the conditions on a set of regression data $(X_1,Y_1),\ldots,(X_n,Y_n)$ which our algorithm from \cref{sec:regression-descent} requires hold with high probability.
Our proof comes in three parts, one for each of \cref{eq:noise,eq:norm-upper,eq:norm-lower}.
We start by establishing some notation.
Throughout the section,
\begin{itemize}
  \item $X_1,\ldots,X_n$ are i.i.d. copies of an $O(1)$-nice, zero mean random variable $X$ on $\R^d$ with covariance $\Id$.
  \item $\tilde{X_i} = X_i \cdot 1_{\|X_i\| \leq \sqrt d \log d}$.
  \item $\tilde \Sigma = \E \tilde{X} \tilde{X}^\top$.
  \item $B_1 \cup \ldots \cup B_k = [n]$ is a partition of $n$ into $k$ equal parts.
  \item For $i \leq k$, $Z_i = \E_{j \sim B_i} \tilde{X}_j \tilde{X}_j^\top - \tilde{\Sigma}$.
  \item $\tilde{Z_i} = Z_i \cdot 1_{\|Z_i\|_2 \leq C \E \|Z_i\|_2 }$ for a big-enough constant $C$.
  \item We assume throughout that $n \geq C k$ and $n \geq k^{1/2} d (\log d)^{C}$ for $C$ large enough.
\end{itemize}

Now we prove some useful facts.
The following \cref{fact:trunc-cov} can be proved by standard applications of Holder's and Markov's inequalities.
\begin{fact}\label[fact]{fact:trunc-cov}
  $\| \tilde{\Sigma} - \Id\|_2 \leq o(1)$.
\end{fact}

The following \cref{fact:trunc-spec-norm} may be proved by straightforward application of the matrix Bernstein inequality.

\begin{fact}\label[fact]{fact:trunc-spec-norm}
%  Let $X$ be a mean-zero $O(1)$-nice random variable on $\R^d$ with covariance $\Id$.
%  Let $\alpha > 0$.
%  Let $\tilde{X} = X \cdot 1_{\|X\| \leq \alpha}$.
%  Let $\tilde{\Sigma} = \E[\tilde{X}\tilde{X}^\top]$.
%  Let $X_1,\ldots,X_m$ be i.i.d. copies of $X$.
%  Let $Z = \frac 1 m \sum_{i \leq m} \tilde{X} \tilde{X}^\top - \tilde{\Sigma}$.
  For all $i \leq k$, $\E \|Z_i\|_2 \leq O(d (\log d) k /n + 1) \log d + O(\sqrt{d \log d k/ n})$.
\end{fact}

The next fact will allow us to control a key variance term in the proof of \cref{lem:regression-deterministic}.

\begin{fact}
\label{fact:reg-det-variance}
  For any $\e > 0, c > 1$ there is $C > 0$ such that
  \[
  \E_{X_i} \max_{\pE} \pE \sum_{i \leq k} 1_{\|Z_i\|_2 \leq c \E \|Z_i\|_2} (\|f\|_i^2 - \|f\|^2)^2 \leq \e k
  \]
  so long as $n \geq k^{1/2} d (\log d)^C$ and $n \geq Ck$.
  Here the maximization is over $\pE$ in variables $f_1,\ldots,f_d$ with $\pE \|f\|^4 \leq 1$.
\end{fact}
\begin{proof}
  Let $\tilde Z_i = Z_i \cdot 1_{\|Z_i\|_2 \leq c \E \|Z_i\|_2}$.
  Let $\tilde{\Sigma} = \E \tilde{X} \tilde{X}^\top$.
  We have for any $\pE$ by the definitions and triangle inequality,
  \begin{align*}
  \pE \sum_{i \leq k} 1_{\|Z_i\|_2 \leq c \E \|Z_i\|_2} (\|f\|_i^2 - \|f\|^2)^2 & = \pE \sum_{i \leq k} (\iprod{f, \tilde{Z_i} f} + \iprod{f, (\tSigma - \Id) f})^2 \\
  & \leq 2 \pE \sum_{i \leq k} \iprod{f, \tilde{Z_i} f}^2 +  2\pE \sum_{i \leq k} \iprod{f, (\tSigma - \Id) f}^2 \\
  & \leq 2 \pE \sum_{i \leq k} \iprod{f, \tilde{Z_i} f}^2 + o(k)  \, \, \text{by \cref{fact:trunc-cov}}\mper
  \end{align*}
To bound the remaining term above, $\pE \sum_{i \leq k} \iprod{f, \tilde{Z_i} f}^2$, we employ \cref{lem:p-main}.
This says that
\[
\E_{X_i,Y_i} \max_{\pE} \pE \sum_{i \leq k} \iprod{f, \tilde{Z_i} f}^2 \leq O \Paren{\frac 1 n \cdot k^{3/2} \cdot d \cdot \sqrt{\log d}} + O\Paren{ \frac{k^2} n } + O\Paren{[\E \|Z\|_2]^2 \cdot \log d}\mper
\]
(See note below on applying \cref{lem:p-main} even though we are maximizing over degree-$4$ pseudodistributions rather than degree-$8$.)
Using \cref{fact:trunc-spec-norm} to bound $\E \|Z \|_2$ and putting it all together, we find that
\[
\E_{X_i,Y_i} \max_{\pE} \sum_{i \leq k} 1_{\|Z_i\|_2 \leq c \E \|Z_i\|_2} (\|f\|_i^2 - \|f\|^2)^2 \leq \e k
\]
using $\alpha = C \sqrt d$ and $n \gg k^{1/2} d$ and $\delta \geq 2^{-\Theta(n)}$.
This concludes the proof

\textbf{Degree 4 versus Degree 8:} In the above we used \cref{lem:p-main}, which as stated applies only to degree-$8$ pseudodistributions. We briefly explain why it can be applied here.
Inspecting the proof of \cref{lem:p-main}, we see that the only place where degree-$8$-ness is used is to bound $\E_X \pE \|X\|^4 \iprod{X,u}^4 \leq (\E_X \|X\|^8 )^{1/2} \cdot (\E_X \iprod{X,u}^8)^{1/2}$, followed by the application of certifiable $(2,8)$ hypercontractivity to bound $\E_X \iprod{X,u}^8$.
In our current setting, we are working with more aggressively truncated variables $\tilde{X_i} = X_i \cdot \Ind_{\|X_i\| \leq O(\sqrt{d})}$ than in the covariance estimation setting.
This truncation allows for an alternative analysis at this point of the proof of \cref{lem:p-main}: $\E \|\tilde{X}\|^4 \pE \iprod{\tilde{X},v}^4 \leq O(d^2) \pE \iprod{X,v}^4$, which requires only degree-$4$ $\pE$.
\end{proof}

\subsection{Proof of \cref{eq:norm-upper}}

By a bounded-differences argument identical to \cref{lem:sos-conc}, it will be enough to prove the following claim.

\begin{claim}
\label[claim]{clm:norm-upper-exp}
  With notation as in \cref{lem:regression-deterministic}, $\E \max_{\pE} \pE \sum_{i \leq k} b_i \|f\|^2 \leq 0.0001k$ where the maximum is taken over all degree-$4$ pseudoexpectations such that $\pE \|f\|^4 \leq 1$, $\pE$ satisfies $\{b_i^2 = b_i\}$, and $\pE b_i \|f\|_i^2 \geq 1.01 \pE b_i \|f\|^2$.
\end{claim}

\begin{proof}[Proof of \cref{clm:norm-upper-exp}]
  We will partition $[k]$ into \emph{good} and \emph{bad} sets $[k] = G \cup B$.
  Let $G = \{i \, : \, Z_i = \tilde{Z}_i \}$ and let $B = [k] \setminus G$.
  For any degree-$4$ pseudoexpectation satisfying the constraints of the norm upper bound SDP, we have
  \[
    \pE \sum_{i \leq k} b_i \|f\|^2 \leq |B| + \sum_{i \in G} \pE b_i \|f\|^2\mper
  \]
  Here we used $\pE \sum_{i \in B} b_i \|f\|^2 \leq \sum_{i \in B} (\pE b_i^2)^{1/2} (\pE \|f\|^4)^{1/2} \leq |B|$.
  Bounding the second term, we have
  \[
  \pE \sum_{i \in G} b_i \|f\|^2 \leq  0.99 \sum_{i \in G} \pE b_i \|f\|_i^2 = 0.99 \sum_{i \in G} \pE b_i \|f\|^2 + 0.99 \sum_{i \in G} \pE b_i (\|f\|_i^2 - \|f\|^2)\mper
  \]
  This rearranges to
  \begin{align}\label{eq:norm-upper-1}
    \pE \sum_{i \in G} b_i \|f\|^2 \leq 100 \sum_{i \in G} \pE b_i (\|f\|_i^2 - \|f\|^2) \leq 100 \sqrt{k}  \Paren{\pE \sum_{i \in G} (\|f\|_i^2 - \|f\|^2)^2}^{1/2} \mper
  \end{align}
  By definition of $G$, we can use \cref{fact:reg-det-variance} to obtain
  \[
  100 \sqrt{k} \E \max_{\pE} \Paren{\pE \sum_{i \in G} (\|f\|_i^2 - \|f\|^2)^2}^{1/2} \leq 0.000001 k\mper
  \]
  By Markov's inequality $\E |B| \leq 0.00001 k$, which completes the proof.
\end{proof}

\subsection{Proof of \cref{eq:norm-lower}}

As in the previous section, it will suffice to prove the following claim.

\begin{claim}
\label[claim]{clm:norm-lower-exp}
  With notation as in \cref{lem:regression-deterministic}, $\E \max_{\pE} \pE \sum_{i \leq k} b_i \|f\|^2 \leq 0.0001k$ where the maximum is taken over all degree-$4$ pseudoexpectations such that $\pE \|f\|^4 \leq 1$, $\pE$ satisfies $\{b_i^2 = b_i\}$, and $\pE b_i \|f\|_i^2 \leq 0.99 \pE b_i \|f\|^2$.
\end{claim}
\begin{proof}[Proof of \cref{clm:norm-lower-exp}]
  The proof is similar to that of \cref{clm:norm-lower-exp}.
  Let the good and bad sets $B,G$ be as they were there.
  Once again, we note
  \[
  \pE \sum_{i \leq k} b_i \|f\|^2 \leq |B| + \sum_{i \in G} \pE b_i \|f\|^2
  \]
  As before, we can use Markov's inequality to conclude $\E |B| \leq 0.000001k$.
  So we just need to bound the second term.
  We have
  \[
  \sum_{i \in G} \pE b_i\|f\|^2 = \sum_{i \in G} \pE b_i \|f\|_i^2 + \sum_{i \in G} \pE b_i (\|f\|^2 - \|f\|_i^2)\mper
  \]
  By the constraints on $\pE$, this is at most
  \[
  0.99 \sum_{i \in G} \pE b_i \|f\|^2 + \sum_{i \in G} \pE b_i (\|f\|^2 - \|f\|_i^2)\mper
  \]
  So rearranging, we get
  \[
  \sum_{i \in G} \pE b_i \|f\|^2 \leq 100 \sum_{i \in G} \pE b_i (\|f\|^2 - \|f\|_i^2)\mper
  \]
  The remainder of the proof proceeds as in \cref{clm:norm-upper-exp}.
\end{proof}

\subsection{Proof of \cref{eq:noise}}

The argument is identical to one which appears in \cite{hopkins2018sub,cherapanamjeri2019fast}.
It suffices to note that $\iprod{Y-g,f}_i = f^\top (\frac k n \sum_{j \in B_i} \tilde{X_j} \e_j)$ where on the RHS of this equation $f$ is considered as a vector of coefficients.
Since the random vector $\frac k n \sum_{j \in B_i} \tilde{X_j}\e_j$ satisfies the hypotheses of Lemma 2.8 of \cite{hopkins2018sub}, this completes the proof by applying Lemma 2.8.

By a bounded-differences argument as in the proof of \cref{lem:sos-conc}, it is enough to prove the following lemmas concerning the \emph{expected} values of the noise SDP, norm upper bound SDP, and norm lower bound SDP.

\begin{lemma}\label[lemma]{lem:regression-cert-1}
  Let $f^*$ be a linear function on $\R^d$.
  Let $X_1,\ldots,X_n$ be i.i.d. copies of an $O(1)$-nice random variable $X$ on $\R^d$ with $\E X = 0$ and $\E XX^\top = \Id$.
  Let $\e_1,\ldots,\e_n$ be i.i.d. copies of a random variable $\e$ on $\R$ with $\E \e = 0$ and $\E \e^2 = 1$.
  Let $Y_i = f^*(X_i) + \e_i$.
  Let $k \in N$ and let $B_1,\ldots,B_k \subseteq [n]$ partition $[n]$ into $k$ parts of size $n/k$.
  \begin{align*}
  \E \Brac{\max_{\substack{ \deg \pE = 4 \\ \pE \text{ satisfies } \cAnoise}} \pE \sum_{i=1}^k b_i  } \leq 0.0005k
  \end{align*}
  where $\cAnoise$ is instantiated with the function $f^*$ and $r^2 \geq C (d/n + k/n)$ and with truncated samples $(\ttX_i, \tY_i)$.
\end{lemma}

\begin{proof}[Proof of \cref{lem:regression-cert-1}]
  The proof is very similar to \cite[Lemma 3.2]{hopkins2018sub}.
  The only twist is that we need to handle truncation of the samples $\tilde{X}_i,\tilde{Y}_i$.
\end{proof}

\section{Mean estimation in (almost) any norm in $\poly(d,1/\delta)$ time}

In this section we prove the following theorem, concerning estimation of the mean of a heavy-tailed random vector in general norms.

\begin{theorem}\label[theorem]{thm:general-norms}
  Let $\cB \subseteq \R^d$ be (the unit ball of) a norm and let $\cO$ be a separation oracle for the unit ball $\cB_*$ of its dual norm.
  There is a universal constant $C$ such that for every $n \in \N$ and $\delta \geq 2^{-n}$ there is an algorithm with the following guarantees.
  Let $X$ be a random vector on $\R^d$ with $\mu = \E X$ and $\Sigma = \E(X - \mu)(X - \mu)^\top$.
  Given $n$ independent samples $X_1,\ldots,X_n$ from $X$ the algorithm produces an estimator $\hat{\mu}$ such that with probability at least $1-\delta$,
  \[
  \|\hat{\mu} - \mu\|_{\cB} \leq \frac{C}{\sqrt{n}} \cdot \max \Paren{ \E_{\sigma,X} \frac{1}{\sqrt{n}}\Norm{\sum_{i \leq n} \sigma_i (X_i - \mu)}, R \sqrt{\log(1/\delta)} }
  \]
  where
  \begin{itemize}
    \item $\sigma_1,\ldots,\sigma_n$ are i.i.d. Rademacher random variables, and
    \item $R = \|\Sigma^{1/2}\|_{2 \rightarrow \cB} = (\sup_{x \in \cB^*} x^\top \Sigma x)^{1 / 2}$ is the $2 \rightarrow \cB$ norm of $\Sigma^{1/2}$.
  \end{itemize}
  The algorithm runs in time $O(dn) + \poly(d,1/\delta)$ and makes at most $\poly(d,1/\delta)$ calls to $\cO$.
\end{theorem}

The algorithm will directly compute a simplified version of the median-of-means estimator of Lugosi and Mendelson \cite{lugosi2018near}.
The nontrivial aspect of our algorithm is that it avoids brute-force search over an $\epsilon$-net of size $2^d$; we show that instead it is possible to limit brute-force search to a space of size $O(1/\delta)$.

To set up our algorithm and its analysis we make one important definition.

\begin{definition}
  Let $Z_1,\ldots,Z_k \in \R^d$ and let $\cB \subseteq \R^d$ be a norm.
  For $r > 0$ and $p \in [0,1]$, a point $x \in \R^d$ is $(r,p)$-central (with respect to $\cB$) if for every $u \in \cB_*$ there are at most $pk$ vectors $Z_1,\ldots,Z_k$ such that $\iprod{Z_i - x, u} > r$.
\end{definition}

We define a median-of-means estimator for the mean with respect to a norm $\cB$.
Our definition matches that of Lugosi and Mendelson \cite{lugosi2018near}, except for a small simplification in the definition of the set of points the estimator chooses from to ensure convexity; it is straightforward to show that this does not change the analysis of the estimator.

\begin{definition}[The $\cB$ mean estimator]
  Let $\cB \subseteq \R^d$ be a norm.
  Let $n \in \N$ and $\delta > 2^{-n}$.
  Let $k = C \log 1/\delta$ for some constant $C$ and let $B_1,\ldots,B_k \subseteq [n]$ partition $[n]$ into equal-size sets.
  We define the following estimator for the mean of a random vector.
  Given $X_1,\ldots,X_n$ i.i.d. copies of a random vector $X$ on $\R^d$, let $Z_i = \E_{j \sim B_i} X_j$.
  For the minimal $r$ such that an $(r,1/10)$-central point with respect to $Z_1,\ldots,Z_k$ exists, find such a point $x$ and output it.
\end{definition}

The main statistical analysis is captured by the following main theorem of \cite{lugosi2018near}.

\begin{theorem}[Theorem 2 of \cite{lugosi2018near}]\label[theorem]{thm:general-norms-info}
  Let $\cB \subseteq \R^d$ be a norm.
  Let $X$ be a random variable on $\R^d$ with $\mu = \E X$ and $\Sigma = \E(X - \mu)(X - \mu)^\top$.
  Let $n \in \N$ and $\delta \geq 2^{-n}$.
  There is a universal constant $C$ such that if independent copies $X_1,\ldots,X_n$ of $X$ are partitioned into $k = C \log(1/\delta)$ buckets $B_1,\ldots,B_k \subseteq [n]$ with $|B_i| = n/k$ and we let $Z_i = \tfrac k n \sum_{j \in B_i} X_j$, then with probability at least $1-\delta$ the vector $\mu$ is $(r,1/10)$-central with respect to $Z_1,\ldots,Z_k$, for
  \[
  r \leq O \Paren{\frac{1}{\sqrt{n}} \cdot \max \Paren{ \E_{\sigma,X} \frac{1}{\sqrt{n}}\Norm{\sum_{i \leq n} \sigma_i X_i}, R \sqrt{\log(1/\delta)}}}
  \]
  where $\sigma,R$ are as in \cref{thm:general-norms}.
\end{theorem}

We note that Theorem 2 of \cite{lugosi2018near} has an additional term $\E \|G\|$ on the right-hand side, where $G$ is a Gaussian with covariance $\Sigma$.
We provide for reference a simple proof of \cref{thm:general-norms-info} in the appendix which shows that this term is unnecessary.

The analysis of this median-of-means estimator is completed by the following simple lemma:

\begin{lemma}\label[lemma]{lem:central-close}
  Suppose $Z_1,\ldots,Z_k \in \R^d$ and $\cB$ is a norm, and $x,y$ are both $(r,1/10)$-central.
  Then $\|x-y\|_{\cB} \leq 2r$.
\end{lemma}
\begin{proof}
    By definition, $\|x - y\|_{cB} = \sup_{u \in \cB_*} \iprod{x-y,u}$.
    Since $x,y$ are $(r,1/10)$ central, for every $u$ there is $Z_i$ such that $\iprod{Z_i-y,u} \leq r$ and $\iprod{x - Z_i,u} \leq r$.
    Adding these two, we obtain $\iprod{x-y,u} = \iprod{x - Z_i,u} + \iprod{Z_i -y,u} \leq 2r$.
\end{proof}

In light of \cref{thm:general-norms-info} and \cref{lem:central-close}, to prove \cref{thm:general-norms} it will suffice to give an algorithm which finds an $(r,1/10)$-central point given $Z_1,\ldots,Z_k$, if such exists.
For this we prove the following two lemmas.

\begin{lemma}\label[lemma]{lem:central-convex}
  Let $\cB \subseteq \R^d$ be a norm and $Z_1,\ldots,Z_k \in \R^d$.
  For all $r > 0$ and $p \in [0,1]$, the set of $(r,p)$-central points is convex.
\end{lemma}
\begin{proof}
  For $u \in \cB_*$, let
  \[
  S_u = \{ x \in \R^d \, : \, \iprod{Z_i - x,u} \leq r \text{ for at least } (1-p)k \text{ $Z_i$'s} \}\mper
  \]
  The set of $(r,p)$-central points is exactly $\bigcap_{u \in \cB_*} S_u$, so it suffices to show that $S_u$ is convex; it is easy to see that $S_u$ is in fact a half-space, so we are done.
\end{proof}

\begin{lemma}\label[lemma]{lem:central-separate}
  Let $\cB \subseteq \R^d$ be a norm and let $Z_1,\ldots,Z_k \in \R^d$.
  For all $r > 0$ and $p \in [0,1]$ there is a separation oracle for the set of $(r,p)$-central points which runs in time $2^k \cdot \poly(d)$ and makes at most $2^k \cdot \poly(d)$ calls to a separation oracle $\cO$ for the norm $\cB_*$.
\end{lemma}
\begin{proof}
  By rearranging the definition of centrality, $x$ is $(r,p)$-central if and only if for all $T \subseteq [k]$ with $|T| = pk+1$ there is no $u \in \cB_*$ such that $\iprod{Z_i-x,u} \geq r$ for all $i \in T$.

  Suppose given $x \in \R^d$.
  The separation oracle proceeds as follows.
  For all $T \subseteq [k]$ with $|T| = pk+1$, check via the ellipsoid method using $\poly(d)$ calls to $\cO$ and $\poly(d)$ additional running time that whether the following set is nonempty:
  \[
  S_T = \{ u \in \cB_* \, : \, \iprod{Z_i-x,u} \geq r \text{ for all } i \in T\}\mper
  \]
  If $S_T = \emptyset$ for all $T$, then $x$ is $(r,p)$-central.
  Otherwise, suppose there exists $u \in S_T$ for some $T \subseteq [k]$.
  Then $u$ separates $x$ from the set $S_u$ (as defined in the proof of \cref{lem:central-convex}), so the separation oracle may output the linear function $f(x) = \iprod{x,u} + r - \iprod{Z_i,u}$, where $Z_i$ is such that $\iprod{Z_i,u}$ is the $pk$-th least number among $\{ \iprod{Z_i,u} \}_{i \in [k]}$.

  There are at most $2^k$ choices for $T$, and hence the separation oracle requires $2^k \cdot \poly(d)$ calls to $\cO$ and $2^k \cdot \poly(d)$ additional running time.
\end{proof}

Now we can prove \cref{thm:general-norms}.

\begin{proof}[Proof of \cref{thm:general-norms}]
  Given $X_1,\ldots,X_n$, our algorithm first computes bucketed means $Z_1,\ldots,Z_k$ for $k = C \log(1/
  \delta)$, with $C$ a big-enough constant to be chosen later..
  Note that for each $r$, together \cref{lem:central-convex,lem:central-separate} imply that there is an Ellipsoid-based algorithm to find an $(r,p)$-central point or determine that none exists.\footnote{Formally to obtain this guarantee from the Ellipsoid method one must ensure that the convex set of interest is sandwiched between exponentially-small and exponentially-large balls in $\R^d$ \cite{MR1261419-Grotschel93}. This is straightforward in our case by observing that if there is any $(r,p)$-central point then there is a small ball of $(r+\e,p)$-central points for small $\e > 0$, and that any central point is contained in the smallest ball containing all of $Z_1,\ldots,Z_k$.}

  It just remains to show that by binary search our algorithm can find (up to a factor of $2$) a value $r'$ such that $r' \leq O(r)$ where $r = \frac{1}{\sqrt{n}} \cdot \max \Paren{ \E_{\sigma,X} \Norm{\sum_{i \leq n} \sigma_i X_i}, R \sqrt{\log(1/\delta)} }$ is as in the theorem statement.

  First we show that the algorithm may easily compute an upper bound on this value of $r$ -- the upper bound is $d \cdot \max_{i,j \leq k} \|Z_i - Z_j\|_2$.
  With probability at least $1-\delta/2$, we have
  \[
  d \cdot \max_{i,j \leq k} \|Z_i - Z_j\|_2 \leq r \cdot \poly(d,1/\delta)
  \]
  by factor-$d$ equivalence of $\cB$ and $\ell_2$ together with Chebyshev's inequality.
  So by running at most $\poly(d,1/\delta)$ steps of binary search on $r'$, the algorithm finds $r' \leq O(r)$ such that there exists an $(r',1/10)$-central point, and outputs that point.
\end{proof}

\subsection{Proof of \cref{thm:general-norms-info}}

We will first formulate a polynomial optimization problem to test whether a given point, $x \in \R^d$ is $(r,p)$-central with respect to the vectors, $Z_1, \dots, Z_k$:

\begin{gather*}
    \max \sum_{i = 1}^k b_i \\
    b_i\iprod{v, Z_i - x} \geq b_i r \\
    b_i^2 = b_i \\
    v \in \cB^* \tag{GEN-TST} \label{eq:gtst}
\end{gather*}

We will denote the optimal value of the above optimization problem instantiated with vectors $\bZ = \{Z_i, \dots, Z_k\}$, a point $x$ and a radius $r$ as $m = \text{\ref{eq:gtst}}(\bZ, x, r)$. Note now, that a point $x$ is $(r,p)$-central with respect to the points $Z_1, \dots, Z_k$ if \ref{eq:gtst}$(\bZ, x, r)$ is less than $pk$. Before we proceed with the proof of \cref{thm:general-norms-info}, we will first show that \text{\ref{eq:gtst}} satisfies the bounded differences condition with respect to the inputs, $Z_i$.

\begin{lemma}
  \label[lemma]{lem:genbd}
  Let $\bZ = \{Z_1, \dots, Z_k\}$ and $\bZ^\prime = \{Z_1, \dots, Z_i^\prime, \dots, Z_k\}$; that is, the $i^{th}$ vector, $Z_i$, is replaced by $Z_i^\prime$. Then, for any $x \in \R^d$ and $r > 0$, we have $m = \text{\ref{eq:gtst}}(\bZ, x, r)$ and $m^\prime = \text{\ref{eq:gtst}}(\bZ^\prime, x, r)$ satisfying:
  \begin{equation*}
      \abs{m - m^\prime} \leq 1
  \end{equation*}
\end{lemma}

\begin{proof}
  Let $x \in \R^d$ and $r > 0$. And let the $v \in \cB^*$ and $b_1, \dots, b_k$be the maximizers of \ref{eq:gtst}$(\bZ, x, r)$. Now, we may construct a candidate solution for \ref{eq:gtst}$(\bZ^\prime, x, r)$, by picking $v^\prime = v$ and $b_j^\prime = b_j$ for all $j \neq i$ and $b_i = 0$. Note that the candidate pair, $(v^\prime, \bb^\prime)$ is feasible for \ref{eq:gtst}$(\bZ^\prime, x, r)$ and therefore, we may conclude that: 
  \begin{equation*}
      m^\prime \geq \sum_{i = 1}^k b^\prime_i \geq \sum_{i = 1}^k b_i - 1 = m - 1
  \end{equation*}
  Similarly, we may conclude that $m \geq m^\prime - 1$. The two statements conclude the proof of the lemma.
\end{proof}

We will now prove a lemma useful in bounding the expected value of the \ref{eq:gtst}$(\bZ, \mu, r)$. Before we do this, we will restate a lemma bounding the Rademacher complexity of compositions of Lipschitz functions with an underlying function class from \cite{ledoux1991probability}:

\begin{theorem}
  \label{thm:ledtal}
  Let $X_1, \dots, X_n \in \R^d$ be i.i.d.~random vectors, $\mathcal{F}$ be a class of real-valued functions on $\R^d$ and $\sigma_i, \dots, \sigma_n$ be independent Rademacher random variables. If $\phi: \R \rightarrow \R$ is an $L$-Lipschitz function with $\phi(0) = 0$, then:

  \begin{equation*}
    \E \sup_{f \in \mathcal{F}} \sum_{i = 1}^n \sigma_i \phi(f(X_i)) \leq L\cdot \E \sup_{f \in \mathcal{F}} \sum_{i = 1}^{n} \sigma_i f(X_i).
  \end{equation*}
\end{theorem}

\begin{lemma}
  \label[lemma]{lem:genexpbnd}
  Let $X_1, \dots, X_n$ be iid random vectors with mean $0$ and covariance, $\Sigma$. In addition, for $i = 1, \dots, k$, let $Z_i = \frac{k}{n} \sum_{j = \frac{i - 1}{k}n}^{\frac{i}{k}n} X_j$. Then, we have that:
  \begin{equation*}
      \E \Brac{\max_{v \in \cB^*} \abs{\iprod{Z_i, v}}} \leq \frac{k}{n} \cdot \Paren{4 \E \Brac{\Norm{\sum_{i = 1}^n \varepsilon_i X_i}} + \sqrt{kn} R}
  \end{equation*}
  where 
  \begin{enumerate}
      \item $\varepsilon_i$ are independent iid Rademacher random variables
      \item $R = \|\Sigma^{1/2}\|_{2 \rightarrow \cB}$ as in the statement of \cref{thm:general-norms-info}
  \end{enumerate}
\end{lemma}

\begin{proof}
  First note that the random vectors, $Z_i$ have covariance matrices $\frac{k}{n} \Sigma$. We now bound the quantity as follows:
  
  \begin{align*}
      \E \Brac{\max_{v \in \cB^*} \sum_{i = 1}^k\abs{\iprod{Z_i, v}}} &\leq \E \Brac{\max_{v \in \cB^*} \sum_{i = 1}^k \abs{\iprod{Z_i, v}} - \E \brac{\abs{\iprod{Z_i^\prime, v}}}} + k \max_{v \in \cB^*} \E \brac{\abs{\iprod{Z_i^\prime, v}}}
  \end{align*}
  
  We bound the second term as follows:
  
  \begin{equation*}
      \max_{v \in \cB^*} \E \Brac{\abs{\iprod{Z^\prime_i, v}}} \leq \max_{v \in \cB^*} \E \Brac{\iprod{Z^\prime_i, v}^2}^{1 / 2} = \max_{v \in \cB^*} \frac{k}{n} v^\top \Sigma v = \frac{k}{n} R^2
  \end{equation*}
  
  Let $h: [n] \rightarrow [k]$ denote the function assigning data points, $X_i$ to buckets. For the first term, we proceed as follows:
  
  \begin{align*}
      \E \Brac{\max_{v \in \cB^*} \sum_{i = 1}^k \abs{\iprod{Z_i, v}} - \E\brac{\iprod{Z^\prime_i, v}}} &\leq \E_{Z_i, Z_i^\prime} \Brac{\max_{v \in \cB^*} \sum_{i = 1}^k \abs{\iprod{Z_i, v}} - \abs{\iprod{Z^\prime_i, v}}} \\
      &\leq \E_{Z_i, Z_i^\prime, \sigma_i} \Brac{\max_{v \in \cB^*} \sum_{i = 1}^k \sigma_i (\abs{\iprod{Z_i, v}} - \abs{\iprod{Z_i^\prime, v}})} \\
      &\leq \E_{Z_i, \sigma_i} \Brac{\max_{v \in \cB^*} \sum_{i = 1}^k \sigma_i \abs{\iprod{Z_i, v}}} + \E_{Z_i^\prime, \sigma_i} \Brac{\max_{v \in \cB^*} -\sigma_i \abs{\iprod{Z_i^\prime, v}}} \\
      &= 2\E_{Z_i, \sigma_i} \Brac{\max_{v \in \cB^*} \sum_{i = 1}^k \sigma_i \abs{\iprod{Z_i, v}}} \leq 2 \E_{Z_i, \sigma_i} \Brac{\max_{v \in \cB^*} \sum_{i = 1}^k \sigma_i \iprod{Z_i, v}} \\
      &=2\cdot \frac{k}{n}\cdot \E_{X_j, \sigma_i} \Brac{\max_{v \in \cB^*} \sum_{j = 1}^n \sigma_{h(j)} \iprod{X_j, v}} \\
      &= 2\cdot \frac{k}{n}\cdot \E_{X_j, \sigma_i} \Brac{\max_{v \in \cB^*} \sum_{j = 1}^n \sigma_{h(j)} \iprod{X_j, v} - \E \sigma_{h(j)}\iprod{X_j^\prime, v}} \\
      &\leq 2\cdot \frac{k}{n}\cdot \E_{X_j, X^\prime_j, \sigma_i, \varepsilon_j} \Brac{\max_{v \in \cB^*} \sum_{j = 1}^n \epsilon_j \sigma_{h(j)} (\iprod{X_j, v} - \iprod{X_j^\prime, v})} \\
      &\leq 2\cdot \frac{k}{n}\cdot \E_{X_j, X^\prime_j, \sigma_i, \varepsilon_j} \Brac{\max_{v \in \cB^*} \sum_{j = 1}^n \epsilon_j (\iprod{X_j, v} - \iprod{X_j^\prime, v})} \\
      &\leq \frac{4k}{n} \E_{X_j, \sigma_j} \Brac{\max_{v \in \cB^*}  \Iprod{v, \sum_{j = 1}^n \varepsilon_j X_j}} = \frac{4k}{n} \E_{X_j, \sigma_j} \Brac{\Norm{\sum_{j = 1}^n \varepsilon_j X_j}}
  \end{align*}
\end{proof}

The proof of \cref{thm:general-norms-info} follows immediately by applying the bounded differences concentration inequality to \ref{eq:gtst} $(\bZ, \mu, r)$ for $k = C \log 1 / \delta$ (by \cref{lem:genbd}) and upper bounding its expected value as follows by an application of \cref{lem:genexpbnd}:

\begin{equation*}
    \E \Brac{\text{\ref{eq:gtst}}(\bZ, \mu, r)} \leq \frac{1}{r} \cdot \E \Brac{\max_{v \in \cB} \sum_{i = 1}^k \abs{\iprod{Z_i - \mu, v}}} \leq \frac{k}{20}
\end{equation*}

\hfill\qedsymbol

\newcommand{\sogmm}{\textsc{sogmm}}

\section{Roadblock to Information-Theoretic Optimality: Single-Spike Block Mixtures}
\label[section]{sec:roadblock}

We describe a simple high-dimensional testing problem which must be solved by any algorithm using our techniques (in a sense we make precise below) to substantially improve on our quantitative error rates in covariance estimation and linear regression.
We also present some mild evidence that this \emph{single-spike block mixtures} problem may be hard for polynomial-time algorithms; we view obtaining stronger evidence for hardness (or, of course, an efficient algorithm) as a fascinating open problem.

\begin{definition}[Single-Spike Block Mixtures]
  Let $d,m \in \N$ and $1 > \lambda > 0$.
  In the \emph{single-spike block mixtures testing problem} the goal is to distinguish, given vectors $y_1,\ldots,y_{md} \in \R^d$, between the following two cases:
  \begin{itemize}
  \item[\textsc{null}:] $y_1,\ldots,y_{md} \sim \cN(0,\Id)$ i.i.d.
  \item[\textsc{planted}:]
  First $x \sim \{ \pm 1/\sqrt{d} \}^d$ and $s_1,\ldots,s_d \sim \{\pm 1\}$.
  Then, $y_1,\ldots,y_m \sim \cN(0,\Id + s_1 \lambda xx^\top)$ and $y_{m+1},\ldots,y_{2m} \sim \cN(0, + s_2 \lambda xx^\top)$, and so forth.
  That is, each \emph{block} of vectors $y_{im},\ldots,y_{(i+1)m - 1}$ has either slightly larger variance in the $x$ direction (if $s_i = 1$) or slightly lesser variance (if $s_i = -1$) than they would in the null case.
  \end{itemize}
\end{definition}

\begin{remark}[Relation to Covariance Estimation and Linear Regression]
  The median-of-means approach to covariance estimation requires us to be able to find an appropriate median of empirical covariance matrices $\overline{\Sigma}_1,\ldots,\overline{\Sigma}_k$.
  In this context, a $r$-median is any matrix $M$ such that for all unit $x$ we have $\iprod{xx^\top, \Sigma_i - M} \leq r$ for at least $0.9k$ of $\Sigma_1,\ldots,\Sigma_k$.
  The $r$ for which it is possible to find an $r$-median translates directly to the error rate of the eventual median-of-means estimator.

  To find such an $r$-median, it seems crucial to be able to recognize one.
  The single-spike block mixtures problem can be reformulated as the problem of deciding whether or not the identity matrix $\Id$ is an $r$-median for the empirical covariances in the blocks.
  
   Any median-of-means algorithm for covariance estimation giving rate $d^{3/4 - \Omega(1)}/\sqrt{n}$ would solve the single-spike block mixtures problem with $\lambda \leq d^{1/4 - \Omega(1)} / \sqrt{m}$.
   This is what our moment-matching lower bound suggests is hard.

  The relationship to linear regression is a little more subtle.
  Our algorithm for linear regression even in the case that $X$ has identity covariance goes via a subroutine which, if improved to improve the overall sample complexity of our algorithm to $d^{3/2 - \Omega(1)}$, would similarly solve the single-spike block mixtures problem.
  Although whether this represents a fundamental roadblock in the case of regression with identity covariance is unclear, we do expect that for regression where $X$ has arbitrary covariance $\Sigma$ this will represent a similar roadblock, since often linear regression algorithms in the latter setting implicitly solve covariance estimation problems.
\end{remark}

\begin{remark}[Generalizations of Single-Spike Block Mixtures]
  The version of the single-spike block mixtures problem presented above has samples $y_1,\ldots,y_{md} \in \R^d$ split into $d$ buckets.
  This corresponds our median-of-means algorithms with parameters set to achieve success probability $1-\exp(-\Omega(d))$.
  To investigate the computational complexity of the more general setting of success probability $1-\delta$, we would instead consider a variant with samples $y_1,\ldots,y_{m \log(1/\delta)}$ in $\log(1/\delta)$ blocks.
  We focus for simplicity on the case $\delta = 2^{-d}$, but similar computations could be carried out for the generalized setting.
\end{remark}

To get a feel for the problem, let us first sketch an argument that it can be solved in exponential time if $\lambda \gg 1/\sqrt{m}$ and $m \gg 1$.

\begin{lemma}
  There is a constant $C$ such that if $\lambda \geq C(1/\sqrt{m} + 1/\sqrt{d})$ and $m \gg 1$ then there is a $2^d \poly(d,m)$ time algorithm which solves single-spike block mixtures with high probability.
\end{lemma}
\begin{proof}
  The algorithm is as follows.
  Let $\Sigma_i = \tfrac 1 m \sum_{j=im}^{(i+1)m} y_j y_j^\top$ be the empirical second moment of the samples in the $i$-th block.
  For all $S \subseteq [d]$, compute the maximum eigenvalue of $\sum_{i \in S} (\Sigma_i - \Id)$.
  If there is $S$ with $|S| \geq d/4$ such that this eigenvalue is greater than $\lambda |S| / 2$, then return \textsc{planted}.
  Otherwise, return \textsc{null}.

  To analyze the algorithm, we first consider what occurs in the \textsc{planted} case.
  Let $S = \{i \, : \, s_i = 1 \}$.
  With high probability, $|S| \geq d/4$.
  Consider
  \[
  \E_{x,y} \sum_{i \in S } x^\top (\Sigma_i - \Id) x = \E |S| \lambda \|x\|^4 \geq |S| \lambda\mper
  \]
  By standard concentration results, $x^\top (\sum_{i \in S} \Sigma_i - \Id)x \geq |S| \lambda /2$ with high probability in the \textsc{planted} case, so the algorithm will output \textsc{planted}

  Now let us see what happens in the \textsc{null} case -- we wish to show that the algorithm outputs \textsc{null} with high probability.
  First consider fixed $S \subseteq [d]$ with $|S| \geq d/4$.
  (Later we will take a union bound.)
  Then
  \[
  \sum_{i \in S} (\Sigma_i -\Id) = \frac 1 m \sum_{i \in S} \sum_{j=im}^{(i+1)m} y_j y_j^\top - |S|\Id
  \]
  Now, since $m \gg 1$ and $|S| \geq d/4$, there are at least $Cd$ vectors $y_j$ in the above sum, for a large-enough constant $C$.
  By standard results on concentration of eigenvalues of Gaussian matrices,
  \[
  \Norm{\frac 1 {m|S|} \sum_{i \in S} \sum_{j=im}^{(i+1)m} y_j y_j^\top - \Id} \leq O(1/\sqrt{m}) \text{ with probability at least } 1-2^{-100d}
  \]
  and hence by a union bound this holds for all $|S| \geq d/4$ with high probability.
  So as long as $\lambda \gg 1/\sqrt{m}$, the algorithm will output \textsc{null}.
\end{proof}

Next, we observe that under a stronger assumption on $\lambda$, the key subroutine from both our algorithms solves the single-spike block mixtures problem.

\begin{lemma}
  If $\lambda \gg d^{1/4} / \sqrt{m}$, then there is a polynomial-time algorithm to solve single-spike block mixtures.
\end{lemma}
\begin{proof}
  Let $\Sigma_i = \tfrac 1 m \sum_{j=im}^{(i+1)m} y_j y_j^\top$ be the empirical second moment of the samples in the $i$-th block.
  The algorithm is to find the smallest $c$ such that there is a degree-8 SoS proof that
  \[
  \max_{b,u} \sum_{i \leq d} b_i \iprod{\Sigma_i - \Id, uu^\top} \leq c
  \]
  subject to $b_i^2 = b_i, \|u\|^2 = 1$.
  If $c \geq \lambda d/4$, then return \textsc{planted}, otherwise return \textsc{null}.

  The analysis of the algorithm follows from \cref{lem:p-main} via arguments as in \cref{lem:regression-deterministic}.
\end{proof}

Now we turn to our main theorem for this section, capturing the moment matching lower bound.

\begin{theorem}
\label[theorem]{thm:roadblock-main}
  If $\lambda \ll \tfrac{d^{1/4}}{\sqrt{m} \poly \log(d,m)}$ then every degree-$(md)^{o(1)}$ function $f \, : \, y_1,\ldots,y_{md} \rightarrow \R$ such that $\E_{\textsc{null}} f(y_1,\ldots,y_{md})^2 = 1$ has $|\E_{\textsc{planted}} f - \E_{\textsc{null}} f| \leq o(1)$.
\end{theorem}

It follows from by-now standard linear algebra that the theorem follows from the following lemmas (in particular \cref{lem:lb-main}). See \cite{hopkins2018statistical} for technical background.

We start with a lemma analyzing the moments of the matrix $(y_1,\ldots,y_{md})$ under the planted distribution.

\begin{lemma}\label[lemma]{lem:hermite-single-vec}
  For a multi-index $\alpha$ over $[d]$, let $H_\alpha$ be the $\alpha$-th (probabilists') Hermite polynomial \cite{wiki:hermite-polynomials,DBLP:books/daglib/0033652}.
  Let $x \in \R^d$ and let $\lambda \in \R$ such that $|\lambda| \|x\|^2 \leq 1$.
  Let $y \sim \cN(0,\Id + \lambda xx^\top)$.
  Then if $\alpha$ is odd, $\E H_\alpha(y) = 0$, and if $\alpha$ is even,
  \[
  \E H_\alpha(y) = \Paren{|\alpha|-1}!! \cdot \lambda^{|\alpha|/2} \cdot x^\alpha
  \]
  where $x^\alpha = \prod_{i \leq d} x_i^{\alpha_i}$ is the monomial specified by $\alpha$.
\end{lemma}
\begin{proof}
  First note that $\cN(0,\Id + \lambda xx^\top)$ is symmetric about the origin, from which the claim for odd $\alpha$ is immediate.
  We turn to even $\alpha$.
  The proof will be by induction on $|\alpha| = \sum_{i \in [d]} \alpha_i$.
  In the base case $\alpha = (0,\ldots,0)$ we have $H_\alpha(y) = 1$ so the claim is clearly true.

  We also consider separately the case of multilinear $\alpha$; that is, $\alpha$ with $\alpha_i \in \{0,1\}$.
  In this case, $H_\alpha(y) = y^\alpha = \prod_{i \in [d]} y_i^{\alpha_i}$.
  By Wick's theorem, $\E y^\alpha = \sum_{m \in M(\alpha)} \prod_{ij \in m} \E y_i y_j$, where $M(\alpha)$ is the set of all matchings on $\{i \, : \, \alpha_i = 1\}$.
    Each term in the sum is clearly equal to $\lambda^{|\alpha|} x^\alpha$, so we obtain $(|\alpha|-1)!! \cdot \lambda^{|\alpha|} \cdot x^\alpha$ (using that the number of matchings on the complete $|\alpha|$-vertex graph is $(|\alpha|-1)!!$).

  Consider the case of $\alpha$ not multilinear.
  There must be some $i \leq d$ with $\alpha_i > 1$; fix such an $i$.
  We can write $\alpha = \beta + e_i$.
  Then $H_\beta$ satisfies the following recurrence:
  \[
  H_\alpha(y) = y_i H_{\beta}(y) - \frac{\partial}{\partial y_i} H_\beta(y)\mper
  \]
  By the multivariate Stein's lemma (see e.g. \cite{liu1994siegel}),
  \[
  \E y_i H_{\beta}(y) - \frac{\partial}{\partial y_i} H_\beta(y) = \sum_{j=1}^d \lambda x_i x_j \E \frac{\partial}{\partial y_j} H_\beta(y) = \sum_{j=1}^d \lambda x_i x_j \beta_j \E H_{\beta- e_j}(y)\mcom
  \]
  where in the last equality we have used that $H_t(z)' = t H_t(z)$ for the $t$-th univariate Hermite polynomial.
  By induction, the above is
  \[
  \Paren{\sum \beta_j} \cdot \lambda^{|\alpha|/2} \cdot x^\alpha \cdot \Paren{\sum \beta_j -2}!!
  \]
  which finishes the proof, because $\Paren{\sum \beta_j}!! = \Paren{\sum \beta_j} \cdot \Paren{\sum \beta_j -2}!!$.
\end{proof}

Now we can characterize the moments of the matrix $(y_1,\ldots,y_{md})$.

\begin{lemma}\label[lemma]{lem:hermite-any}
  Let $\alpha$ be a multi-index over $[d] \times [md]$.
  Let $H_\alpha$ be the $\alpha$-th Hermite polynomial.
  Let $y = (y_1,\ldots,y_{md})$ be sampled according to the single-spike block mixture planted distribution.
  Let $\alpha^{(i)}$ be the portion of $\alpha$ corresponding to $y_{im},\ldots,y_{(i+1)m-1}$.
  Then if every $\alpha^{(i)}$ is divisible by $4$ and $\alpha_j$, by which we mean $\alpha$ restricted to $y_j$, is even for all $j$, and finally $\alpha$ restricted to each of the $d$ \emph{rows} of the matrix $(y_1,\ldots,y_{md})$ is even, then
  \[
  \E_{x,s,y} H_\alpha(y) = (\lambda/d)^{|\alpha|/2} \cdot \prod_{j \in [md]} (|\alpha_j|-1)!!
  \]
  Otherwise, $\E_{x,s,y} H_\alpha(y) = 0$.
\end{lemma}
\begin{proof}
  First of all, after conditioning on $x$ and $s$, all of $y_1,\ldots,y_{md}$ become independent.
  If $\alpha^{(i)}$ is the portion of $\alpha$ corresponding to $y_i$, then we have
  \[
  \E_{x,s} \E_y H_\alpha(y) = \E_{x,s} \prod_{i \in [m]} \E_{y_i} H_{\alpha^{(i)}}(y_i)\mper
  \]
  Now we can use \cref{lem:hermite-single-vec} to see that this is in turn
  \[
  \lambda^{|\alpha|/2} \cdot \prod_{j \in [md]} (|\alpha_j|-1)!! \cdot \E_{x,s}  \prod_{i \in [m]} s_i^{|\alpha^{(i)}|/2} x^{\alpha^{(i)}}\mper
  \]
  Here by $x^{\alpha^{(i)}}$ we mean the monomial in $x$ whose degree in $x_j$ is equal to the total degree of coordinates of the form $(j,\ell)$ in $\alpha$ for any $\ell \in [im,(i+1)m-1]$.
  If any $|\alpha_i|/2$ is odd, then the above is zero using symmetry of $s_i^{|\alpha_i|/2}$.
  Otherwise, it is equal to $\E x^\alpha$.
  This in turn is $0$ unless $\alpha$ has even degree in every row of the matrix $(y_1,\ldots,y_{md})$ and hence every $x_j$ appears in $x^\alpha$ evenly often; in this case it is equal to $d^{-|\alpha|/2}$.
\end{proof}

\begin{lemma}
\label[lemma]{lem:lb-main}
  There is a universal $c > 0$ such that for any $t \in \N$,
  if $t \lambda m^{1/2} / d^{1/4} \leq c$, then
  \[
  \sum_{|\alpha| \leq t} (\E H_\alpha(y))^2 \leq 1 + O(t \lambda m^{1/2}/d^{1/4})\mper
  \]
\end{lemma}
\begin{proof}
  Let us focus on a fixed $t$.
  Since $\E H_\alpha(y) =0$ if $\alpha$ is odd, we may assume $t$ is even.
  Let us call $\alpha$ satisfying the conditions of \cref{lem:hermite-any} \emph{super even}.
  Then we have
  \[
  \sum_{|\alpha|=t} (\E H_\alpha(y))^2 = (\lambda/d)^t \cdot \sum_{|\alpha|=t} \prod_{j \in [md]}(|\alpha_j|-1)!! \cdot 1[\alpha \text{ is super even}]\mper
  \]
  Straightforward counting of the super even $\alpha$'s shows that this quantity is at most
  \[
  \Paren{\frac{C \lambda t}{d}}^t \cdot d^{3t/4} m^{t/2}
  \]
  for some big-enough constant $C$.
\end{proof}

\section*{Acknowledgements}
We thank Tselil Schramm for helpful remarks as this manuscript was being prepared.

% BIBLIOGRAPHY

  % assumes hyperref
  \phantomsection
  \addcontentsline{toc}{section}{References}
  \bibliographystyle{amsalpha}
  \bibliography{bib/mathreview,bib/dblp,bib/custom,bib/scholar}

\appendix

% APPENDIX

\section{Linear Algebra and Probability Results}
Here we collect the statements (and proofs) of useful results from linear algebra and probability.

\begin{lemma}\label[lemma]{lem:linalg1}
Let $M$ be a (random) symmetric matrix, then $\normt{\E[M]^2} \leq \normt{\E[M^2]}$ and $\normt{\E(M-\E[M])^2} \leq 2 \normt{\E[M^2]}$.
\end{lemma}
\begin{proof}
    Note that $\E[(M-\E[M])^2] \succeq 0 \implies \E[M^2] \succeq \E[M]^2$. Since both matrices are p.s.d. it follows that $\normt{\E[M]^2} \leq \normt{\E[M^2]}$. The second claim follows since $\normt{\E(M-\E[M])^2} \leq \normt{\E[M^2]} + \normt{\E[M]^2} \leq 2 \normt{ \E[M^2]}$.
\end{proof}

\begin{lemma} \label[lemma]{lem:l4l2norm}
    Let $x \sim \cD$ be a random vector from a distribution that is L$8$-L$2$ hypercontractive -- $\E[\langle v, x \rangle^8] \leq L^2 (\E[\langle v, x \rangle^2])^4$ -- then
    \[\normt {\E[\normt{x}^2 xx^\top]} \leq L \Tr(\Sigma) \normt{\Sigma} \]
\end{lemma}
\begin{proof}
     We introduce a vector $v$ with $\normt{v} \leq 1$. Then,
\begin{align*}
    & \E[\iprod{v, \normt{x}^2 xx^\top v}] = \E[\normt{x}^2 \iprod{v, x}^2] \leq (\E[\normt{x}^8])^{1/4} (\E[\iprod{v, x}^8])^{1/4}.
\end{align*}
by Cauchy-Schwarz and the Jensen inequality.
For the first term we have $(\E[\normt{x}^8])^{1/4} \leq \sqrt{L} \Tr \Sigma$ by \cref{lem:l8l2vec}. For the second term once again using L$8$-L$2$ hypercontractivity we have,
$(\E[\iprod{v, x}^8])^{1/4} \leq \sqrt{L} \E[\iprod{v, x}]^2 \leq \sqrt{L} \normt{\Sigma}$.
\end{proof}

\begin{lemma} \label[lemma]{lem:l8l2vec}
    Let $x \sim \cD$ be a random vector with a distribution that is L$8$-L$2$ hypercontractive. Then,
    \begin{align*}
      \E[\normt{x}^8] \leq L^2 (\Tr \Sigma)^4.
    \end{align*}
\end{lemma}

\begin{proof}
A short computation using the Cauchy-Schwarz inequality and L8-L2 equivalence shows that, 
\begin{align*}
    & \E[\normt{x}^8] = \E[(\sum_{i=1}^d \iprod{x, e_i}^2)^4] = \E[\sum_{a,b,c,d} \iprod{x, e_a}^2 \iprod{x, e_b}^2 \iprod{x, e_c}^2 \iprod{x, e_d}^2] \leq \\
    & \sum_{a,b,c,d} (\E[\iprod{x, e_a}^8] \E[\iprod{x, e_b}^8] \E[\iprod{x, e_c}^8] \E[\iprod{x, e_d}^8])^{1/4} \leq \\
    & L^2 \sum_{a,b,c,d} \E[\iprod{x, e_a}^2] \E[\iprod{x, e_b}^2] \E[\iprod{x, e_c}^2] \E[\iprod{x, e_d}^2] \leq L ^2(\Tr \Sigma)^4.
\end{align*}
\end{proof}

\end{document}